%% file: main.tex
\documentclass[10.5pt,letterpaper]{article}

\usepackage{graphicx} 
\graphicspath{ {./plots/} }
\def\leftintendforlists{4mm}

\usepackage{amsmath,amssymb,amsthm}

\usepackage{color,mathtools}
\usepackage{pifont}
\newcommand{\cmark}{\ding{51}}%
\newcommand{\xmark}{\ding{55}}%
\usepackage[round]{natbib}

\usepackage{multirow}
\usepackage{enumitem}

\usepackage{multicol}

\makeatletter
\newenvironment{figurehere}
  {\def\@captype{figure}%
   \par\medskip\noindent\minipage{\linewidth}}
  {\endminipage\par\medskip}
\makeatother  
  
\usepackage{algorithm}
\usepackage{algorithmic}

\usepackage{hyperref}

\newcommand{\mVone}{$k$-SVRG-V1}
\newcommand{\mVtwo}{$k$-SVRG-V2}
\newcommand{\mpract}{$k_2$-SVRG}
\newcommand{\mHone}{-SVRG-V1}
\newcommand{\mHtwo}{-SVRG-V2}
\newcommand{\mHpract}{-SVRG}

\input{notation.tex}

\input{theorem_notation.tex}

\newcommand{\lin}[1]{{\left\langle #1 \right\rangle}}

\usepackage{fullpage}

\allowdisplaybreaks[4]

\title{$k$-SVRG: Variance Reduction for Large Scale Optimization}
\author{Anant Raj\thanks{MPI T\"{u}bingen, Germany. Email: \texttt{anant.raj@tuebingen.mpg.de}} 
\and Sebastian U. Stich\thanks{EPF Lausanne (EPFL), Switzerland. Email: \texttt{sebastian.stich@epfl.ch}}
}

\date{}

\begin{document}
\maketitle

\begin{abstract} 
Variance reduced stochastic gradient (SGD) methods converge significantly faster than the vanilla SGD counterpart. However, these methods are not very practical on large scale problems, as they either i) require frequent passes over the full data to recompute gradients---without making any progress during this time (like for SVRG), or ii)~they require additional memory that can surpass the size of the input problem (like for SAGA).

In this work, we propose $k$-SVRG that addresses these issues by making 
best use of the \emph{available} memory and minimizes the stalling phases without progress. We prove linear convergence of $k$-SVRG on strongly convex problems and convergence to stationary points on non-convex problems. Numerical experiments show the effectiveness of our method.
\end{abstract} 

\section{Introduction}
\label{sec:intro}
We study optimization algorithms for empirical risk minimization problems $f \colon \mathbb{R}^d \to \mathbb{R}$ of the form
\begin{align}
x^\star := \argmin_{x} f(x)\,,    & & \text{with} & &f(x) : = \frac{1}{n} \sum_{i = 1}^n f_i(x)\,, \label{eq:opt_prob}
\end{align}
where each $f_i \colon \R^d \to \R$ is $L$-smooth. 

Problems with this structure are omnipresent in machine learning, especially in supervised learning applications~\citep{christopher2016pattern}. 

Stochastic gradient descent (SGD)~\citep{Robbins:1951ko} is frequently used to solve optimization problems in machine learning. One drawback of SGD is that it 
does not converge at the optimal rate on many problem classes (cf.~\citep{nemirovski2009robust,lacoste2012simpler}).
\emph{Variance reduced} methods have been introduced to overcome this challenge. Among the first of these methods were SAG~\citep{LeRoux:2012nips}, SVRG~\citep{johnson2013accelerating}, SDCA~\citep{ShalevShwartz:2013wl} and SAGA~\citep{defazio2014saga}. 
The variance reduced methods can roughly be divided in two classes, namely i)~methods that achieve variance reduction by computing (non-stochastic) gradients of~$f$ from time to time, as for example done SVRG, and ii)~methods that maintain a table of previously computed stochastic gradients, such as done in SAGA.

Whilst these technologies allow the variance reduced methods to converge at a faster rate than vanilla SGD, they do not scale well to problems of very large scale. The reasons are simple: i)~not only is computing a full batch gradient $\nabla f(x)$ almost inadmissible when the number of samples $n$ is large, the optimization progress of SVRG \emph{completely stalls} while this expensive computation takes place. This is avoided in SAGA, but ii)~at the cost of $\cO(dn)$ \emph{additional} memory. When the data is sparse and the stochastic gradients $\nabla f_i(x)$ are not, the memory requirements can thus surpass the size of the dataset by orders of magnitude.

In this work we address these issues and propose a class of variance reduced methods that have i)~shorter stalling phases of only order $\cO(n/k)$ at the expense of only $\tilde{\cO}(kd)$ additional memory. Here $k$ is a parameter that can be set freely by the user. To get short stalling phases, it is advisable to set $k$ such as to fit the capacity of the fast memory of the system. We show that the new methods converge as fast as SVRG and SAGA on convex and non-convex problems, but are more practical for large $n$. As a side-product of our analysis, we also crucially refine the previous theoretical analysis of SVRG, as we will outline in Section~\ref{sec:contributions} below.

\begin{table*}[t] 
\centering
 \begin{tabular}{|l l l  l c |} 
 \hline
 method & complexity  & additional memory & \emph{in situ} $\nabla f_i$ comp. & no full pass \\
 \hline\hline
Gradient Descent & $\cO(n\kappa\log{\frac{1}{\epsilon}})$ &  $\cO(d)$ & $\cO(n)$ & \xmark  \\ 
SAGA & $\cO((n+\kappa)\log{\frac{1}{\epsilon}})$&   $\cO(dn)$ &  $\cO(1)$ &\cmark  \\
SVRG & $\cO((n+\kappa)\log{\frac{1}{\epsilon}})$  &  $\cO(d)$ & $\cO(n)$ &\xmark \\ 
 SCSG  & $\cO((\frac{\kappa}{\epsilon} \wedge   n+\kappa)\log{\frac{1}{\epsilon}})$ &  $\cO(d)$ & $<n$ &\cmark   \\
 $k$-SVRG & $\cO((n+\kappa)\log{\frac{1}{\epsilon}})$&  $\cO((dk+n) \log k)$ & $\cO(\frac{n}{k})$ &\cmark  \\[0ex] 
 \hline
 \end{tabular}
 \caption {Comparison of running times and (additional) storage requirement for different algorithms on strongly convex functions, where $\kappa = L/\mu $ denotes the condition number.  Most algorithms require \emph{in situ} computations of many $\nabla f_i(x)$ for the same $x$ without making progress. The longest such stalling phase is indicated, sometimes amounting to a full pass over the data (also indicated). 
 }
 \label{tb:table1}
\end{table*}

\subsection{SVRG, SAGA and $k$-SVRG}
\label{sec:introsvrg}

\textbf{SVRG} is an iterative algorithm, where in each each iteration only stochastic gradients, i.e. $\nabla f_i(x)$ for a random index $i \in [n]$, are computed, much like in SGD. In order to attain variance reduction a full gradient $\nabla f(x)$  is computed at a \emph{snapshot} point in every few epochs.
There are three issues with SVRG: 
i)~the computation of the full gradient requires a full pass over the dataset. No progress (towards the optimal solution) is made during this time (see illustration in Figure~\ref{fig:intro}). On large scale problems, where one pass over the data might take several hours, this can yield to wasteful use of resources; ii)~the theory requires the algorithm to restart at every snapshot point, resulting in discontinuous behaviour (see Fig.~\ref{fig:intro}) and iii)~on strongly convex problems, the snapshot point can only be updated every $\Omega(\kappa)$ iterations~(cf.~\citep{Bubeck:2014vm,johnson2013accelerating}), where $\kappa = L/\mu$ denotes the condition number (see~\eqref{def:strongc}). When the condition number is large, this means that the algorithm relies for a long time on ``outdated'' deterministic information. In practice---as suggested in the original paper by~\citet{johnson2013accelerating}---the update interval is often set to $O(n)$, without theoretical justification.

\textbf{SAGA} circumvents the stalling phases by treating every iterate as a \emph{partial snapshot} point. That is, for each index $i \in [n]$ a full dimensional vector is kept in memory and updated with the current value $\nabla f_i(x)$ if index $i$ is picked in the current iteration. Hence, intuitively, in SAGA the gradient information at partial snapshot point does have more recent information about the gradient as compared to SVRG.

A big drawback of this method is the memory consumption: unless there are specific assumptions on the structure\footnote{Cf. the discussion in~\citep[Sec. 4]{defazio2014saga}.} of $f$, this requires $O(dn)$ memory (sparsity of the data does not necessarily imply sparsity of the gradients). For large scale problems it is impossible to keep all data available in fast memory (i.e. cache or RAM) which means we can not run SAGA on large scale problems which do not have GLM structure. Although SAGA can sometimes converge faster than SVRG (but not always, cf.~\citep{defazio2014saga}), the high memory requirements prohibit it's use. 
One main advantage of this algorithm is that the convergence can be proven for every single iterate\footnote{More precisely, convergence is not directly shown on the iterates, but in terms of an auxilarly Lyapunov function.}---thus justifying stopping the algorithm at any arbitrary time---whereas for SVRG convergence can only be proven for the snapshot points. 
 
We propose \textbf{$k$-SVRG}, a class of algorithms that addresses the limitations of both, SAGA and SVRG. 
Compared to SVRG the proposed schmes have a reduced memory footprint of only $\tilde{\cO}(kd)$ and therefore allow to optimally use the available (fast) memory. Compared to SVRG the schemes avoid long stalling phases on large scale applications (see Fig.~\ref{fig:intro}). The methods do not require restarts and show smoother convergence than SVRG (see Fig.~\ref{fig:intro}).
As for SVRG, the convergence can only be guaranteed for snapshot points. However, unlike as in the original SVRG, the proposed $1$-SVRG updates the snapshot point every single epoch ($n$ iterations) and thus provides more fine grained performance guarantees than the original SVRG with $\Omega(\kappa)$ iterations between snapshot points. 

\begin{figure*}[t]
\centering
\vspace{2mm}
  \includegraphics[width=0.32\linewidth]{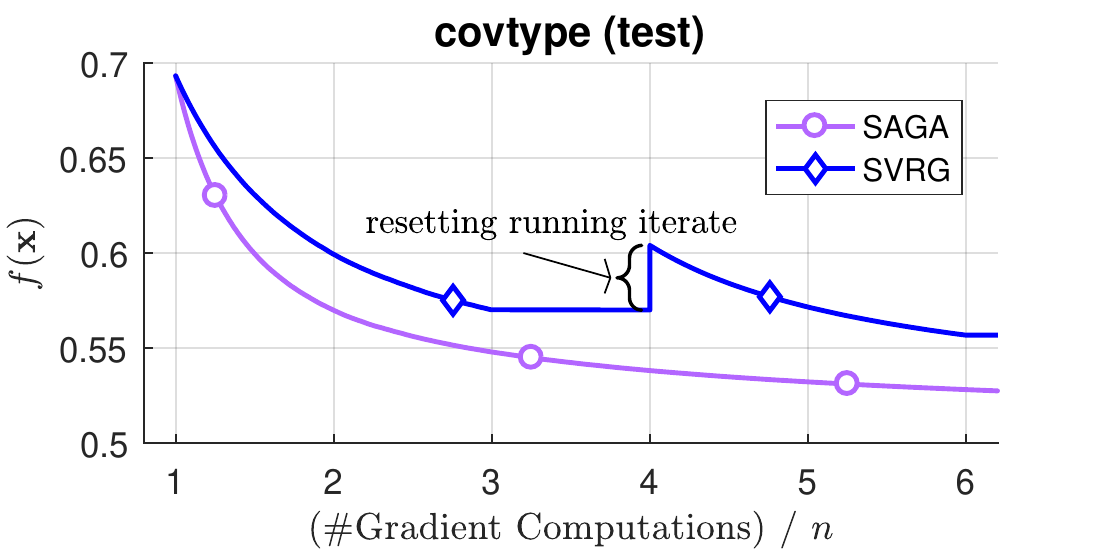}
\hfill
  \includegraphics[width=0.32\linewidth]{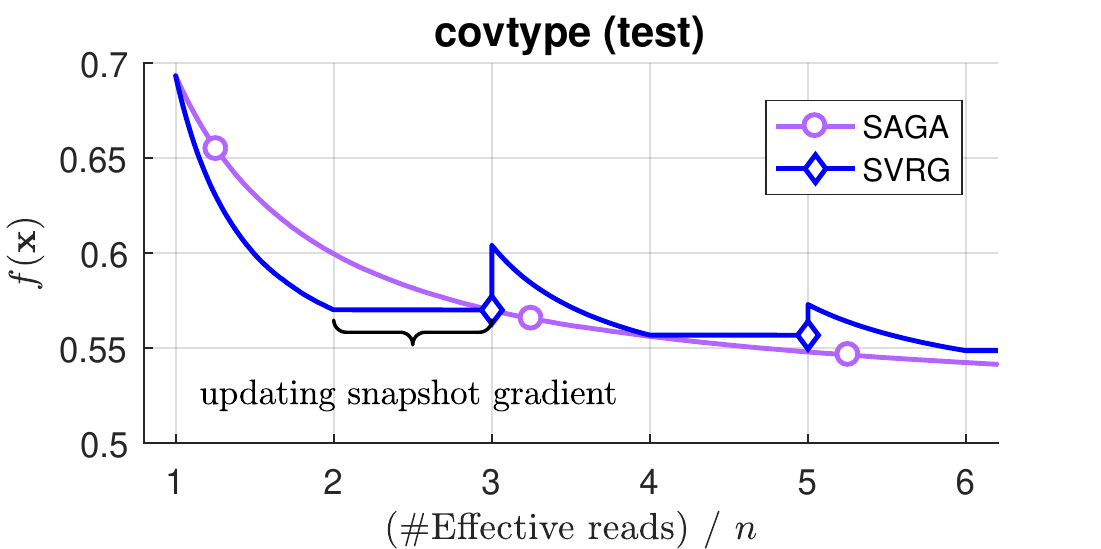}
\hfill
  \includegraphics[width=0.32\linewidth]{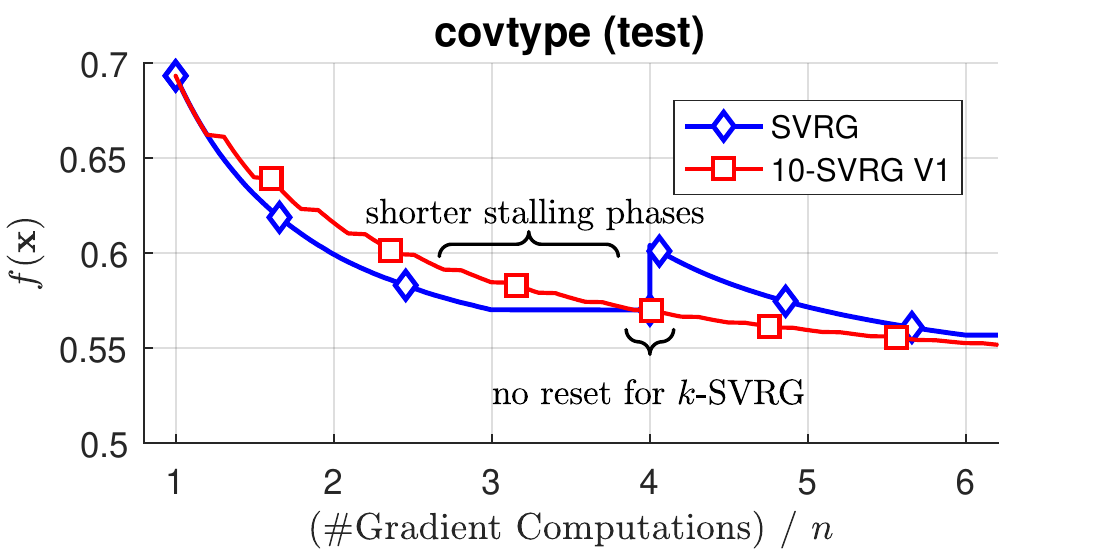}  
\caption{Convergence behavior of SAGA, SVRG and $k$-SVRG. Left \& Middle: SVRG recomputes the gradient at the snapshot point which yields to stalling for a full epoch both with respect to computation (left) and memory access (middle). SAGA requires only one stochastic gradient computation per iteration (left), but also one memory access (middle: roughly the identical performance as SVRG w.r.t. memory access). Right: $k$-SVRG does not reset the iterates at a snapshot point and equally distributes the stalling phases.}
\label{fig:intro}
\end{figure*}

\subsection{Contributions} 
\label{sec:contributions}
We present $k$-SVRG, a limited memory variance reduced optimization algorithm that combines several good properties of SVRG as well as of SAGA. We propose two variants of $k$-SVRG that require to store $\tilde{\cO}(k)$ vectors and enjoy the theoretical convergence guarantees, and one (more practical) variant that requires only $2k$ additional vectors in memory. Some key properties of our proposed approaches are:
\vspace{-2mm}
\begin{itemize}[leftmargin=4mm]
\setlength{\itemsep}{-1pt}
 \item Low memory requirements (like SVRG, unlike SAGA): We break the memory barrier of SAGA. The required additional memory can freely be chosen by the user (parameter $k$) and thus all available fast memory (but not more!) can be used by the algorithm.
 \item Avoiding long stalling phases (like SAGA, unlike SVRG): 
 This is in particular useful in large scale applications.
 \item Refinement of the SVRG analysis. To the best of our knowledge we present the first analysis that allows arbitrary sizes of inner loops, not only $\Omega(\kappa)$ as was supported by previous results.
 \item Linear convergence on strongly-convex problems (like SVRG, SAGA), cf. Table~\ref{tb:table1}.
 \item Convergence on non-convex problems (like SVRG, SAGA).
\end{itemize}

\paragraph{Outline.}
We informally introduce $k$-SVRG in Section~\ref{sec:informal} and give the full details in Section~\ref{sec:algrithms}.
All theoretical results are presented in Section~\ref{sec:theory}, the proofs can be found in Appendix~\ref{app:proof_convex} and~\ref{app:non_convex_proof}. We discuss the empirical performance in Section~\ref{sec:expts}.

\subsection{Related Work}

Variance reduction alone is not sufficient to obtain the optimal convergence rate on problem~\eqref{eq:opt_prob}. Accelerated schemes that combine the variance reduction with momentum as in Nesterov's acceleration technique~\citep{Nesterov:1983wy} achieve optimal convergence rate~\citep{allen2017katyusha,Lin:2015ue}. We do not discuss accelerated methods in this paper, however, we assume that it should be possible to accelerate the presented algorithm with the usual techniques.

There have also been significant efforts in developing stochastic variance reduced methods for non-convex problems~\citep{allen2016improved,reddi2016fast,reddi2015variance,allen2016variance,shalev2016sdca,paquette2017catalyst}. We will especially build on the technique proposed in~\citep{reddi2016fast} to derive the convergence analysis in the non-convex setting.

Recent work has also addressed the issue of making the stalling phase of SVRG shorter. In~\citep{lei2016less,lei2017non} the authors propose SCSG, a method that makes only a batch gradient update instead of a full gradient update. However, this gives a slower rate of convergence (cf. Table~\ref{tb:table1}). In another line of work, there was an effort to combine the SVRG and SAGA approach in an asynchronus optimization setting~\citep{reddi2015variance} (HSAG) to run different updates in parallel. HSAG interpolates between SAGA and SVRG ``per datapoint'' which means snapshot points corresponding to indices in a (fixed) set $S$ are updated like in SAGA, whereas all other snapshot points are updated after each epoch. This is orthogonal to our approach: we treat all datapoints ``equally''. All snapshot points are updated in the same, block-wise fashion. Also, convergence of HSAG is not guaranteed for every value of $k$. In another line of work~\citet{hofmann2015variance} studied a version of SAGA with more than one update per iteration. 

\section{$k$-SVRG: A Limited Memory Approach} 
\label{sec:informal}
In this section, we informally introduce our proposed limited memory algorithm $k$-SVRG.
For this, we will first present a unified framework that allows us to describe the algorithms SVRG and SAGA in concise notation.
 Let $x_0,x_1,\dots,x_T$ denote the iterates of the algorithm, where $x_0 \in \R^d$ is the starting point. For each component $f_i$, $i \in [n]$, of the objective function~\eqref{eq:opt_prob} we denote by $\theta_i \in \R^d$ the corresponding snapshot point. The updates of the algorithms take the form
\begin{align}
\begin{split}
 x_{t+1} &= x_t - \eta g_{i_t}(x_t)\,, \qquad\qquad \text{with}  \label{eq:gen_update} \\
 g_{i_t}(x_t) &:= \nabla f_{i_t}(x_t) - \nabla f_{i_t}(\theta_{i_t}) +\frac{1}{n} \sum_{i=1}^n \nabla f_i (\theta_i)\,, 
\end{split} 
\end{align}
where $\eta >0 $ denotes the stepsize, and $i_t \in [n]$ an index (typically selected uniformly at random from the set $[n]$).
The updates of SVRG and SAGA can both be written in this general form, as we will review now. 
\begin{description}[leftmargin=\leftintendforlists]
\setlength{\itemsep}{-1pt}
\item[SVRG] As mentioned before, SVRG maintains only one active snapshot point $x$, i.e. $\theta_i = x$ for all $i \in [n]$. Instead of storing all components $\nabla f_i(x)$ separately, it suffices to store one single snapshot point $x$ as well as $\nabla f(x)$ in memory, as all components of the gradient $\nabla f_i(x)$ can be recomputed when applying the update~\eqref{eq:gen_update}. This results in a slight increase in the computation cost, but in drastic reduction in the memory footprint.
\item[SAGA] The update of SAGA takes exactly the form~\eqref{eq:gen_update}. In general $\theta_i \neq \theta_j$ for $i \neq j$. Thus all $\theta_i$ parameters need to be kept in memory. In practice often $\nabla f_i (\theta_i)$ is stored instead, as this avoids recomputation of $\nabla f_i (\theta_i)$. 
\item[$k$-SVRG] 
As a natural interpolation between those two algorithms we propose the following: instead of maintaining just one single snapshot point or $n$ of them, just maintain \emph{a few}. Precisely, the proposed algorithm maintains a set of snapshot points $\Theta \subset \R^d$ of cardinality $\tilde{\cO}(k\log k)$, with the property $\theta_i \in \Theta$ for each $i \in [n]$. Therefore, it suffices to store only $\Theta$ in the memory, and a mapping from each index $i$ to its corresponding element in $\Theta$. This needs $\tilde{\cO}((dk+n)\log k)$ memory. Opposed to SAGA, it is not adviced to store $\nabla f_i(\theta_i)$ directly, as this would require $O(dn)$ memory.
\item[$k_2$-\textbf{SVRG}] 
We also propose a heuristic variant of $k$-SVRG that maintains at most $2k$ snapshot points. This method comes without theoretical convergence rates, however, it shows quite good performance in practice.
\end{description}

We will give a formal definition of the algorithm in the next Section~\ref{sec:algrithms}. Below we introduce some notation that will be needed later.

\subsection{Notation}

Our algorithm consists of updates of two types: updates of the iterates as in~\eqref{eq:gen_update}, performed in the \emph{inner loop} and the updates of the snapshot points at the end of the inner loops (thus constituting the \emph{outer loop}). 
We denote the iterates of the algorithm by $x_t^m$, where $t$ denotes the counter of the inner loop (consisting of $\ell$ iterations), and $m \geq 0$ the counter of the outer loop. For our algorithm (unlike in SAGA), the iterate at the end of an inner loop coincides with the first iterate of the next inner loop, $x_{\ell}^m = x_{0}^{m+1}$. Whenever we only consider the iterates $x_0^m$ we will drop the index zero for convenience.

For clarity, we will also index the snapshot points by $m$, that is we write $\theta^m_i$ for the snapshot point corresponding to the component $f_i$ in the $m^{th}$ outer loop. And consequently, $\Theta^m :=\{\theta^m_i \colon i \in [n]\}$. Thus the update~\eqref{eq:gen_update} now reads

\begin{align}
\begin{split}
x^m_{t+1}  &= x^m_t - \eta g^m_{i_t}(x_t^m)\,, \qquad\qquad \text{with} \label{eq:k-svrg_update}  \\ 
 g^m_{i_t}(x_t^m) &=  \nabla f_{i_t}(x_t^m) - \nabla f_{i_t}(\theta_{i_t}^m) + \frac{1}{n}  \sum_{i=1}^n \nabla f_i (\theta_i^m) \,. 
 \end{split} 
\end{align} 
It will be convenient to define
\begin{align}
 \alpha_i^m &:= \nabla f_i(\theta_i^m) \,, &
 \bar{\alpha}^m &:= \frac{1}{n} \sum_{i=1}^n \alpha_i^m \,.
\end{align}
\paragraph{Notation for Expectation.} $\mathbb{E}$ denotes the full expectation with respect to the joint distribution of all chosen data points. Frequently, we will only consider the updates within one outer loop, and condition on the past iterates. Let $\mathcal{I}_t^m :=\{i_0,\dots,i_{t-1}\}$ denote the set of chosen indices in the $m^{th}$ outer loop until the $t^{th}$ inner loop iteration. Then $\mathbb{E}_{t,m} = \mathbb{E}_{\mathcal{I}_t^m}$ denotes the expectation with respect to the joint distribution of all indices in $\mathcal{I}_t^m$. The algorithm $k$-SVRG-V2 samples additional $q$ indices, independent of $\mathcal{I}_\ell^m$ and we denote the expectation over those samples by $\bbE_q'$. Finally, we also denote $\bbE_{\ell ,m}\bbE_q'$ as $\bbE'_{q,m}$ and $\bbE_{\ell ,m}$ as $\bbE_m$.

\section{The Algorithm} \label{sec:algrithms}
In this section, we present $k$-SVRG in detail.
The pseudecode is given in Algorithm~\ref{algo:svrg_saga_1_2}.
$k$-SVRG consist of inner and outer loops similar to SVRG, however the size of the inner loops is much smaller. Recall that $t =0,\dots,\ell -1$ denotes the counter of the inner loop (where $\ell = \lceil n/k \rceil$), and $m \geq 0$ denotes the counter of the outer loop. Similar as in SVRG, a new snapshot point (denoted by $\tilde{x}^{m+1}$) is computed as an average of the iterates $x_t^m$. However, in our case is a weighted average 
\begin{align}
\tilde{x}^{m+1} :=  \frac{1}{S_{\ell}} \sum_{t = 0}^{\ell -1} (1-\eta\mu)^{\ell -1 -t} x_t^m \,, \label{eq:average}
\end{align}
where the normalization $S_{\ell}$ is defined in line~\ref{lst:defS}. Note that $\mu = 0$ for non-convex functions and the weighted average in~\eqref{eq:average} reduces to a uniform average. 

In Algorithm~\ref{algo:svrg_saga_1_2}, we describe two variants of $k$-SVRG. 
 These variants differ in the way how the snapshot points $\theta_i^m$ are updated at the end of each inner loop. 
 \begin{description}[leftmargin=\leftintendforlists]
\item[\textbf{V1}]  In $k$-{SVRG-V1},
we update the snapshot points as follows, before moving to the $(m+1)^{th}$ outerloop:
\begin{align}
\theta^{m+1}_i := \begin{cases}
    \theta^{m}_i , & \text{if $i \not\in \Phi^m$,}\\
   \tilde{x}^{m+1}, & \text{otherwise}.
  \end{cases} \label{eq:thetaup}
\end{align}
The set $\Phi^m$ keeps track of the selected indices in the inner loop (line~\ref{lst:phi}). Hence, we don't need to store $\abs{\Phi^m}$ copies of the the snapshot point $\tilde{x}^{m+1}$ in memory, it suffices to store one copy and the set $\Phi^m$, as mentioned in Section~\ref{sec:informal} before.
\end{description}
It is not required that the set of indices that are used to update the $\theta_i^m$ are identical with the indices used to compute $\tilde{x}^{m+1}$ in the inner loop. Moreover, also the number points does not need to be the same. The following version of $k$-SVRG makes this independence explicit.
\begin{description}[leftmargin=\leftintendforlists]
\item[\textbf{V2}] 
 In $k$-{SVRG-V2($q$)}, we sample $q$ indices without replacement from $[n]$ at the end of the $m^{th}$ outer loop, which form the set $\Phi^m$, and then update the snapshot points as before in~\eqref{eq:thetaup}.
The suggested choice of $q$ is $\cO(n/k)$, and whenever we drop the argument, we simply set $q= \ell = \lceil n/k \rceil$. 
\end{description}
%
\begin{algorithm}[t]
\begin{algorithmic}[1] 
 \STATE \textbf{goal} minimize $f(x) = \frac{1}{n} \sum_{i = 1}^n f_i(x)$
  \STATE \textbf{init} $x_0^0$, $\ell$, $\eta$, $\mu$, $\alpha_i^0 ~\forall i \in [n]$, $\bar{\alpha}^0 \gets \frac{1}{n} \sum_{i=1}^n\alpha_i^0$ 
  \STATE $S_{\ell} \gets \sum_{i=0}^{\ell -1} (1 - \eta \mu)^{i}$ \label{lst:defS} 
  \STATE \textbf{for} {$m=0\dots M-1$}
  \STATE \quad \textbf{init} $\Phi^m \gets \emptyset$
  \STATE \quad  \textbf{for} {$t=0\dots \ell -1$}
   \STATE \quad \quad pick $i_t \in [n]$ uniformly at random
   \STATE \quad \quad $\alpha^{m}_{i_t} \gets \nabla f_{i_t}(\theta^{m}_{i_t})$  \label{lst:updatealpha}
  \STATE \quad \quad $x_{t+1}^m \gets x_{t}^m -\eta \big(\nabla f_{i_t}(x_t^m) - \alpha_{i_t}^m + \bar{\alpha}^m\big)$ \label{lst:alphaneeded}
  \STATE \quad \quad $\Phi^m \gets \Phi^m \cup \{i_t\}$ \label{lst:phi}
  \STATE \quad \textbf{end for}
 \STATE \quad $\tilde{x}^{m+1} \gets \frac{1}{S_{\ell}} \sum_{t=0}^{\ell -1} (1-\eta\mu)^{\ell -1 -t} x_t^m $ \label{lst:update_tilde_x}
  \STATE \quad $x^{m+1}_0 \gets x^m_{\ell}$
  \STATE \quad \textbf{if} variant  \mVtwo$(q)$ \label{lst:v2start}
  \STATE \quad  \quad  $\Phi^m \gets \text{ sample without replacement }(q,n)$
  \STATE \quad \textbf{end if} \label{lst:V2end}
   \STATE \quad $\theta_i^{m+1} \gets \begin{cases} \tilde{x}^{m+1}, & \text{if } i \in \Phi^m \\  \theta_i^m, &\text{otherwise} \end{cases}$ \label{ln:theta}
     \STATE \quad $ \bar{\alpha}^{m+1} \gets \bar{\alpha}^{m} + \frac{1}{n} \sum_{i\in \Phi^m} \nabla f_i(\theta_i^{m+1}) - \frac{1}{n} \sum_{i\in \Phi^m} \nabla f_i(\theta_i^{m}) $ \label{lst:baralpha} 
  \STATE \textbf{end for}
  \STATE \textbf{return} $\tilde{x}_{M}$ 
\end{algorithmic}
 \caption{\mVone{} / \mVtwo$(q)$ }
 \label{algo:svrg_saga_1_2}
\end{algorithm}

\paragraph{Memory Requirement.}
To estimate the memory requirement we need to know the number of different elements in the set $\Theta$ of snapshot points. The well-studied Coupon-Collector problem~(cf.~\citep{Holst86couponcollector}) tells us that in expectation there are $\cO(n \log n)$ uniform samples needed to pick every index of the set~$[n]$ at least once. 
In Algorithm~\ref{algo:svrg_saga_1_2} precisely $\ell$ samples are picked in each iteration of the inner loop, which implies each single index in $[n]$ gets picked after $\cO(k \log k)$ outer loops. Thus there are in expectation only $\cO(k \log k)$ different different snapshot points at any time ($n \leq k \ell$). These statements do also hold with high probability at the expense of additional poly-log factors in $n$. Thus, $\tilde{\cO}((dk +n)\log k)$ memory suffices to invoke Algorithm~\ref{algo:svrg_saga_1_2}.

We can enforce a hard limit on the memory by slightly violating the random sampling assumption: instead of sampling without replacement in $k$-SVRG-V2, we just process all indices according to a random permutation, and reshuffle after each epoch (the pseudocode is given in Algorithm~\ref{algo:svrg_saga_prac} in Appendix~\ref{app:pseudo_prac}). Clearly, as we process the indices by the order given by random permutations, each index gets picked at least once every $2n$ iterations, i.e. at least once after $2n/\ell \leq 2k$ outer loops. Therefore, there are at most $2k$ distinct snapshot points at any time.
\begin{description}[leftmargin=\leftintendforlists]
\item[$k_2$-SVRG] $k_2$-SVRG deviates from $k$-SVRG-V2 on lines~\ref{lst:v2start}--\ref{lst:V2end}. Instead of sampling $q=\ell$ distinct indices in each outer loop independently, we process the indices by blocks. Concretely, every $k^{th}$ outer loop we sample a random partition $[n]= \mathcal{P}_0^m \cup \cdots \cup \mathcal{P}_{k-1}^m$, $\abs{\mathcal{P}_i} = \ell$ for $i=0,\dots,k-1$ independently at random, and then process the indices of the sets $\mathcal{P}_i$ the $(m+i)^{th}$ outer loop (to not clutter the nation we assumed here $n=k\ell$). We give the pseudocode for $k_2$-SVRG in Appendix~\ref{app:pseudo_prac}. 
\end{description}
\begin{remark}[Implementation]\label{rem:implementation}
One of the main advantages of $k$-SVRG is that no full pass over the data is required at the end of an outer loop. The update of $\tilde{x}^{m+1}$ in line~\ref{lst:update_tilde_x} can be computed on the fly with the help of an extra variable. To implement the update of the $\theta_i$'s on line~\ref{ln:theta} we use the compressed representation of the set $\Theta$ as discussed above. The update of $\bar{\alpha}^{m+1}$ in line~\ref{lst:baralpha} requires $2\ell$ gradient computations for \emph{$k$-SVRG-V2}, but only $\ell$ for \emph{$k$-SVRG-V1}, as
 \begin{align}
  \frac{1}{n} \sum_{i\in \Phi^m} \nabla f_i(\theta_i^{m})  =  \frac{1}{n} \sum_{i\in \Phi^m} \alpha_i^m\,. \label{eq:alpha_bar_update}
 \end{align}
for computed values $\alpha_i^m$ for $i \in \Phi^m$.
 \end{remark}

\section{Theoretical Analysis } \label{sec:theory}
In this section, we  provide the theoretical analysis for the proposed algorithms from the previous section. We will first discuss the convergence in the convex case in Section~\ref{subsec:convex_th} and then later will discuss the convergence in the non-convex setting in Section~\ref{subsec:nonconvex_th}. 
For both cases we will assume
that the functions  $f_i$, $i \in [n]$, are $L$-smooth.
Let us recall the definition: 
A function $f \colon \R^d \to \R$ is $L$-smooth if it is differentiable and
\begin{align}
 \| \nabla f(x) - \nabla f(y) \| \leq L \|x - y \|\,, ~~\forall x,y \in \mathbb{R}^d. \label{eq:beta_smooth}
\end{align}

\subsection{Strongly Convex Problems}  \label{subsec:convex_th}

In this subsection we additionally assume $f$ to be $\mu$-strongly convex for $\mu > 0$, i.e. we assume it holds:
\begin{align}
f(y) \geq f(x) + \lin{\nabla f(x), y-x} + \frac{\mu}{2} \norm{y-x}^2\,, ~~\forall x,y \in \mathbb{R}^d. \label{def:strongc}
\end{align}
It will also become handy to denote $f^\delta(x) := f(x) - f(x^\star)$, following the notation in~\citep{hofmann2015variance}. 

\paragraph{Lyapunov Function.} 
Similar as in~\citep{defazio2014saga} and~\citep{hofmann2015variance}, we show convergence of the algorithm by studying a suitable Lyapunov function. In fact, we are using the same family of functions as in~\citep{hofmann2015variance} where $\mathcal{L} \colon \R^n \times \R \to \R$ is defined
as follows:
\begin{align}
\mathcal{L}(x, H) := \|x - x^\star \|^2 + \gamma \sigma H\,,  \label{def:lyap}
\end{align}
with $\gamma := \frac{\eta n}{L}$ and $0 \leq \sigma \leq 1$ a constant parameter that we will set later. We will evaluate this function at tuples $(x^m, H^m)$, where $x^m=x_0^m$ are the iterates of the algorithm. In order to show convergence we therefore also need to define a sequence of parameters $H^m$ that are updated in sync with $x^m$. Clearly, if $H^m \to 0$ for $m \to \infty$, then convergence of $\mathcal{L}(x^m,H^m) \to 0$ implies $x^m \to x^\star$. We will now proceed to define a sequence $H^m$ with this property. It is important to note that these quantities do only show up in the analysis, but neither need to be be computed nor updated by the algorithm.

Similar as in~\citep{hofmann2015variance}, we will define quantities $H^m_i$ with the property $H^m_i \geq \| \alpha^m_i - \nabla f_i(x^\star) \|^2$, and thus their sum,  
$H^m := \frac{1}{n}\sum_{i=1}^n H^m_i$ is an upper bound on $\bbE \| \alpha_i^m - \nabla f_i (x^\star) \|^2$. 
Let us now proceed to precisely define $H^m_i$.
For this let $h_i^m \colon \R^d \to \R$ be defined as
\begin{align}
 h_i^m (x) := f_i(x) - f_i(x^\star) - \lin{x-x^\star,  \nabla f_i(x^\star)} \,.
\end{align}
We initialize (conceptually) $\alpha_i^0 = 0$ and $H^0_i = \norm{\nabla f_i(x^\star)}^2$ for $i \in [n]$, and then update the bounds
 $H^m_i$ in the following manner:
\begin{align}
H^{m+1}_i = \begin{cases}
    2 L h^m_i(\tilde{x}^{m+1}), & \text{if $i  \in \Phi^m$},\\
    H^m_i, & \text{otherwise}.
  \end{cases} \label{eq:updateH}
\end{align}
Here $\Phi^m$ denotes the set of indices that are used to compute $\tilde{x}^{m+1}$ in either \mVone{} or \mVtwo, see Algorithm~\ref{algo:svrg_saga_1_2}.

\paragraph{Convergence Results.}
We now show the linear convergence of \mVone{} (Theorem~\ref{thm:var_1_mu}) and \mVtwo{} (Theorem~\ref{thm:var_2_mu_gen}).
\begin{theorem} \label{thm:var_2_mu_gen}
Let $\{x^m\}_{m \geq 0}$ denote the iterates in the outer loop of \emph{\mVtwo($q$)}. If $\mu > 0$, parameter $q \geq \frac{\ell}{3}$, and step size $\eta \leq \frac{1}{3(\mu n+ 2L)}$ then
\begin{align}
\bbE_{q,m}' \mathcal{L}(x^{m+1}, H^{m+1}) \leq \big( 1 - \eta \mu \big)^{\ell} \mathcal{L}(x^{m}, H^{m})\,.   \label{eq:eq_thm_1}
\end{align}
\end{theorem}

\begin{proof}[Proof Sketch]
By applying Lemmas~\ref{lem:similar_hoff} and \ref{lem:lemma2}, we directly get the following relation:
\begin{gather}
\bbE_m \|x^{m+1} - x^\star \|^2 + \gamma \sigma \bbE_{q,m}' H^{m+1}  
\leq (1-\eta\mu)^{\ell} \|x^m -x^\star \|^2  + p_2 H^m - r_2 \bbE_m f^\delta (\tilde{x}^{m+1}) \,, 
\end{gather}
where $p_2$ and $r_2$ are constants that will be specified in the proof. From this expression 
it becomes clear that we get the statement of the theorem if we can ensure $p_2 \leq (1 -\eta\mu)^\ell $ and $r_2 \geq 0$.
These calculations will be detailed in the proof in Appendix~\ref{app:proof_convex}. 
\end{proof}

\begin{theorem}\label{thm:var_1_mu}
Let $\{x^m\}_{m \geq 0}$ denote the iterates in the outer loop of \emph{\mVone}. If $\mu > 0$, and step size $\eta \leq \frac{2\big(1 -\frac{\ell -1}{2n}\big)}{5(\mu n+ 2L)} < \frac{1}{5(\mu n+ 2L)} $ then 
\begin{align}
\bbE_m \mathcal{L}(x^{m+1}, H^{m+1}) \leq ( 1 - \eta \mu )^{\ell} \mathcal{L}(x^{m}, H^{m})\,.  \label{eq:eq_thm_2}
\end{align}
\end{theorem}

\begin{proof}
The proof of Theorem~\ref{thm:var_1_mu} is very similar to the one of Theorem~\ref{thm:var_2_mu_gen}. A detailed proof is provided in the Appendix~\ref{app:proof_convex}.
\end{proof}

\noindent Let us state a few observations:
\begin{remark}[Convergence rate]\label{rem:rate}  Both results show convergence at a linear rate. The convergence factor $(1-\eta \mu)$ is the same that appears also in the convergence rates of SVRG and SAGA. For SAGA a decrease by this factor can be show in every iteration for the corresponding Lyapunov function. Thus, after $\ell$ steps, SAGA achieves a decrease of $(1-\eta \mu)^\ell$, i.e. of the same order\footnote{Note, the decrease is not exactly identical if different stepsizes are used.} as $k$-SVRG. On the other hand, the proof for SVRG shows decrease by a constant factor after $\kappa$ iterations. The same improvement is attained by $k$-SVRG after $\min\{\lceil n/\ell\rceil,\lceil \kappa/\ell \rceil\}$ inner loops, i.e. $\min\{n,\kappa\}$ total updates. Hence, our rates do not fundamentally differ from the rates of SVRG and SAGA (in case $n \gg \kappa$ we even improve compared to the former method), but they provide an interpolation between both results.
\end{remark}

\begin{remark}[Relation to SVRG]
For $k=1$ and $q = \ell = n$, our algorithms resemble SVRG with geometric averaging. However, our proof gives the flexibility to prove convergence of SVRG with inner loop size $n$, instead of $\Omega(n+\kappa)$ as in~\cite{johnson2013accelerating}. The analysis of SVRG is further strengthened in many subtle details, for instance we don’t require $x^m = \tilde{x}^m$ as in vanilla SVRG, we have shorter stalling phases (for $k \gg 1$) and the possibility to choose $q$ and $\ell$ differently opens more possibilities for tuning.
\end{remark}

\begin{remark}[Relation to SAGA]
In SAGA, exactly one snapshot point is updated per iteration. The same number of updates are performed (on average) per iteration for the setting $q = \ell$. 
\citet{hofmann2015variance} study a variant of SAGA that performs more updates per iteration ($q \geq \ell$), but there was no proposal of choosing $q < \ell$.
\end{remark}
\begin{remark}[Dependence of the convergence rate on $q$ and $k$]
For ease of presentation we have state here the convergence results in a simplified way, omitting dependence on $k$ entirely (see also Remark~\ref{rem:rate}). 
However, some mild dependencies can be extracted from the proof. For instance, it is intuitively clear that choosing a larger $q$ in Theorem~\ref{thm:var_2_mu_gen} should yield a better rate. This is indeed true. Moreover, also setting $q < \ell/3$ smaller will still give linear convergence, but at a lower rate. 
For our application we aim to choose $q$ as small as possible (reducing computation), without sacrificing too much in the convergence rate.
\end{remark}

In the rest of this subsection, we will give some tools that are required to prove Theorems~\ref{thm:var_2_mu_gen} and~\ref{thm:var_1_mu}. The proof of both statements is given in Appendix~\ref{app:proof_convex}. Lemma~\ref{lem:similar_hoff} establishes a recurrence relation between subsequent iterates in the outer loop.

\begin{lemma} \label{lem:similar_hoff}
Let $\{x^m\}_{m \geq 0}$ denote the iterates in the outer loop of Algorithm~\ref{algo:svrg_saga_1_2}. Then it holds:
\begin{align}
 \label{eq:lemma1_eq}
\bbE_{m}  \norm{x^{m+1}_{0} - x^\star}^2 \leq  (1 - \eta\mu)^{\ell} \norm{x^{m}_{0} - x^\star}^2  - 2\eta(1-2 L \eta) S_\ell  \bbE_{m} \left[ f^\delta(\tilde{x}^{m+1})\right] + 2 \eta^2   S_{\ell} \bbE_{\{i\}} \| \alpha_i^m -\nabla  f_i(x^\star) \|^2 \,,
\end{align}
where  $\tilde{x}^{m+1} = \frac{1}{S_{\ell}}\sum_{t = 0}^{\ell -1}  (1 - \eta \mu)^{\ell -1 -t} x^m_{t} $ and $S_{\ell}  = \sum_{t = 0}^{\ell -1} (1-\eta\mu)^t$.
\end{lemma}

We further need to bound the expression  $\| \alpha_i^m - \nabla f_i (x^\star) \|^2$ that appears in the right hand side of equation~\eqref{eq:lemma1_eq}.
Recall that we have already introduced bounds $H_i^m \geq \| \alpha_i^m - \nabla f_i (x^\star) \|^2$ for this purpose.
We now follow closely the machinery that has been developed in~\citep{hofmann2015variance} in order to show how these bounds  decrease (in expectation) from one iteration to the next.
\begin{lemma}\label{lem:lemma2}
Let the sequence $\{H^m\}_{m \geq 0}$ be defined as in Section~\ref{subsec:convex_th} and updated according to equation~\eqref{eq:updateH} and let $\{\tilde{x}^m\}_{m \geq 0}$ denote the sequence of snapshot points in Algorithm~\ref{algo:svrg_saga_1_2}. Then it holds:
\begin{align}
\mathbb{E}_m H^{m+1} &=   \frac{2L Q_\ell}{n} \mathbb{E}_{m} f^\delta (\tilde{x}^{m+1})  +\Big( 1 - \frac{1}{n}\Big)^\ell  H^m\,,  & &\text{(for  \mVone)} \label{eq:rec_H_case1} \\
\mathbb{E}_{q,m}' H^{m+1} &=  \frac{2L q}{n} \bbE_{m}  f^\delta(\tilde{x}^{m+1})  + \Big( 1 - \frac{q}{n}\Big) H^m\,, & & \text{(for \mVtwo)} \label{eq:rec_H_case2}
\end{align}
where $Q_\ell = \sum_{t = 0}^{\ell -1} \Big( 1 - \frac{1}{n} \Big)^{t}$. 
\end{lemma}

\subsection{Non-convex Problems} \label{subsec:nonconvex_th}
In this section, we discuss the convergence of the proposed algorithm for non-convex problems. In order to employ Algorithm~\ref{algo:svrg_saga_1_2} on non-convex problems we use the setting $\mu=0$. We limit our analysis for only non-convex smooth functions.

Throughout the section, we assume that each $f_i$ is $L$-smooth~\eqref{eq:beta_smooth}, and provide the convergence rate of algorithm \mVtwo{} only. However, convergence of the algorithm \mVone{} for the non-convex case can be shown in the similar way as for \mVtwo.  The convergence also extends to the class of gradient dominated  functions by standard techniques (cf.~\citep{reddi2016fast,allen2016improved}).
We follow the proof technique from~\citep{reddi2016fast} to provide the theoretical justification of our approach. However, the proof is not  straight forward, due to the difficulty that is imposed by the
block wise update of the snapshot points in $k$-SVRG-V2. 

\paragraph{Lyapunov Function.} 
 For the analysis of our algorithms, we again choose a suitable Lyapunov function similar to the one chosen in~\citep{reddi2016fast}. In the following, let $M$ denote the total number of outer loops performed. For $m=0,\dots,M$ define $\mathcal{L}^m \colon \R^d \times R$ as:
\begin{align}
\mathcal{L}^m(x) :=  f(x) + \frac{c^m}{n} \sum_{i =1}^n \| x^m_0 - \theta^m_i \|^2 \,, \label{eq:lyapunov_non_convex}
\end{align}
where $\{c^m\}_{m=0}^M$ denotes a sequence of parameters that we will introduce shortly (note the superscript indices). By initializing $\theta_i^0 = x^0$ we have $\mathcal{L}^0(x^0) = f(x^0)$. If we define the sequence $\{c^m\}_{m=0}^M$  such that it holds $c^M = 0$ then $\mathcal{L}^M (x^M)= f(x^M)$. These two properties will be exploited in the proof below.

Similar to the previous section, we define quantities $H^m := \frac{1}{n}\sum_{i =1}^n H_i^m$ with $H^m_i := \| x^m_0 - \theta_i^m \|^2$. With this notation we can equivalently write $\mathcal{L}^m(x)=f(x) + c^m H^m$.
We now define the sequence $\{c^m\}_{m=0}^M$ and an auxiliary sequence $\{\Gamma^m\}_{m=1}^M$ that will be used in the proof:
\begin{align}
c^m &:= c^{m+1}\big(1 - \frac{\ell}{n} + \gamma \eta \ell + 4b_1 \eta^2 L^2 \ell^2\big) + 2b_1\eta^2L^3 \ell\,, \label{eq:defc} \\
\Gamma^m &:= \eta -  c^{m+1}\frac{\eta}{\gamma} - b_1 \eta^2 L - 2 b_1 c^{m+1}\eta^2 \ell  \,,
\end{align}
with $b_1 := (1 - 2L^2 \eta^2 \ell^2)^{-1}$ and $\gamma \geq 0$ a parameter that will be specified later. As mentioned, we will set $c^M = 0$ and~\eqref{eq:defc} provides the values of $c^m$ for $m=M-1,\dots,0$.
%
It will be convenient to denote the update in 
the $m^{th}$ outer loop and $t^{th}$ inner loop with $v_t^m$, that is $x_{t+1}^m = x^m_t - \eta v^m_t$. Then we can define a matrix $V^m$ that consists of the columns $v^m_t$ for $t=0,\dots,\ell-1$ and a matrix $\nabla F^m$ that consists of columns $\nabla f(x_t^m)$ for $t=0,\dots, \ell-1$. Here $\| \cdot\|_F$ denotes the Frobenius norm.
By the notation just defined we have 
$\|V^m \|_F^2 = \sum_{t = 0}^{\ell -1} \|v^m_t \|^2$ and by the tower property of conditional expectations 
$\bbE_m\|V^m \|_F^2 = \sum_{t = 0}^{\ell -1} \bbE_{t+1,m}\|v^m_t \|^2$. By similar reasoning 
\begin{align}
\bbE_m\|\nabla F^m \|_F^2 = \sum_{t = 0}^{\ell -1} \bbE_{t+1,m}\|\nabla f(x^m_t) \|^2 = \sum_{t = 0}^{\ell -1} \bbE_{t,m}\|\nabla f(x^m_t) \|^2\,.
\end{align}

\paragraph{Convergence Results.}
 Now we provide the main theoretical result of this subsection. Theorem~\ref{thm:main_theorem_non_convex} shows  sub-linear convergence for non-convex functions.   
\begin{theorem} \label{thm:main_theorem_non_convex}
Let $\{x^m_t\}_{t=0,m = 0}^{\ell-1,M}$ denote the iterates of \emph{\mVtwo}.
Let $\{c^m\}_{m=0}^M$ be defined as in~\eqref{eq:defc} with $c^M = 0$ and $\gamma \geq 0$ and such that $\Gamma^m > 0$ for $m = 0, \dots, M-1$. 
Then:
\begin{align}
\sum_{m=0}^{M-1} \bbE\|\nabla F^m \|^2_F \leq \frac{f(x^0_0) - f^\star}{\Gamma} \,,  \label{eq:thm_main_result1}
\end{align}
where $\Gamma := \min_{0\leq m \leq M-1} \Gamma^m$.
In particular, for parameters $\eta = \frac{1}{5L n^{2/3}}$, $\gamma = \frac{L}{n^{1/3}}$ and $\ell =\frac{3}{2} n^{1/3}$ and $n>15$ it holds:
\begin{align}
\sum_{m=0}^{M-1} \bbE\|\nabla F^m \|^2_F \leq  {15 L n^{2/3} \left( f(x_0^0) - f^\star \right)}\,. 
\end{align}
\end{theorem}

\begin{proof}[Proof Sketch]
We need to rely on some technical results that will be presented in Lemmas~\ref{lemm:non_convex_lemm1}, \ref{lemm:non_convex_lemm2} and \ref{lem:non_convex_general_result} below. Equation~\eqref{eq:thm_main_result1} can be readily be derived from Lemma~\ref{lem:non_convex_general_result} by first taking expectation and then using telescopic summation. Since $\Gamma = \min_{0\leq m \leq M-1} \Gamma^m$, we get:
\begin{align}
 \Gamma \sum_{m=0}^{M-1} \bbE \|\nabla F^m \|^2 \leq {\bbE\mathcal{L}^0(x^m_0) - \bbE \mathcal{L}^M(x^{m+1}_0)} \,. 
\end{align}
By setting $\theta_i^0 = x_0^0$ for $i=1,\dots,n$ we have $\mathcal{L}^0(x_0^0) = f(x_0^0)$ and as $c^M = 0$ clearly $\mathcal{L}^M(x^M_0) = f(x_0^M)$. We find a lower bound on $\Gamma$ as a final step in our proof. Details about all the constants are given in detail in the Appendix~\ref{app:non_convex_proof}.
\end{proof}

\begin{remark}[Upper bound on $\ell$]
It is important to note here that unlike in the convex setting, Theorem~\ref{thm:main_theorem_non_convex} does not allow to set the number of steps in the inner loop, i.e. $\ell$, arbitrarily large. That essentially means that the number of snapshot points cannot be reduced below a certain threshold in $k$-SVRG-V2 for non-convex problems. 
The limitation on $\ell$ occurs due to the fact that we cannot work with a Lyapunov function which only depends on the inner loop iteration as done in~\citep{reddi2016stochastic} and hence the expected variance keeps on adding itself to the next variance term which finally gives an extra dependence of the order $\ell^2$. But we do believe that the limitation on $\ell$ can be improved further. 
Besides that limitation on $\ell$, we  get the same convergence rate for our method as that of non-convex SVRG and non-convex SAGA. 
\end{remark}

Now we discuss the lemmas which are helpful in proving Theorem~\ref{thm:main_theorem_non_convex}. The proofs of these lemmas are deferred to Appendix~\ref{app:non_convex_proof}.
Lemma~\ref{lemm:non_convex_lemm1} establishes the recurrence relation between the second term of the Lyapunov function, $H^{m+1}$, with $H^m$. 
\begin{lemma}\label{lemm:non_convex_lemm1}
Consider the setting of Theorem~\ref{thm:main_theorem_non_convex}. Then, conditioned on the iterates obtained before the $m^{th}$ outer loop, it holds for $\gamma > 0$:
\begin{align}
\begin{split} \label{eq:lemm_non_convex_lemm1}
\bbE_{\ell ,m}' H^{m+1} \leq  \eta^2  \ell \  \bbE_{\ell ,m} \| V^m \|_F^2  + \left( 1 - \frac{\ell}{n}\right)\frac{\eta}{\gamma}\ \bbE_{\ell ,m} \|\nabla F^m \|_F^2+\left(1+\gamma \eta \ell\right)\left(1 - \frac{\ell}{n}\right) H^m \,.
\end{split}
\end{align}
\end{lemma}

This result suggests that we now should relate the variance of the stochastic gradient update with the expected true gradient and the Lyapunov function. This is done in Lemma~\ref{lemm:non_convex_lemm2}, with the help of the result from Lemma~\ref{lem:lem_non_convex_1} which is provided in Appendix~\ref{app:non_convex_proof}.

\begin{lemma}\label{lemm:non_convex_lemm2}
Consider the setting of Theorem~\ref{thm:main_theorem_non_convex}. Upon completion of the $m^{th}$ outer loop it holds:
\begin{align}
\begin{split} \label{eq:non_convex_lemm2}
(1-2L^2\eta^2\ell^2) \bbE_m \| V^m\|_F^2  
& =2 \bbE_m \|\nabla F^m\|^2 + {4L^2\ell}H^m \,.
\end{split} 
\end{align}
\end{lemma}

Finally, we can proceed to present the most important lemma of this section from which the main Theorem~\ref{thm:main_theorem_non_convex} readily follows. 
\begin{lemma}\label{lem:non_convex_general_result}
Consider the setting of Theorem~\ref{thm:main_theorem_non_convex}, that is $c^m, c^{m+1}$ and 
$\gamma > 0$ are such that $\Gamma^m > 0$. Then:
\begin{align}
 \Gamma^m \cdot \bbE_m \|\nabla F^m \|^2  \leq  \mathcal{L}^m(x^m_0) - \bbE_{\ell ,m}' \mathcal{L}^{m+1}(x^{m+1}_0) \,. \label{eq:recurrence_non_convex}
\end{align}
\end{lemma}

\section{Experiments} \label{sec:expts}
\begin{figure*}[t]
\centering
  \includegraphics[width=0.32\linewidth]{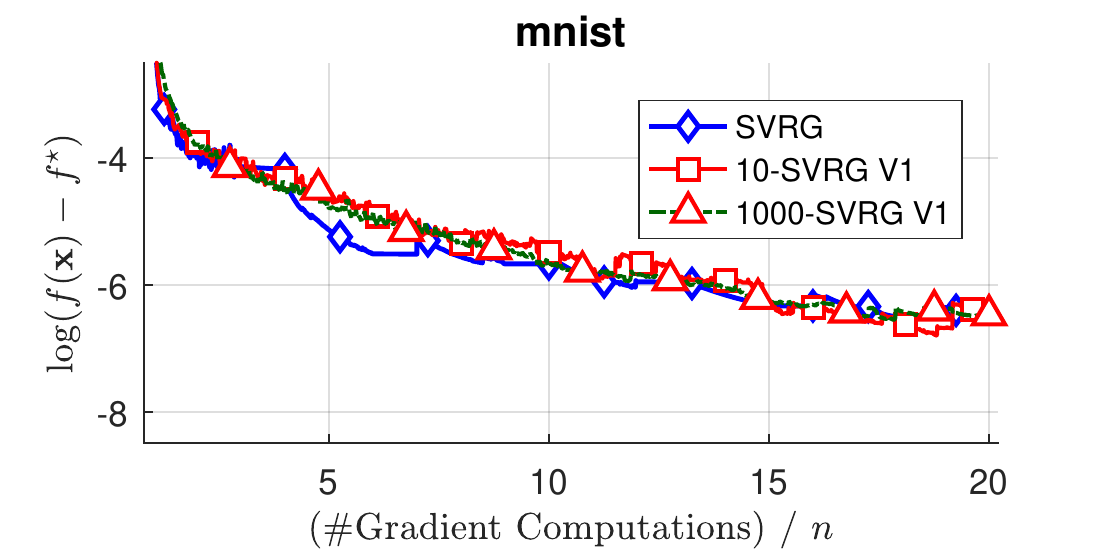}
\hfill
  \includegraphics[width=0.32\linewidth]{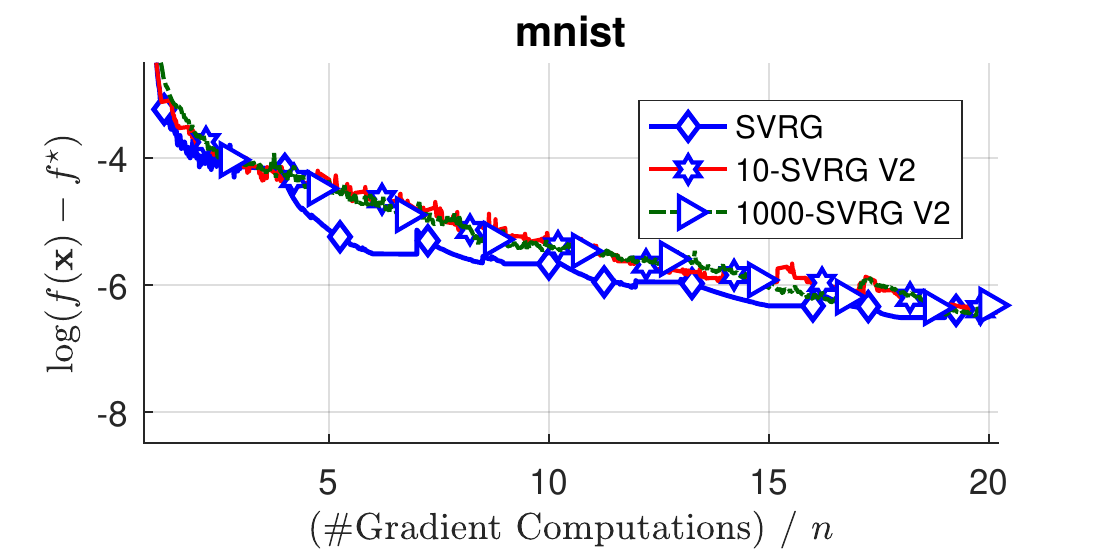}
\hfill
  \includegraphics[width=0.32\linewidth]{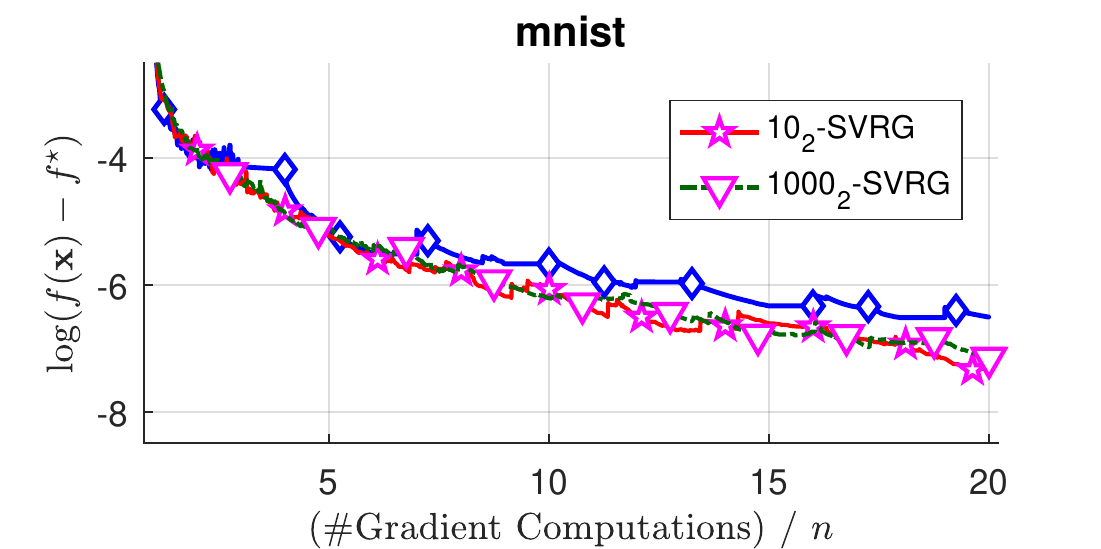}  
\caption{Residual loss on \emph{mnist} for SVRG, \mVone{} (left), \mVtwo{} (middle) and \mpract{} (right) for  $k=\{10,1000\}$.}
\label{fig:mnist_main}
\end{figure*}

\begin{figure*}[t]
\centering
  \includegraphics[width=0.32\linewidth]{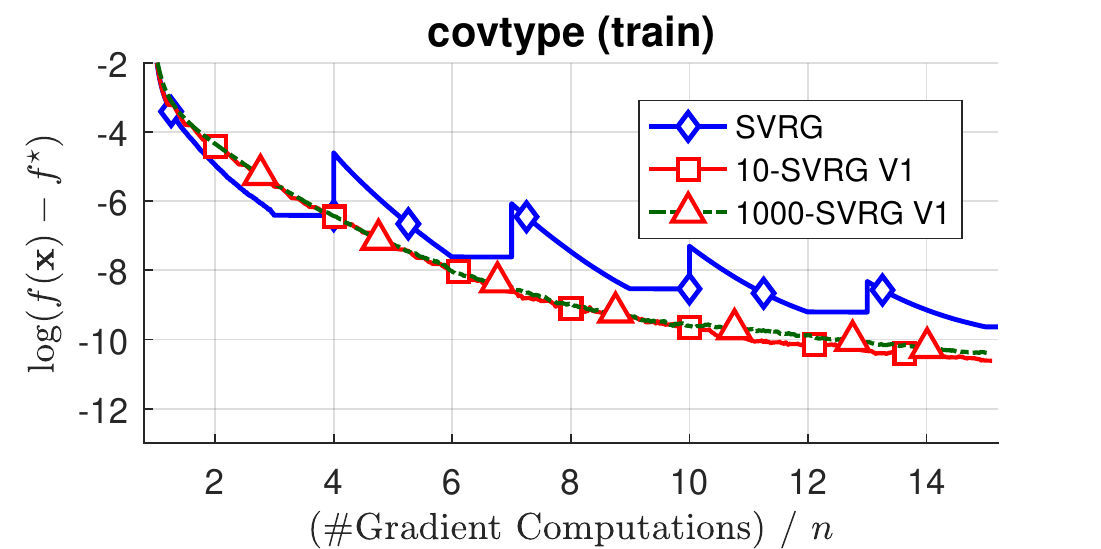}
\hfill
  \includegraphics[width=0.32\linewidth]{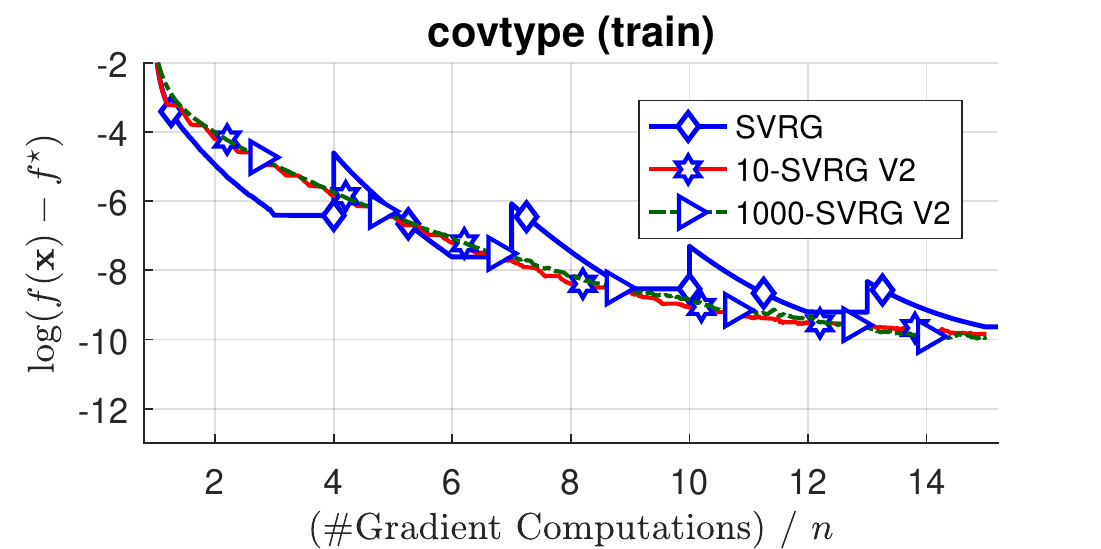}
\hfill
  \includegraphics[width=0.32\linewidth]{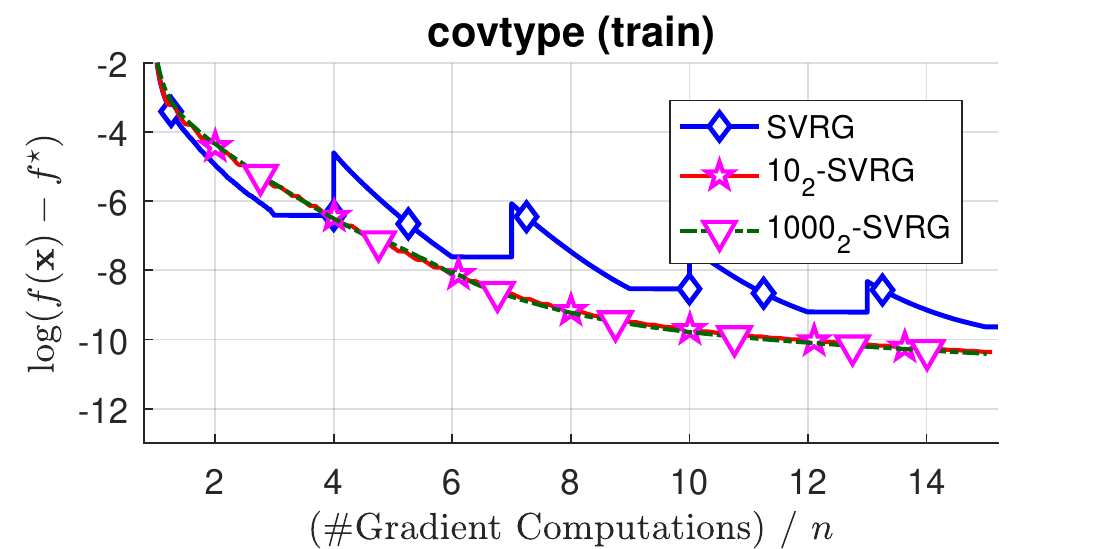}  
\caption{Residual loss on \emph{covtype (train)} for  SVRG, \mVone{} (left), \mVtwo{} (middle) and \mpract{} (right) for  $k=\{10,1000\}$.}
\label{fig:covertype_main}
\end{figure*}
To support the theoretical analysis, we present numerical results on 
$\ell_2$-regularized logistic regression problems, i.e. problems of the form
\begin{align}
f(x) = \frac{1}{n} \sum_{i=1}^n \log(1 + \exp(-b_i \lin{a_i,x}) + \frac{\lambda}{2} \norm{x}^2\,. 
\end{align}
The regularization parameter $\lambda$ is set to $1/n$, as in~\citep{nguyen2017sarah}.
We use the datasets \emph{covtype(train,test)} and \emph{MNIST(binary)}\footnote{All datasets are available at \href{http://manikvarma.org/code/LDKL/download.html}{http://manikvarma.org/code/LDKL/download.html}}.
Some statistics of the datasets are summarized in Table~\ref{tab:datasets}. 
For all experiments we use $x_0 = 0$ and perform a \emph{warm start} of the algorithms, that is we provide $\nabla f(x_0)$ as input. Several cold start procedures (where $\nabla f_i$ are injected one by one) have been suggested (cf.~\citep{defazio2014saga}) but discussing the effects of these heuristics is not the focus of this paper.

\begin{table}[t]
\begin{center}
\begin{tabular}{|c|c|c|c|} \hline
Dataset & $d$ & $n$ & $L$ \\ \hline \hline
\emph{covtype (test)} & 54 & 58\,102 & 1311\\ \hline
\emph{covtype (train)} & 54 & 522\,910 & 43\,586 \\ \hline
\emph{mnist} & 784 & 60\,000 & 38\,448\\[0px] \hline
\end{tabular}
\caption{Summary of datasets used for experiments. We use $L = \frac{1}{4} \max_i \|a_i\|^2$, where $a_i$ represents the $i^{th}$ data point.  The factor of 4 is due to the use of the logistic loss.}
\label{tab:datasets}
\end{center}
\end{table}

We conduct experiments with 
SAGA, SVRG (we fix the size of the inner loop to $n$)
and the proposed $k$-SVRG for $k=\{1,10,100,1000\}$ in all variants ($k$-SVRG-V1, $k$-SVRG-V2 and $k_2$-SVRG). For simplicity we use the parameters $l = q = \lceil n/k \rceil$ throughout. 

The running time of the algorithms is dominated by two important components: the time for computation and the time to access the data. The actual numbers depend on the hardware and problem instances.
\begin{description}[style=unboxed,leftmargin=0.3cm]
\item[Gradient Computations (\#GC).] Fig.~\ref{fig:intro} (left).
We count the number of gradient evaluations of the form $\nabla f_i(x)$. 
In SAGA, each step of the inner loop only comprises one computation, whereas for SVRG, two gradients have to be computed in the inner loop. The figure nicely depicts the stalling of SVRG after one pass over the data (when a full gradient has to be computed \emph{in situ}). 

\item[Effective Data Reads (\#ER).] Fig.~\ref{fig:intro} (middle).
We count the number of access to the data, that is when a $d$-dimensional vector needs to be fetched from memory. In the SVRG variants this is one data point in each iteration of the inner loop, and $\cO(\lceil n/k \rceil)$ data points when updating the gradients (see Remark~\ref{rem:implementation}). For SAGA  
in each iteration two values have to be fetched. For the $k$-SVRG variants the stalling phases are more equally distributed (for $k$ large). Moreover, there is no big jump in function value as the current iterate does not have to be updated (a difference to SVRG).
\end{description}

\subsection{Illustrative Experiment, Figure~\ref{fig:intro}}
\label{sec:veryfing}
For the results displayed in Figure~\ref{fig:intro} in Section~\ref{sec:introsvrg} we set the learning rate to an artificially low value $\eta = 0.1/L$ for all algorithms. This allows to emphasize the distinctive features of each method. Figure~\ref{fig:demo} in the appendix depicts additional $k$-SVRG variants for the same setting.

\subsection{Experiments on Large Datasets}
Due to the large memory constrained of SAGA, we do not run SAGA on large scale problems. Even though for every method there is a \emph{theoretical safe} stepsize~$\eta$, it is common practice
to tune the stepsize according to the dataset~(cf.~\citep{defazio2014saga,schmidt2017minimizing}). 
By extensive testing we determined the stepsizes that achieve the smallest training error after $10n$ \#ER for \emph{covtype (test)} and after $30n$ \#ER for \emph{mnist}.\footnote{We like to emphasize that the optimal stepsize crucially depend on the maximal budget. I.e. the optimal values might be different if the application demands higher or lower accuracy.} The determined optimal learning rates are summarized in Table~\ref{tab:stepsize}. 
For \emph{covtype (train)} we figured $\eta = 5.7/L$ is a reasonable setting for all algorithms.

\begin{table}[t]
\begin{center}
\scalebox{1}{
\begin{tabular}{|c|c|c|c|} \hline
Algorithm/Dataset & \emph{covtype (test)} & \emph{mnist} & \emph{covtype (train)} \\ \hline \hline
SVRG & $2.0/L$ &  $18.5/L$ & \multirow{4}{*}{$5.7/L$}\\ \cline{1-3}
\mVone & $(1.2,1.3,1.7,1.5)/L$ &  $(-,17,17,14)/L$ & \\ \cline{1-3}
\mVtwo & $(1.8,1.7,1.7,1.8)/L$ &  $(-,18,17,17.5)/L$&  \\ \cline{1-3} 
\mpract & $(1.9,1.9,1.8,1.8)/L$ &  $(-,19,18,17.5)/L$&  \\[0px] \hline
\end{tabular}
}
\caption{Determined optimal stepsizes $\eta$ for the datasets \emph{covtype (test)} and \emph{mnist} and parameters $k=(1,10,100,1000)$.}
\label{tab:stepsize}
\end{center}
\end{table}

In Figure~\ref{fig:mnist_main} we compare all algorithms on \emph{mnist}. We observe that $k_2$-SVRG performs best on \emph{mnist}, followed by the other $k$-SVRG variants which perform very similar to  SVRG. In Figure~\ref{fig:covertype_main} we compare all algorithms  on \emph{covertype (train)} and the picture is similar: $k_2$-SVRG works the best, followed by \mVone{}, then \mVtwo{} and all variants of $k$-SVRG outperform SVRG.
We observe that the parameter $k$ seems to affect the performance only by a small factor on these datasets. However, it is not easy to  predict the best possible $k$ without tuning it but larger values of $k$ do not seem to make performance worse; allowing to choose $k$ as large as supported on the system used.
Additional results are displayed in Appendix~\ref{app:add_expts}.

\section{Conclusion}
We propose $k$-SVRG, a variance reduction technique suited for large scale optimization and show convergence on convex and non-convex problems at the same theoretical rates as SAGA and SVRG. Our algorithms have a very mild memory requirement compared to SAGA and the memory can be tuned according to the available resources. By tuning the parameter $k$, one can pick the algorithm that fits best to the available system resources. I.e. one should pick a picking large $k$ for systems with fast memory, and smaller $k$ when data access is slow (in order that the additional memory still fits in RAM). This can provide a huge amount of flexibility inn distributed optimization as we can choose different $k$ on different machine.  We could also imagine that automatic tuning of $k$ as the optimization progresses, i.e. automatically adapting to the system resources, might yield the best performance in practice. However, this feature needs to be investigated further.

For future work, we plan to extend our analysis of \mpract{} using tools along the line of the recently proposed analysis of reshuffled SGD~\citep{haochen2018random}. From the computational point of view, it is also important to investigate if the gradients at the snapshot points could be replaced with inexact approximations of the gradients which are computationally cheaper to compute.

\clearpage
\bibliography{opt-ml}
\bibliographystyle{plainnat}

\newpage
\onecolumn
\appendix
\begin{center}
{\centering \LARGE Appendix }
\vspace{1cm}
\sloppy

\end{center}
\input{appendix}

\input{appendix-figs}
 \clearpage

\end{document}

%% file: notation.tex
\DeclareMathOperator*{\argmin}{arg\,min}

\def\R{{\mathbb{R}}}

\providecommand{\abs}[1]{{\left\lvert#1\right\rvert}}
\newcommand{\norm}[1]{{ \left\lVert#1\right\rVert }}

\def\bbE{{\mathbb{E}}}
\def\cO{{\mathcal{O}}}


%% file: theorem_notation.tex
\newtheorem*{rep@theorem}{\rep@title}
\newcommand{\newreptheorem}[2]{%
\newenvironment{rep#1}[1]{%
 \def\rep@title{#2 \ref{##1}}%
 \begin{rep@theorem}}%
 {\end{rep@theorem}}}
\newreptheorem{lemma}{Lemma'}
\newreptheorem{definition}{Definition'}
\newreptheorem{proposition}{Proposition'}
\newreptheorem{theorem}{Theorem'}

\newtheorem{theorem}{Theorem}
\newtheorem{lemma}[theorem]{Lemma}
\newtheorem{remark}{Remark}
\renewcommand{\text}[1]{{\textnormal{#1}}}



%% file: appendix.tex
\section{Pseudo-code for  $k_2$-SVRG} \label{app:pseudo_prac}
We provide the pseudo code $k_2$-SVRG in Algorithm~\ref{algo:svrg_saga_prac} below. For simplicity we assume here $n$ (mod) $\ell = 0$, i.e. $n=k\ell$.
\begin{algorithm}[h!]
\begin{algorithmic}[1] 
 \STATE \textbf{goal} minimize $f(x) = \frac{1}{n} \sum_{i = 1}^n f_i(x)$
  \STATE \textbf{init} $x_0^0$, $\ell$, $\eta$, $\mu$, $\alpha_i^0 ~\forall i \in [n]$ and $\bar{\alpha}^0 \gets \frac{1}{n} \sum_{i=1}^n\alpha_i^0$
  \STATE $S_{\ell} \gets \sum_{t=0}^{\ell-1} (1 - \eta \mu)^{t}$
  \STATE $k \gets \frac{n}{\ell}$
  \STATE \textbf{for} {$m=0\dots M-1$}
  \STATE \quad  ind $\gets$ randperm(n)
  \STATE \quad  \textbf{for} {$j=0\dots k-1$}
   \STATE \quad \quad \textbf{init} $\Phi^m \gets \emptyset$
  \STATE \quad \quad \textbf{for} {$t=0\dots \ell-1$}
   \STATE \quad \quad \quad pick $i_t \in [n]$ uniformly at random
   \STATE \quad \quad \quad $\alpha_{i_t}^m \gets \nabla f_{i_t}(\theta^m_{i_t})$
  \STATE \quad \quad \quad $x_{t+1}^m \gets x_{t}^m -\eta \left(\nabla f_{i_t}(x_t^m) - \alpha_{i_t}^m + \bar{\alpha}_m\right)$
  \STATE \quad \quad \quad $\Phi^m \gets \Phi^m \cup \left\{\text{ind}[j*\ell+t]\right\}$ \label{ln:phi} 
  \STATE \quad \quad \textbf{end for}
  \STATE \quad \quad  $\tilde{x}^{m+1} \gets \frac{1}{S_{\ell}} \sum_{t=0}^{\ell -1 -t} (1-\eta\mu)^{\ell -t} x_t^m $
  \STATE \quad \quad  $x^{m+1}_0 \gets x^m_{\ell}$
   \STATE \quad \quad $\theta_i^{m+1} \gets \begin{cases} \tilde{x}^{m+1}, &\text{if } i \in \Phi_m \\ \theta_i^m, &\text{otherwise}\end{cases}$
  \STATE \quad \textbf{end for}
  \STATE \quad $ \bar{\alpha}^{m+1} \gets \bar{\alpha}^{m} + \frac{1}{n} \sum_{i\in \Phi^m} \nabla f_i(\theta_i^{m+1})- \frac{1}{n} \sum_{i\in \Phi^m} \nabla f_i(\theta_i^{m}) $  \label{lst:a22} 
  \STATE \textbf{end for}
\end{algorithmic}
 \caption{$k_2$-SVRG }
 \label{algo:svrg_saga_prac}
\end{algorithm}

Like in Algorithm~\ref{algo:svrg_saga_1_2}, no full pass over the data is required at the end of the outer loop. In particular, line~\ref{lst:a22} requires only $\ell$ gradient computations, as explained in Remark~\ref{rem:implementation}. 

\section{Definitions and Notations} \label{app:mixed}
We reiterate some definitions here again before proving the main results of this paper. 

\paragraph{Function classes.} A a differentiable convex function $f \colon \R^d \to \R$ is $L$-smooth if:
\begin{align}
f(y) \leq f(x) + \langle \nabla f(x), y-x \rangle + \frac{L}{2} \| x -y\|^2 ~~\forall x,y \in \mathbb{R}^d\,, \label{eq:convex_smooth}
\end{align}
which is equivalent to
\begin{align}
\| \nabla f(x) - \nabla f(y) \| \leq L \| x -y  \|  ~~\forall x,y \in \mathbb{R}^d\,. \label{eq:non-convex-smooth}
\end{align}
A differentiable non-convex function is $L$-smooth if~\eqref{eq:non-convex-smooth} holds.
A differentiable convex function $f \colon \R^d \to \R$ is $\mu$-strongly convex if
\begin{align}
f(y) \geq f(x) + \langle \nabla f(x), y-x \rangle + \frac{\mu}{2} \| x -y\|^2 ~~\forall x,y \in \mathbb{R}^d\,.\label{eq:convex_strong}
\end{align}
Frequently, we will be denoting $f^\star := f(x^\star).$

\paragraph{Series Expansion.} 
The following observation will be useful in the analysis later. For any integer $k$ and real number $\zeta < 1$ we have
\begin{align}
(1- \zeta )^k = 1 - k \zeta + \frac{k(k-1)}{2!} \zeta^2 - \frac{k(k-1)(k-2)}{3!} \zeta^3 + \mathcal{O}( \zeta^4)\,, \label{eq:mc_series}
\end{align}
and it is easily verified that whenever $\zeta \leq \frac{1}{k} $:
\begin{align}
(1-\zeta )^k &\geq 1 - k \zeta \,,\label{eq:mc_series1} \\
(1-\zeta )^k &\leq 1 - k \zeta  + \frac{k(k-1)}{2} \zeta^2 \,. \label{eq:mc_series2}
\end{align}

\paragraph{Frequently used Inequalities.} For $a,b \in \mathbb{R}^d$ we have: 
\begin{align}
\|a+b \|_2^2 \leq (1 + \beta^{-1})\| a\|_2^2 + (1+\beta) \| b \|_2^2\,, ~~ \forall \beta > 0\,. \label{eq:norm_decom}
\end{align}
For for $\beta =1$ this simplifies to:
\begin{align}
\|a+b \|_2^2 \leq 2\| a\|_2^2 + 2 \| b \|_2^2 \,. \label{eq:norm_decom_spec}
\end{align}
Also the following inequality holds:
\begin{align}
-\left\langle a,b \right \rangle \leq \frac{\gamma}{2} \|a \|_2^2 + \frac{1}{2\gamma} \| b \|_2^2\,, ~~\forall \gamma > 0\,. \label{eq:dot_product_decom}
\end{align}

\paragraph{Notation for Non-Convex Proofs (see Section~\ref{app:non_convex_proof}).}
As defined in equation~\eqref{eq:k-svrg_update}, we have the following optimization updates:
\begin{align}
\begin{split}
x^m_{t+1}  &= x^m_t - \eta \left(  \nabla f_{i_t}(x_t^m) - \nabla f_{i_t}(\theta_{i_t}^m) + \frac{1}{n}  \sum_{i=1}^n \nabla f_i (\theta_i^m) \right) = x^m_t - \eta v^m_t  
 \end{split} \label{eq:k-svrg_update_app}
\end{align}
where $v^m_t =  \nabla f_{i_t}(x_t^m) - \nabla f_{i_t}(\theta_{i_t}^m) + \frac{1}{n}  \sum_{i=1}^n \nabla f_i (\theta_i^m)$ as defined in Section~\ref{subsec:nonconvex_th}.
Note that $\bbE_{\{i_t\}} v_t^m = \nabla f(x^m_t)$. As defined earlier in Section~\ref{subsec:nonconvex_th},  
\begin{align}
\|V^m \|_F^2 := \sum_{t=0}^{\ell -1} \|v_t^m \|^2\, & &\text{and} & &
\|\nabla F^m \|_F^2 := \sum_{t = 0}^{\ell -1} \|\nabla f(x^m_{t}) \|^2\,.
\end{align} 
Also, we will be using the following relations which immediately follow by taking expectation:
\begin{align}
\bbE_m\|V^m \|_F^2 &= \sum_{t=0}^{\ell -1}\bbE_{t+1,m} \|v_t^m \|^2 \,\\
\bbE_m\|\nabla F^m \|_F^2 &= \sum_{t = 0}^{\ell -1} \bbE_{t+1,m}\|\nabla f(x^m_{t}) \|^2 =\sum_{t = 0}^{\ell -1} \bbE_{t,m}\|\nabla f(x^m_{t}) \|^2\,.
\end{align}

\section{Proofs for Convex Problems} \label{app:proof_convex}

In this section we provide the proof of Theorems~\ref{thm:var_2_mu_gen} and~\ref{thm:var_1_mu}. We first mention an important lemma from \cite{hofmann2015variance} which relates the two consecutive iterates for SAGA. 
\begin{lemma}[\cite{hofmann2015variance}]  \label{lem:hoffman_lem}
For the iterate sequence of any algorithm that evolves solutions according to equation~\eqref{eq:gen_update}, the following holds for a single update step, in expectation over the choice of $i_t$ given $x_t$:
\begin{align*}
 \bbE_{\{i_t\}} \|& x_{t+1} - x^\star \|^2 \leq (1 - \eta \mu ) \| x_t - x^\star \|^2 + 2\eta^2 \bbE_{\{i_t\}}  \| \alpha_{i_t} - \nabla f_{i_t}(x^\star)\|^2 - 2\eta(1- 2\eta L) f^\delta(x_t)\,.
\end{align*}
\end{lemma}
The result in Lemma~\ref{lem:hoffman_lem} is the initial step towards proving a similar result to relate the iterates of two consecutive outer loops, as stated in Lemma~\ref{lem:similar_hoff}. 
\begin{proof}[\textbf{Proof of Lemma~\ref{lem:similar_hoff}}]
With Lemma~\ref{lem:hoffman_lem}, we obtain

\begin{align*}
\bbE_{\ell ,m} \norm{x^m_{\ell} - x^\star}^2 &\leq (1-\eta \mu) \bbE_{\ell -1,m}\norm{x^m_{\ell -1}-x^\star}^2 - 2\eta(1-2 L \eta) \bbE_{\ell -1,m} f^\delta (x^m_{\ell -1}) + 2 \eta^2  \bbE_{\ell,m}\| \alpha_{i_t}^m - \nabla f_{i_t}(x^\star) \|^2  \\
&= (1-\eta \mu) \bbE_{\ell -1,m}\norm{x^m_{\ell -1}-x^\star}^2 - 2\eta(1-2 L \eta) \bbE_{\ell -1,m} f^\delta (x^m_{\ell -1}) + 2 \eta^2 \bbE_{\{i\}}  \| \alpha_i^m - \nabla f_i(x^\star) \|^2 
\end{align*}

We now apply Lemma~\ref{lem:hoffman_lem} recursively to find the following:
\begin{align}
\begin{split} \label{eq:recursion}
\bbE_{\ell ,m} \norm{x^m_{\ell} - x^\star}^2 &\leq (1-\eta \mu) \bbE_{{\ell -1},m}\norm{x^m_{\ell -1}-x^\star}^2 - 2\eta(1-2 L \eta) \bbE_{{\ell -1},m} f^\delta (x^m_{\ell -1}) + 2 \eta^2 \bbE_{\{i\}}  \| \alpha_i^m - \nabla f_i(x^\star) \|^2 \\
&\leq (1-\eta \mu)^2\bbE_{{\ell -2},m} \norm{x^m_{\ell -2}-x^\star}^2 - 2\eta(1-2 L \eta) \left[ \bbE_{{\ell -1},m} f^\delta(x^m_{\ell -1}) + (1 - \eta \mu) \bbE_{{\ell -2},m} f^\delta(x^m_{\ell -2})\right] \\
& \qquad \qquad \qquad \qquad \qquad \qquad \qquad \qquad \qquad \qquad  + 2 \eta^2  \bbE_{\{i\}}  \| \alpha_i^m - \nabla f_i(x^\star) \|^2 \left[ 1 + (1 - \eta \mu) \right] \\
&\leq (1 - \eta\mu)^{\ell}  \norm{x^m_0 - x^\star}^2- 2\eta(1-2 L \eta) \sum_{t=0}^{\ell -1} (1  - \eta \mu)^{t}  \bbE_{{\ell -t},m} f^\delta(x^m_{\ell -t-1})\\
& \qquad \qquad \qquad \qquad \qquad \qquad \qquad \qquad \qquad \qquad   + 2 \eta^2 \bbE_{\{i\}}  \| \alpha_i^m - \nabla f_i(x^\star) \|^2 \cdot \sum_{t = 0}^{\ell -1} (1-\eta \mu)^t \\
& = (1 - \eta\mu)^{\ell}  \norm{x^m_0 - x^\star}^2- 2\eta(1-2 L \eta) \bbE_{{\ell -1},m}\left[\sum_{t=0}^{\ell -1} (1  - \eta \mu)^{t}    f^\delta(x^m_{\ell -t-1}) \right]  \\
&\qquad  \qquad \qquad \qquad \qquad \qquad \qquad \qquad \qquad \qquad \qquad+ 2 \eta^2  S_{\ell} \bbE_{\{i\}}  \| \alpha_i^m - \nabla f_i(x^\star) \|^2  \\
& = (1 - \eta\mu)^{\ell}  \norm{x^m_0 - x^\star}^2- 2\eta(1-2 L \eta)S_{\ell}   \bbE_{{\ell -1},m}\left[\sum_{t=0}^{\ell -1}  \frac{(1  - \eta \mu)^{t}}{S_{\ell}}    f^\delta(x^m_{\ell -t-1}) \right]  \\
&\qquad \qquad \qquad \qquad \qquad \qquad \qquad \qquad \qquad \qquad \qquad + 2 \eta^2  S_{\ell} \bbE_{\{i\}}  \| \alpha_i^m - \nabla f_i(x^\star) \|^2  \\
& = (1 - \eta\mu)^{\ell}   \norm{x^m_0 - x^\star}^2- 2\eta(1-2 L \eta)S_{\ell}   \bbE_{{\ell},m}\left[\sum_{t=0}^{\ell -1}  \frac{(1  - \eta \mu)^{t}}{S_{\ell}}    f^\delta(x^m_{\ell -t-1}) \right]  \\
&\qquad \qquad \qquad \qquad \qquad \qquad \qquad \qquad \qquad \qquad \qquad+ 2 \eta^2  S_{\ell} \bbE_{\{i\}}  \| \alpha_i^m - \nabla f_i(x^\star) \|^2  \\
& = (1 - \eta\mu)^{\ell}   \norm{x^m_0 - x^\star}^2- 2\eta(1-2 L \eta)S_{\ell}   \bbE_{{\ell},m}\left[\sum_{t=0}^{\ell -1}  \frac{(1  - \eta \mu)^{\ell -t-1}}{S_{\ell}}    f^\delta(x^m_{t}) \right] 
\end{split} \notag \\
&\qquad \qquad \qquad \qquad \qquad \qquad \qquad \qquad \qquad \qquad \qquad+ 2 \eta^2  S_{\ell} \bbE_{\{i\}}  \| \alpha_i^m - \nabla f_i(x^\star) \|^2 
\end{align}
Since $f$ is a convex function, we have by Jensen's inequality for weights $\alpha_i \geq 0$, $\sum_{i =1}^\ell  \alpha_i =1$, 
\begin{align}
f \left( \sum_{i =1}^\ell  \alpha_i x_i \right) \leq \sum_{i =1}^\ell  \alpha_i f(x_i)\,. \label{eq:jensen}
\end{align}
By definition $\tilde{x}^{m+1} = \frac{1}{S_{\ell}}\sum_{t = 0}^{\ell -1} (1 - \eta \mu)^{\ell -t-1} x^m_{t} $ and $x^m_{\ell} = x^{m+1}_0$. Hence from equations~\eqref{eq:recursion} and \eqref{eq:jensen}, we get the result:
\begin{align}
\bbE_{m}  \norm{x^{m+1}_{0} - x^\star}^2  &\leq  (1 - \eta\mu)^{\ell } \norm{x^{m}_{0} - x^\star}^2  - 2\eta(1-2 L \eta) S_\ell  \bbE_{m} f^\delta(\tilde{x}^{m+1}) + 2 \eta^2   S_{\ell} \bbE_{\{i\}}   \| \alpha_i^m -\nabla  f_i(x^\star) \|^2 \,. \qedhere
\end{align}
\end{proof}

\begin{proof}[\textbf{Proof of Lemma~\ref{lem:lemma2}}]
Recall that we defined $h^m_i(x) =  f_i(x)  - f_i(x^\star)  - \langle x - x^\star,\nabla f_i ({x}^{\star})  \rangle$. It is important to note
\begin{align}
\bbE_{\{i\}} [h^m_i(x)] = f(x) - f^\star = f^\delta(x)\,.
\end{align}
We need to derive an upper bound on $H^{m+1}$. By the update equation~\eqref{eq:updateH} we have $H_i^{m+1}=H_i$ for $i \notin \Phi^m$ and $H_i^{m+1}= 2L h_i^m(\tilde{x}^{m+1})$ for $i \in \Phi^m$. As $\tilde{x}^{m+1}$ is not known until the inner loop has terminated, we will now proof a slightly more general statement.

Define $H^{m+1}(x):= \sum_{i \notin \Phi^m} H_i^m + \sum_{i \in \Phi^m} 2L h_i^m(x)$. We will now proof that the claimed statements hold for $H^{m+1}(x)$ and we will put  $\tilde{x}^{m+1}$ in place of $x$ at the end of the proof.

\begin{description}
 \item[$k$-SVRG-V1] The process can be seen as doing sampling with replacement $\ell$ number of times.
Define the auxiliary quantities $H_i^{m,0}(x):= H_i^m$ and $H^{m,t}(x)$ by the following equation
\begin{align*}
H^{m,t}_i(x ) = \begin{cases}
    2 L h^m_i(x), & \text{if $i^{th}$ data point is chosen in $t^{th}$ inner loop iteration}.\\
    H^{m,t-1}_i(x ) , & \text{otherwise}.
  \end{cases}
\end{align*}

Now for any fixed but arbitrary $x$, we have:
 \begin{align}
 \mathbb{E}_{\ell ,m} H^{m+1} (x) &  = \mathbb{E}_{\ell ,m} H^{m,{\ell}}(x) =  \mathbb{E}_{\ell ,m} \left[ \frac{1}{n}\sum_{i=1}^n  H^{m,{\ell}}_i(x)  \right] =  \frac{1}{n}\sum_{i=1}^n \mathbb{E}_{\ell ,m} H^{m,{\ell}}_i(x)  \notag \\
 &= \frac{2L}{n} \bbE_{\ell ,m}  [h^m_i(x)] + \left( 1 - \frac{1}{n}\right) \mathbb{E}_{\ell -1,m}  H^{m,\ell -1}_i(x)  \notag \\
 &= \frac{2L}{n}  f^\delta (x)    + \left( 1 - \frac{1}{n}\right) \left[ \frac{2L}{n} f^\delta (x)  +  \left( 1 - \frac{1}{n}\right) \mathbb{E}_{\ell -1,m} H^{m,\ell -2}_i(x) \right] \notag \\
 &=  \frac{2L}{n} f^\delta (x)   \sum_{t = i}^\ell  \left( 1 - \frac{1}{n} \right)^{t -1} + \left( 1 - \frac{1}{n}\right)^\ell  H^m \notag \\
 &= \frac{2L Q_\ell}{n}  f^\delta (x)  +\left( 1 - \frac{1}{n}\right)^\ell  H^m\label{eq:rec_H_case1_proof}
 \end{align}
 where $Q_\ell = \sum_{t = 0}^{\ell -1} \left( 1 - \frac{1}{n} \right)^{t}$. Now if we replace $x$ by $\tilde{x}^{m+1}$ we get the claimed result:
\begin{align}
 \mathbb{E}_{m} H^{m+1} =  \mathbb{E}_{\ell ,m} H^{m+1} = \frac{2L Q_\ell}{n} \mathbb{E}_{m} f^\delta (\tilde{x}^{m+1})  +\left( 1 - \frac{1}{n}\right)^\ell  H^m \,.
 \end{align}
  
 \item[$k$-SVRG-V2] Finding the relation between $H^{m+1}$ and $H^m$ is much more simpler for $k$-SVRG-V2 as a set of independet $q$ points are used for the update of $H^{m+1}$.
 \begin{align}
\bbE_{\ell ,m} \mathbb{E}_q' H^{m+1} &= \bbE_{\ell ,m} \mathbb{E}_q' H^{m,\ell} =  \bbE_{\ell ,m} \mathbb{E}_q' \left[ \frac{1}{n}\sum_{=1}^n H^{m,\ell}_i \right]  =  \frac{1}{n}\sum_{i=1}^n  \bbE_{\ell ,m} \mathbb{E}_q'  H^{m,\ell}_i \notag \\
 &= \frac{2L q}{n}\bbE_{\ell ,m} \mathbb{E}_q' [h^m_i(\tilde{x}^{m+1})]  + \left( 1 - \frac{q}{n}\right)  H^m \notag \\
 & =  \frac{2L q}{n} \bbE_{\ell ,m}  f^\delta(\tilde{x}^{m+1})  + \left( 1 - \frac{q}{n}\right) H^m \,,
 \end{align}
 which is the claimed bound.

 \end{description}
\vspace{-2em}
\end{proof}

Using the results obtained in Lemmas~\ref{lem:similar_hoff} and \ref{lem:lemma2}, we are now ready to prove the main theoretical results of the Section~\ref{subsec:convex_th}.
\begin{proof}[\textbf{Proof of Theorem~\ref{thm:var_2_mu_gen}}]

We apply the results  from Lemma~\ref{lem:similar_hoff} and Lemma~\ref{lem:lemma2} for $k$-SVRG-V2 to estimate the Lyapunov function:
\begin{align}
\bbE_{q,m}' \mathcal{L}(x^{m+1}_0, H^{m+1}) &= \bbE_{m} \|x^{m+1}_0 - x^\star \|^2 + \gamma \sigma \bbE_{q,m}' H^{m+1}  \notag \\
&\leq (1 - \eta\mu)^{\ell} \norm{x^{m}_{0} - x^\star}^2  - 2\eta(1-2 L \eta) S_{\ell} \bbE_{m} f^\delta(\tilde{x}^{m+1}) + 2 \eta^2   S_{\ell}  \bbE_{\{i\}} \| \alpha_i^m - \nabla f_i (x^\star) \|^2\notag \\ & \qquad \qquad \qquad \qquad \qquad \qquad \qquad + \gamma \sigma \left[ \frac{2L q}{n}\bbE_{m} f^\delta(\tilde{x}^{m+1})+\left( 1 - \frac{q}{n}\right) H^m \right] \notag \\
&\leq (1- \eta\mu)^{\ell}  \norm{{x}^m_0 - x^\star}  +   H^m \underbrace{\left[ \gamma \sigma \left(1 -\frac{q}{n}\right)+ {2 \eta^2}S_{\ell} \right]}_{=:p_2}   \notag \\
& \qquad \qquad \qquad \qquad \qquad \qquad \qquad  - \underbrace{\left({2\eta(1-2 L \eta){S}_{\ell} } - \gamma \sigma \frac{2L q}{n}  \right)}_{=:r_2}\bbE_{m} f^\delta(\tilde{x}^{m+1}) \,. \label{eq:final_conv_case2}
\end{align}
Now in equation~\eqref{eq:final_conv_case2}, we need to find parameters such that 
\begin{align*}
p_2 &= \gamma \sigma \left(1 -\frac{q}{n}\right)+ {2 \eta^2}S_{\ell} \leq \gamma \sigma (1 - \eta\mu)^{\ell}\,, & &\text{(Condition 1)} \\
r_2 &= {2\eta(1-2 L \eta){S}_{\ell} } - \gamma \sigma \frac{2L q}{n}  \geq 0\,. & &\text{(Condition 2)} 
\end{align*}

\paragraph{Condition 1:} If we choose $\eta \leq \frac{\sigma \frac{q}{\ell}}{\mu n + 2 L}$, then (Condition 1) is satisfied. We show the calculations below: 
\begin{align}
\gamma \sigma \left(1 -\frac{q}{n}\right)+ {2 \eta^2}S_{\ell} - \gamma \sigma (1 - \eta \mu)^{\ell} &= \gamma \sigma \left( 1 -\frac{q}{n} \right) + \frac{2\eta^2 \left( 1 - (1- \eta\mu)^\ell  \right)}{\eta\mu} - \gamma \sigma (1 - \eta \mu)^{\ell} \notag \\
&=\frac{n\eta}{L} \sigma \left( 1 -\frac{q}{n} \right) + \frac{2\eta \left( 1 - (1- \eta\mu)^\ell  \right)}{\mu} - \frac{n\eta}{L} \sigma (1 - \eta \mu)^{\ell} \tag{$\gamma = \frac{n \eta }{L}$}\\
&=\eta \left( \sigma \frac{n}{L} + \frac{2}{\mu} - \sigma \frac{q}{L} \right) - \eta \left(\sigma \frac{n}{L} + \frac{2}{\mu} \right)  (1 - \eta \mu)^{\ell}
\end{align}
After division by $\eta \left(\sigma \frac{n}{L} + \frac{2}{\mu} \right)$, the right hand side reads as $1 - \frac{\sigma q}{\sigma n + 2 \frac{L}{\mu} } - (1-\eta\mu)^\ell$, and by the observation in~\eqref{eq:mc_series1} we see that the condition is satisfied for $\eta \leq \frac{\sigma \frac{q}{\ell}}{\mu n + 2 L}$.

\paragraph{Condition 2:}
If we choose $ \eta  \leq \min\left\{ \frac{1}{2L}\left(  1 -\frac{2q  \sigma  \left( n+ 2\frac{L}{\mu} \right)}{\ell \left((2n -q) + 4\frac{L}{\mu}\right)}  \right), \frac{\sigma \frac{q}{\ell}}{\mu n + 2 L}\right\} $ then the (Condition 2) is satisfied. The outline of the calculations are provided below. As $\gamma = \frac{\eta n }{L}$ the condition reads as
\begin{align}
\eta(1-2 L \eta){S}_{\ell}  -  \sigma \eta q  \geq 0\,,
\end{align}
which can be equivalently stated as $2L\eta \leq 1-\sigma \frac{q}{S_\ell}$, or
\begin{align}
2 L \eta  \leq 1 - \sigma \frac{q \eta \mu }{1 - (1- \eta \mu)^\ell } \,, \label{eq:thm2_cond2} 
\end{align}
using the definition of $S_\ell$.
Observe
\begin{align}
1 - \sigma \frac{q \eta \mu }{1 - (1- \eta \mu)^\ell }   &\leq 1  - \frac{\sigma q}{ \ell - \frac{\ell(\ell -1)}{2} \eta \mu} \tag{with~\eqref{eq:mc_series2}} \notag \\
&  \leq 1 - \frac{\sigma q}{ \ell \left( 1 - \frac{\ell}{2}\eta \mu \right) } \notag \\
& \leq 1 -\frac{\sigma q}{\ell \left( 1 - \frac{\sigma q \mu }{2(\sigma n \mu +2L)}\right)} \notag 
\intertext{%
where we used $\eta \leq \frac{\sigma \frac{q}{\ell}}{\mu n + 2 L}$. Since $\sigma \leq 1$ we have $\frac{\sigma q \mu}{ \sigma \mu n + 2L} \leq \frac{ q \mu}{  \mu n + 2L}$ and we can further estimate:
} 
&\leq 1 - \frac{\sigma q}{\ell \left( 1 - \frac{ q \mu }{2( n \mu +2L)}\right)} \notag\\
 &\leq 1 -  \frac{ 2q \sigma \left( n+ 2\frac{L}{\mu} \right)}{\ell \left((2n -q) + 4\frac{L}{\mu}\right)} \,.
\end{align} 
Hence, the condition in equation~\eqref{eq:thm2_cond2} is satisfied if
\begin{align}
 \eta  \leq \frac{1}{2L}\left(  1 -  \frac{ 2q\sigma \left( n+ 2\frac{L}{\mu} \right)}{ \ell \left((2n -q) + 4\frac{L}{\mu}\right)}  \right)
 \end{align}
as claimed. 

Finally, if we choose $q \geq \frac{\ell}{3}$ and $\sigma = \frac{\ell}{2q} \left( \frac{2L}{2L+\mu n} + \frac{2n + 2L /\mu}{ 2n-q + 4L/\mu}\right)^{-1}$  then choosing $\eta \leq \frac{1}{2(\mu n + 2L)}$ satisfies both the constraints.
\end{proof}

\begin{proof}[\textbf{Proof of Theorem~\ref{thm:var_1_mu}}]
We apply the result  from Lemma~\ref{lem:similar_hoff} and Lemma~\ref{lem:lemma2} for $k$-SVRG-V1 to estimate the Lyapunov function:
\begin{align}
\bbE_m \mathcal{L}(x^{m+1}_0, H^{m+1}) &= \bbE_m \|x^{m+1}_0 - x^\star \|^2 + \gamma \sigma \bbE_m H^{m+1}  \notag \\
&\leq (1 - \eta\mu)^{\ell} \norm{x^{m}_{0} - x^\star}^2  - 2\eta(1-2 L \eta) S_{\ell} \bbE_m f^\delta(\tilde{x}^{m+1}) + 2 \eta^2   S_{\ell} \bbE_{\{i\}}\| \alpha_i^m - \nabla f_i (x^\star) \|^2\notag \\ & \qquad \qquad \qquad \qquad \qquad \qquad +  \gamma \sigma \left[ \frac{2L Q_\ell}{n}\bbE_m f^\delta(\tilde{x}^{m+1}) +\left( 1 - \frac{1}{n}\right)^\ell  H^m \right] \notag \\
&\leq (1 - \eta\mu)^{\ell} \norm{x^{m}_{0} - x^\star}^2  +  H^m \underbrace{\left[ \gamma \sigma \left(1 -\frac{1}{n}\right)^{\ell}+ {2 \eta^2}S_{\ell} \right ]}_{=:p_1} \notag \\
& \qquad \qquad \qquad \qquad \qquad \qquad  - \underbrace{\left({2\eta(1-2 L \eta){S}_{\ell} } - \gamma \sigma \frac{2L Q_\ell}{n}  \right)}_{=:r_1}\bbE_m f^\delta(\tilde{x}^{m+1}) \label{eq:final_conv_case1}
\end{align}
Now in equation~\eqref{eq:final_conv_case1}, we need to find parameters such that
\begin{align*}
p_1 &= \gamma \sigma \left(1 -\frac{1}{n}\right)^\ell + {2 \eta^2}S_{\ell}  \leq \gamma \sigma (1 - \eta\mu)^{\ell} \,, & &\text{(Condition 1)} \\
r_1 &= 2\eta(1-2 L \eta){S}_{\ell}  - \gamma \sigma \frac{2L Q_\ell}{n}  \geq 0\,. & &\text{(Condition 2)} 
\end{align*}

\paragraph{Condition 1:} If we choose $ \eta \leq  \frac{\sigma\left( 1 - \frac{\ell -1}{2n} \right)}{ \mu n + 2 L}$ then (Condition 1) is satisfied. We show the calculations below:
\begin{align}
\gamma \sigma \left(1 -\frac{1}{n}\right)^\ell + {2 \eta^2}S_{\ell}  &= \gamma \sigma \left(1 -\frac{1}{n}\right)^\ell + {2 \eta}\frac{(1 - \eta \mu)^{\ell}}{\mu} \notag \\
&=\eta \left(  \sigma \frac{n}{L}\left( 1- \frac{1}{n} \right)^{\ell} + \frac{2}{\mu}\left( 1 - (1 - \eta \mu)^{\ell} \right)\right) \notag\\
&\leq \eta \left(  \sigma \frac{n}{L}\left( 1- \frac{\ell}{n}+ \frac{\ell (\ell-1)}{2 n^2}\right) + \frac{2}{\mu}\left( 1 - (1 - \eta \mu)^{\ell} \right)\right)
\end{align}
with~\eqref{eq:mc_series2}.
Hence, (Condition 1) is satisfied if it holds:
\begin{align}
\eta \left(  \sigma \frac{n}{L}\left( 1- \frac{\ell}{n}+ \frac{\ell (\ell-1)}{2 n^2}\right) + \frac{2}{\mu}\left( 1 - (1 - \eta \mu)^{\ell} \right)\right) \leq \sigma \frac{\eta n}{L} (1 - \eta \mu)^{\ell}\,.
\end{align}
We now finish the proof similarly as the proof of (Condition 1) in the proof of Theorem~\ref{thm:var_2_mu_gen} above. With the help of equation~\eqref{eq:mc_series1} we derive that $ \eta \leq  \frac{\sigma\left( 1 - \frac{\ell -1}{2n} \right)}{ \mu n + 2 L}$ is a sufficient condition to imply (Condition 1).

\paragraph{Condition 2:} If we choose $\eta \leq \min \left\{ \frac{1}{2L}\left( 1 - \frac{2\sigma \left( n + 2\frac{L}{\mu}\right)}{ 2n - \ell\left(1-\frac{\ell -1}{2n}\right) + 4 \frac{L}{\mu} } \right),  \frac{\sigma\left( 1 - \frac{\ell -1}{2n} \right)}{ \mu n + 2 L} \right\} $ then (Condition 2) is satisfied. 
By the definition of $Q_\ell$ and $S_\ell$, the condition can equivalently be written as
\begin{align}
 2\eta(1-2 L \eta){S}_{\ell}  - \gamma \sigma \frac{2L }{n} \sum_{t =1}^\ell  \left( 1 - \frac{1}{n} \right)^{t-1}  = 2\eta(1-2 L \eta) \frac{1 - (1- \eta\mu)^\ell }{\eta\mu }   - \gamma \sigma \frac{2L }{n} \frac{1 - (1-\frac{1}{n})^\ell }{\frac{1}{n}} \geq 0\,. \label{eq:thm1_cond2}
\end{align}
From equation~\eqref{eq:mc_series1}, we have $\left(1 - \frac{1}{n}\right)^{\ell} \geq 1 - \frac{\ell}{n}$. Hence it suffices to choose $\eta$ such that
\begin{align}
 2\eta(1-2 L \eta) \frac{1 - (1- \eta\mu)^\ell }{\eta\mu }   - \gamma \sigma \frac{2L \ell}{n} \geq 0 \,.
\end{align}
We simplify the above equation further to get:
\begin{align}
 {2\eta(1-2 L \eta) \frac{1 - (1- \eta\mu)^\ell }{\eta\mu } }  - \gamma \sigma \frac{2L \ell}{n}  
&=  2\eta \left( (1 - 2 L \eta )\frac{1 - (1- \eta\mu)^\ell }{\eta\mu }  - \sigma \ell \right) \\
& \geq 2\eta \ell \underbrace{\left((1 - 2 L \eta )\left( 1 - \frac{\ell -1}{2} \eta\mu \right) - \sigma\right)}_{=:s_1} \,,
\end{align}
with~\eqref{eq:mc_series2}. We will now derive a condition on $\eta$ such that $s_1 \geq 0$. By rearranging the terms in $s_1$ we see that it suffices to hold
\begin{align}
2L \eta &\leq 1 - \frac{\sigma}{1 - \frac{\ell -1}{2} \eta\mu} \leq 1 - \frac{\sigma}{1 - \frac{\ell}{2} \eta\mu} \leq 1 - \frac{\sigma}{1 - \frac{\ell}{2} \frac{\sigma \mu \left( 1 - \frac{\ell -1}{2n} \right)}{ \sigma \mu n + 2 L} }
\end{align}
where we used the assumption $\eta \leq \frac{\sigma\left( 1 - \frac{\ell -1}{2n} \right)}{  \mu n + 2 L}$ in the last inequality. Thus it suffices if
\begin{align}\eta \leq \frac{1}{2L}\left( 1 - \frac{2\sigma \left( n + 2\frac{L}{\mu}\right)}{ 2n - \ell\left(1-\frac{\ell -1}{2n}\right) + 4 \frac{L}{\mu} } \right)\,.
\end{align}
Finally, we see that if we choose $\eta \leq \frac{2\left(1 -\frac{\ell -1}{2n}\right)}{5(\mu n+ 2L)}$ and $\sigma = \left(2\frac{L\left(1 - \frac{\ell -1}{2n}\right)}{L+\mu n} + \frac{n + 2\frac{L}{\mu}}{ 2n-\ell (1- \frac{\ell -1}{2n}) + 4\frac{L}{\mu}}\right)^{-1} $  (which is of the same order as the $\sigma$ in the theorem~\ref{thm:var_2_mu_gen} upto a constant factor) then (Condition 1) and (Condition 2) both hold simultaneously. 
\end{proof}

\section{Proofs for Non-Convex Problems}\label{app:non_convex_proof}
 
In this section we derive the proof of Theorem~\ref{thm:main_theorem_non_convex}. First of all, we mention a result from~\cite{reddi2016fast} which is not directly applicable to our case as the setting is different, but which served as an inspiration for the proof.
\begin{lemma}[\cite{reddi2016fast}] \label{eq:result_suvrit_paper.}
Consider the SAGA updates for non-convex optimization problem where each $f_i$ is $L$-smooth and $v_t = \nabla f_{i_t}(x_t) - \nabla f_{i_t}(\theta_{i_t}) +\frac{1}{n} \sum_{i=1}^n \nabla f_i (\theta_i)$   in equation~\eqref{eq:gen_update} then:
\begin{align}
\bbE \| v_t\|^2 \leq 2 \bbE \|\nabla f(x_t) \|^2 + \frac{2L^2}{n} \sum_{i = 1}^n \bbE  \|  x_t - \theta_i \|^2\,.
\end{align}
\end{lemma}
We will now derive a similar statement that holds for our proposed algorithm.
\begin{lemma}\label{lem:lem_non_convex_1}
Consider the setting of Theorem~\ref{thm:main_theorem_non_convex}. Then it holds:
\begin{align}
\bbE_{t+1,m} \|v^m_t \|^2 \leq  2 \bbE_{t,m} \big\| \nabla f(x^m_t) \big \|^2 + 4L^2 \eta^2 t  \sum_{j =0}^{t-1}\bbE_{j+1,m} \| v_j^m \|^2 + \frac{4L^2}{n} \sum_{i = 1}^n \| x^m_0 - \theta^m_i \|^2 \,.
\end{align}
\end{lemma}

\begin{proof}
We use the following notation, $\xi_t^m  := \left(\nabla f_{i_t}(x^m_t) - \nabla f_{i_t}(\theta_t^m) \right)$. 
Now, 
\begin{align}
\bbE_{t+1,m} \|v_t^m \|^2 &= \bbE_{t+1,m} \big\| \xi_t^m + \frac{1}{n} \sum_{i =1}^{n} \nabla f(\theta_i^m) \big\|^2 \notag \\
&= \bbE _{t+1,m}\big\| \xi_t^m + \frac{1}{n} \sum_{i =1}^{n} \nabla f(\theta_i^m) - \nabla f(x_t^m) + \nabla f(x_t^m)  \big \|^2 \notag \\
&\stackrel{\eqref{eq:norm_decom_spec}}{\leq}  2 \bbE_{t+1,m} \big\| \nabla f(x^m_t) \big \|^2 + 2 \bbE_{t+1,m}  \big\| \xi_t^m - \bbE_{\{t\}} \xi_t^m \big\|^2\notag \\
&\leq  2 \bbE_{t+1,m} \big\| \nabla f(x^m_t) \big \|^2 + 2 \bbE_{t,m} \bbE_{\{t\}} \big\| \xi_t^m - \bbE_{\{t\}}\xi_t^m \big\|^2 \notag \\
&\leq 2 \bbE_{t+1,m} \big\| \nabla f(x^m_t) \big \|^2 + 2 \bbE_{t,m}  \bbE_{\{t\}}  \big\| \xi_t^m \big \|^2\notag \\
&\leq 2 \bbE_{t+1,m} \big\| \nabla f(x^m_t) \big \|^2 + \frac{2}{n}\sum_{i =1}^n \bbE_{t,m} \| \nabla f_i(x^m_t) -  \nabla f_i(\theta^m_i)\|^2 \notag \\
&= 2 \bbE_{t,m} \big\| \nabla f(x^m_t) \big \|^2 + \frac{2}{n}\sum_{i =1}^n \bbE_{t,m}\| \nabla f_i(x^m_t) - \nabla f_i(x^m_0) + \nabla f_i(x^m_0) -  \nabla f_i(\theta^m_i)\|^2 \notag \\ 
&\stackrel{\eqref{eq:norm_decom_spec}}{\leq} 2 \bbE _{t,m}\big\| \nabla f(x^m_t) \big \|^2 + \frac{4}{n}\sum_{i =1}^n \bbE_{t,m}\| \nabla f_i(x^m_t) - \nabla f_i(x^m_0) \|^2 + \frac{4}{n} \sum_{i = 1}^n \|  \nabla f_i(x^m_0) -  \nabla f_i(\theta^m_i) \|^2\notag \\ 
&\stackrel{\eqref{eq:non-convex-smooth}}{\leq} 2 \bbE_{t,m} \big\| \nabla f(x^m_t) \big \|^2 + 4L^2 \bbE_{t,m}\| x^m_t - x^m_0 \|^2 + \frac{4}{n} \sum_{i = 1}^n \|  \nabla f_i(x^m_0) -  \nabla f_i(\theta^m_i) \|^2\notag \\
&= 2 \bbE_{t,m} \big\| \nabla f(x^m_t) \big \|^2 + 4L^2 \eta^2\bbE_{t,m}\big \| \sum_{j=0}^{t-1} v_j^m \big\|^2 + \frac{4}{n} \sum_{i = 1}^n \|  \nabla f_i(x^m_0) -  \nabla f_i(\theta^m_i) \|^2\notag \\
&\leq  2 \bbE_{t,m} \big\| \nabla f(x^m_t) \big \|^2 + 4L^2 \eta^2 t  \sum_{j =0}^{t-1}\bbE_{t,m} \| v_j^m \|^2 + \frac{4}{n} \sum_{i = 1}^n \|  \nabla f_i(x^m_0) -  \nabla f_i(\theta^m_i) \|^2\notag \\
&\leq  2 \bbE_{t,m} \big\| \nabla f(x^m_t) \big \|^2 + 4L^2 \eta^2 t  \sum_{j =0}^{t-1}\bbE_{j+1,m} \| v_j^m \|^2 + \frac{4L^2}{n} \sum_{i = 1}^n \| x^m_0 - \theta^m_i \|^2 \,.
\end{align}
Hence, finally we have 
\begin{align*}
\bbE_{t+1,m} \|v_t^m \|^2  \leq  2 \bbE_{t,m} \big\| \nabla f(x^m_t) \big \|^2 + 4L^2 \eta^2 t  \sum_{j =0}^{t-1}\bbE_{j+1,m} \| v_j^m \|^2 + \frac{4L^2}{n} \sum_{i = 1}^n \| x^m_0 - \theta^m_i \|^2 \,. &\qedhere
\end{align*}
\end{proof}

\begin{lemma}\label{lem:non_convex_2}
Consider the iterates $\{x^m_t\}$ of Algorithm~\ref{algo:svrg_saga_1_2} and the new snapshot point at the end of the $m^{th}$ outer loop, $\tilde{x}^{m+1} = \frac{1}{\ell} \sum_{t = 0}^{\ell -1}x_t^m$.  Then the  following relation holds:
\begin{align}
 \bbE_{m} \|x^{m+1} - \tilde{x}^{m+1} \|^2 \leq \frac{\eta^2 (\ell +1)(2\ell +1)}{6\ell} \sum_{t=0}^{\ell -1} \bbE_{t,m} \|v_t^m \|^2 \leq \eta^2 \ell \bbE \sum_{t=0}^{\ell -1} \bbE_{t,m} \|v_t^m \|^2 =\eta^2 \ell \bbE_{\ell ,m} \|V^m \|_F^2\,. \label{eq:lem_non_convex_2}
 \end{align}
\end{lemma}
\begin{proof}
\begin{align}
\bbE_{m} \|x^{m+1} - \tilde{x}^{m+1} \|^2 &= \bbE_{\ell ,m} \|x^{m+1} - \tilde{x}^{m+1} \|^2 = \bbE_{\ell ,m} \|x^{m}_{\ell} -  \tilde{x}^{m+1}  \|^2 \notag \\
& =  \bbE_{\ell ,m} \Big \|x^{m}_\ell - \frac{1}{\ell} \sum_{t = 0}^{\ell -1}x_t^m \Big \|^2 = \frac{1}{{\ell}^2} \bbE_{\ell ,m} \Big \| \sum_{t = 0}^{\ell -1} \left(x^{m}_\ell - x_t^m\right) \Big \|^2 \notag \\ 
& =  \frac{1}{{\ell}^2} \bbE_{\ell ,m}  \Big \| -\eta  \sum_{t = 0}^{\ell -1} (i +1)v_t^m \Big \|^2 \notag \\
& = \frac{\eta^2}{{\ell}^2 } \bbE_{\ell ,m} \Big \|   \sum_{t = 0}^{\ell -1} (i +1)v_t^m \Big \|^2 \,. \label{eq:non_convex_lem2_eq_1}
\end{align}
Applying Cauchy-Schwarz in~\eqref{eq:non_convex_lem2_eq_1} gives, 
\begin{align}
\bbE_{\ell ,m}  \Big \|   \sum_{t = 0}^{\ell -1} (i +1)v_t^m \Big \|^2 \leq \frac{\ell(\ell +1)(2\ell +1)}{6} \sum_{t=0}^{\ell -1} \bbE_{t+1,m}\|v_t^m \|^2\,,
\end{align}
from which the final expression follows:
\begin{align*}
\bbE_{m} \|x^{m+1} - \tilde{x}^{m+1} \|^2 \leq \frac{\eta^2 (\ell +1)(2\ell +1)}{6 \ell} \sum_{t=0}^{\ell -1} \bbE_{t+1,m}\|v_t^m \|^2 \leq \eta^2 \ell \sum_{t=0}^{\ell -1} \bbE_{t+1,m}\|v_t^m \|^2 =\eta^2 \ell \bbE_{m} \|V^m \|_F^2\,. &\qedhere
\end{align*}
\end{proof}

\begin{proof}[\textbf{Proof of Lemma~\ref{lemm:non_convex_lemm1}}]
We take the expectation of the Lyapunov function:
\begin{align}
\bbE_{\ell ,m}' \mathcal{L}^{m+1}(x_0^{m+1}) 
& =\bbE_{\ell ,m}f(x^{m+1}_0) + \frac{c_{m+1} }{n} \sum_{i=1}^n \bbE_{\ell ,m}' \| x_0^{m+1} - \theta_i^{m+1}\|^2\,.
\end{align}
Note that we here only analyze \mVtwo{} for which the samples to update the snapshot point are independent of the samples used to generate the sequence $x_t^m$. Also recall that $q = \ell$.

First we consider the second part of the Lyapunov function which is $ \frac{c_{m+1}}{n} \sum_{i =1}^n \bbE_{\ell ,m}' \| x^{m+1}_0 - \theta^{m+1}_i \|^2$ and find its recurrence relation with $\frac{c_{m}}{n} \sum_{i =1}^n \| x^{m}_0 - \theta^{m}_i \|^2$.  
\begin{align}
\bbE_{\ell ,m}'H^{m+1} &=\sum_{i =1}^n \bbE_{\ell ,m}'\| x^{m+1}_0 - \theta^{m+1}_i \|^2 =  \sum_{i =1}^n \left( \frac{\ell}{n}  \bbE_{\ell ,m} \|x^{m+1}_0 - \tilde{x}^{m+1} \|^2 + \frac{n-\ell }{n} \bbE_{\ell ,m} \|x^{m+1}_0 - \theta_i^m  \|^2 \right) \notag \\
&= \ell \bbE_{\ell ,m}\|x^{m+1}_0 - \tilde{x}^{m+1} \|^2 + \left( 1 - \frac{\ell}{n} \right)\sum_{i =1}^n  \bbE_{\ell ,m} \|x^{m+1}_0 - \theta_i^m  \|^2\,. \label{eq:lemm2_non_convex_p_1}
\end{align}
From Lemma~\ref{lem:non_convex_2}, we know that:
\begin{align}
 \bbE_{\ell ,m} \|x^{m+1} - \tilde{x}^{m+1} \|^2 \leq \frac{\eta^2 (\ell +1)(2\ell +1)}{6\ell} \sum_{t=0}^{\ell -1} \bbE_{t,m}\|v_t^m \|^2 \leq \eta^2 \ell \bbE_{m}\|V^m \|_F^2\,.
\end{align}
We consider now the second term in~\eqref{eq:lemm2_non_convex_p_1}, keeping in mind that $x^{m+1}_0 = x^m_{\ell}$:
\begin{align}
\bbE_{\ell ,m}\|x^{m+1}_0 - \theta_i^m  \|^2 |  & = \bbE_{\ell ,m} \|x^{m+1}_0 -x^{m}_0 + x^{m}_0 -\theta_i^m  \|^2   \notag \\
&= \bbE_{\ell ,m} \left[ \|x^{m+1}_0 -x^{m}_0\|^2 + \|x^{m}_0 -\theta_i^m  \|^2  - 2 \big\langle x^m_0 - x^{m+1}_0, x^{m}_0 -\theta_i^m \big\rangle  \right] \notag \\
&= \bbE_{\ell ,m} \left[ \|x^{m+1}_0 -x^{m}_0\|^2 + \|x^{m}_0 -\theta_i^m  \|^2  - 2 \sum_{t = 0}^{\ell -1} \big\langle x^m_t - x^{m}_{t+1}, x^{m}_0 -\theta_i^m \big\rangle \right] \notag \\
&= \bbE_{\ell ,m} \left[ \|x^{m+1}_0 -x^{m}_0\|^2 \right] + \|x^{m}_0 -\theta_i^m  \|^2 - 2 \bbE_{\ell ,m} \left[ \Big\langle \sum_{t = 0}^{\ell -1} (x^m_t- x^{m}_{t+1}), x^{m}_0 -\theta_i^m \Big\rangle  \right] \notag \\
&= \bbE_{\ell ,m} \left[ \|x^{m+1}_0 -x^{m}_0\|^2 \right] + \|x^{m}_0 -\theta_i^m  \|^2 - 2 \left[ \Big\langle \sum_{t = 0}^{\ell -1}  \bbE_{\ell ,m}[x^m_t - x^{m}_{t+1}], x^{m}_0 -\theta_i^m \Big\rangle  \right] \notag \\
&= \bbE_{\ell ,m} \left[ \|x^{m+1}_0 -x^{m}_0\|^2 \right] + \|x^{m}_0 -\theta_i^m  \|^2 - 2 \left[ \Big\langle \sum_{t = 0}^{\ell -1}  \bbE_{t+1,m}[x^m_t - x^{m}_{t+1}], x^{m}_0 -\theta_i^m \Big\rangle  \right] \notag \\
&= \bbE_{\ell ,m} \left[ \|x^{m+1}_0 -x^{m}_0\|^2 \right] + \|x^{m}_0 -\theta_i^m  \|^2 - 2 \left[ \Big\langle \sum_{t = 0}^{\ell -1}  \bbE_{t,m}\bbE_{\{t\}}[x^m_t - x^{m}_{t+1}], x^{m}_0 -\theta_i^m \Big\rangle  \right] \notag \\
&= \bbE_{\ell ,m} \left[ \|x^{m+1}_0 -x^{m}_0\|^2 \right] + \|x^{m}_0 -\theta_i^m  \|^2 - 2\eta \left[  \Big\langle   \sum_{t = 0}^{\ell -1} \bbE_{t,m }\nabla f(x^m_t), x^{m}_0 -\theta_i^m \Big\rangle \right] \notag \\
&\stackrel{\eqref{eq:dot_product_decom}}{=}\bbE_{\ell ,m} \left[ \|x^{m+1}_0 -x^{m}_0\|^2 \right] + \|x^{m}_0 -\theta_i^m  \|^2  \notag \\
& \qquad \qquad \qquad \qquad \quad   + 2\eta\left[    \frac{1}{2\gamma}  \sum_{t = 0}^{\ell -1} \| \bbE_{t,m }\nabla f(x^m_t) \|^2 + \frac{\gamma \ell }{2} \|x^{m}_0 -\theta_i^m\|^2  \right] \tag{$\gamma > 0$} \\ 
&\leq \bbE_{\ell ,m} \left[ \|x^{m+1}_0 -x^{m}_0\|^2 \right] + \|x^{m}_0 -\theta_i^m  \|^2 \notag  \\
& \qquad \qquad \qquad \qquad \qquad \qquad  + 2\eta \left[    \frac{1}{2\gamma}  \sum_{t = 0}^{\ell -1} \bbE_{t,m }\| \nabla f(x^m_t) \|^2 + \frac{\gamma \ell }{2} \|x^{m}_0 -\theta_i^m\|^2  \right] \notag \\ 
&= \bbE_{\ell ,m} \|x^{m+1}_0 -x^{m}_0\|^2 + \left(1 +  \gamma \eta \ell\right) \|x^{m}_0 -\theta_i^m  \|^2  +  \frac{\eta}{\gamma} \sum_{t=0}^{\ell -1} \bbE_{t,m} \ \|\nabla f(x^m_t)\|^2 \notag \\
&=  \eta^2 \ell \sum_{t=0}^{\ell -1}\bbE_{\ell ,m}\|v^m_t \|^2 + \left(1 +  \gamma \eta \ell\right)  \|x^{m}_0 -\theta_i^m  \|^2  +  \frac{\eta}{\gamma} \sum_{t=0}^{\ell -1} \bbE_{t,m} \ \|\nabla f(x^m_t)\|^2 \notag \\
&=  \eta^2 \ell \sum_{t=0}^{\ell -1}\bbE_{t+1,m}\|v^m_t \|^2 + \left(1 +  \gamma \eta \ell\right)  \|x^{m}_0 -\theta_i^m  \|^2  +  \frac{\eta}{\gamma} \sum_{t=0}^{\ell -1} \bbE_{t,m} \ \|\nabla f(x^m_t)\|^2\,. \label{eq:lem2_cond_exp_2}
\end{align}
Combining equation~\eqref{eq:lem2_cond_exp_2} and Lemma~\ref{lem:non_convex_2}, we get:
\begin{align}
\frac{1}{n}\sum_{i =1}^n \bbE_{\ell ,m}'\| x^{m+1}_0 - \theta^{m+1}_i \|^2 &\leq  \frac{\eta^2 (\ell +1)(2\ell +1)}{6n} \sum_{t=0}^{\ell -1} \bbE_{t+1,m} \|v_t^m \|^2 + \left( 1 - \frac{\ell}{n}\right) \eta^2 \ell \sum_{t=0}^{\ell -1} \bbE_{t+1,m} \|v_t^m \|^2  \notag \\
& \qquad +\left( 1 - \frac{\ell}{n}\right)\left[\frac{(1+ \gamma \eta \ell)}{n}\sum_{i=0}^{n} \|x^m_0 - \theta^m_i \|^2 + \frac{\eta}{\gamma } \sum_{t=0}^{\ell -1} \bbE_{t,m} \ \|\nabla f(x^m_t)\|^2 \right] \notag \\
&=  \eta^2 \sum_{t=0}^{\ell -1} \bbE_{t+1,m} \|v_t^m \|^2\left( \frac{(\ell +1)(2\ell +1)}{6n} + \frac{\ell (n-\ell )}{n}\right) \notag\\
& \qquad  + \left( 1 - \frac{\ell}{n}\right)\frac{\eta}{\gamma}\sum_{t=0}^{\ell -1} \bbE_{t,m} \|\nabla f(x^m_t)\|^2 
 + \left(1+\gamma \eta \ell\right)\left(1 - \frac{\ell}{n}\right)\frac{1}{n} \sum_{i=0}^{n} \|x^m_0 - \theta^m_i \|^2 \notag \\
&\stackrel{\eqref{eq:lem_non_convex_2}}{\leq}  \eta^2 \ell \sum_{t=0}^{\ell -1} \bbE_{t+1,m}\|v_t^m \|^2 +  \left( 1 - \frac{\ell}{n}\right)\frac{\eta}{\gamma}\sum_{t=0}^{\ell -1} \bbE_{t,m}  \|\nabla f(x^m_t)\|^2 \notag \\
&\qquad \qquad \qquad \qquad \qquad  + \left(1+\gamma \eta \ell\right)\left(1 - \frac{\ell}{n}\right)\frac{1}{n} \sum_{i=0}^{n} \|x^m_0 - \theta^m_i \|^2 \notag \\
&= \eta^2 \ell  \bbE_{\ell ,m} \| V^m \|_F^2 + \left( 1 - \frac{\ell}{n}\right)\frac{\eta}{\gamma}  \bbE_{\ell ,m} \|\nabla F^m \|_F^2 \notag \\
&\qquad \qquad \qquad \qquad \qquad + \left(1+\gamma \eta \ell\right)\left(1 - \frac{\ell}{n}\right)\frac{1}{n} \sum_{i=0}^{n} \|x^m_0 - \theta^m_i \|^2 \,.
\end{align}
Hence, we have:
\begin{align*}
\bbE_{\ell ,m}' H^{m+1} \leq  \eta^2 \ell  \bbE_{\ell ,m} \| V^m \|_F^2  + \left( 1 - \frac{\ell}{n}\right)\frac{\eta}{\gamma} \bbE_{\ell ,m} \|\nabla F^m \|_F^2+\left(1+\gamma \eta \ell\right)\left(1 - \frac{\ell}{n}\right) H^m\,.   &\qedhere
\end{align*}
\end{proof}

\begin{proof}[\textbf{Proof of Lemma~\ref{lemm:non_convex_lemm2}}]
From Lemma~\ref{lem:lem_non_convex_1}, we have:
\begin{align}
\bbE_{t+1,m} \|v^m_t \|^2 \leq  2 \bbE_{t,m} \big\| \nabla f(x^m_t) \big \|^2 + 4L^2 \eta^2 t  \sum_{j =0}^{t-1}\bbE_{j+1,m} \| v_j^m \|^2 + \frac{4L^2}{n} \sum_{i = 1}^n \| x^m_0 - \theta^m_i \|^2\,. \label{eq:result_from_lem_nc1}
\end{align}
We sum the equation~\eqref{eq:result_from_lem_nc1} for $t=0$ to $t=\ell -1$ to get the following:
\begin{align}
\sum_{t=0}^{\ell -1} \bbE_{t+1,m} \|v^m_t \|^2 &\leq 2\sum_{t=0}^{\ell -1} \bbE_{t,m} \big\| \nabla f(x^m_t) \big \|^2 +4 L^2\eta^2 \sum_{t=0}^{\ell -1} t  \sum_{j =0}^{t-1}\bbE_{j+1,m} \| v_j^m \|^2 + \frac{4L^2l}{n} \sum_{i = 1}^n \| x^m_0 - \theta^m_i \|^2  \notag \\
&\leq  2 \bbE_{\ell ,m}\|\nabla F^m\|_F^2 + 2 L^2\eta^2 \ell(\ell -1) \sum_{t=0}^{\ell -1} \bbE_{t+1,m} \|v^m_t \|^2 + \frac{4L^2\ell}{n} \sum_{i = 1}^n \| x^m_0 - \theta^m_i \|^2  \notag \\
& \leq  2 \bbE_{\ell ,m}\|\nabla F^m\|_F^2 + 2 L^2\eta^2 {\ell}^2 \sum_{t=0}^{\ell -1} \bbE_{t+1,m} \|v^m_t \|^2 + \frac{4L^2 \ell}{n} \sum_{i = 1}^n \| x^m_0 - \theta^m_i \|^2 \,.
\end{align}
Since, $\sum_{t=0}^{\ell -1} \bbE_{t+1,m} \|v^m_t \|^2 = \bbE_{m} \|V^m \|_F^2$, we get the following relation:
\begin{align}
(1-2L^2\eta^2{\ell}^2)\bbE_{\ell ,m} \| V^m\|_F^2  &\leq  2 \bbE_{\ell ,m}\|\nabla F^m\|_F^2 + \frac{4L^2 \ell}{n} \sum_{i = 1}^n  \| x^m_0 - \theta^m_i \|^2 \notag \\
&\leq  2 \bbE_{\ell ,m}\|\nabla F^m\|_F^2 + {4L^2\ell} H^m\,.
\end{align}
Hence, finally we have:
\begin{align}
(1-2L^2\eta^2{\ell}^2)\bbE_{m} \| V^m\|_F^2 \leq \bbE_{m} \|\nabla F^m\|_F^2 + {4L^2\ell}  H^m\,. \label{eq:lem_7_final}
\end{align}
\begin{remark}
Unfortunately, equation~\eqref{eq:lem_7_final} limits us to choose $\ell$ as large as we would like (e.g. $\ell =n$ in case $k=1$), as otherwise the term  $(1-2L^2\eta^2{\ell}^2)$ would become too small. In the proof of Theorem~\ref{thm:main_theorem_non_convex} we will choose $\eta = \cO(\frac{1}{Ln^{2/3}})$ and hence $\ell $ should be less than of the order of $\mathcal{O}(n^{2/3})$.
\end{remark}
\vspace{-1.65em}
\end{proof}
\begin{proof}[\textbf{Proof of Lemma~\ref{lem:non_convex_general_result}}]
The Lyapunov function is of the form:
\begin{align}
\bbE_{\ell ,m}' \mathcal{L}^m(x^{m+1}_0) 
 &= \bbE_{\ell ,m} f(x^{m+1}_0) + \frac{c_{m+1}}{n} \sum_{i =1}^n \bbE_{\ell ,m}' \| x^{m+1}_0 - \theta_i^m \|^2\,.
\end{align}
First we analyze the term $\bbE_{\ell ,m} f(x^{m+1}_0)$ in the Lyapunov function. By the smoothness assumption:
\begin{align}
f(x_{t+1}^m) \leq f(x_t^m) + \big \langle \nabla f(x_t^m) , x_t^m - x_{t+1}^m \big\rangle +\frac{L}{2} \| x_t^m - x_{t+1}^m \|^2\,.
\end{align} 
Now if we take expectation conditioned on $x_t^m$, we get:
\begin{align}
\bbE_{\{i_t\}}  f(x_{t+1}^m) & \leq f(x_t^m) - \eta \|\nabla f(x_t^m)  \|^2 +\frac{\eta^2 L}{2} \bbE_{\{i_t\}} \| v_t^m\|^2\,. \label{eq:cond_exp_smooth}
  \end{align}
In equation~\eqref{eq:cond_exp_smooth}, we apply the property of tower of conditional expectations and sum equation~\eqref{eq:cond_exp_smooth} from $t =0$ to $t =\ell -1$ in the $m^{th}$ outer loop  to get the following:
\begin{align}
&\bbE_{\ell ,m} f(x^m_{\ell}) \leq f(x^m_0) - \eta  \sum_{t = 0}^{\ell -1}\bbE_{t,m}\|\nabla f(x^m_{t}) \|^2 + \frac{\eta^2 L}{2} \sum_{t = 0}^{\ell -1} \bbE_{t+1,m} \|v_t^m \|^2\,.
\end{align}
Hence, we have:
\begin{align}
\bbE_m f(x^{m+1}_0)  \leq f(x^m_0) - \eta \bbE_{m} \|\nabla F^m \|_F^2  + \frac{\eta^2 L}{2} \bbE_{m} \|V^m \|_F^2\,. \label{eq:smooth_telescopic_sum}
\end{align}
We now analyze the complete Lyapunov function by using the results from Lemmas~\ref{lemm:non_convex_lemm1} and \ref{lemm:non_convex_lemm2}.
\begin{align}
\bbE_{\ell ,m}' \mathcal{L}^m(x^{m+1}_0)  
&= \bbE_{\ell ,m} f(x^{m+1}_0) + \frac{c_{m+1}}{n} \sum_{i =1}^n \bbE_{\ell ,m}'  \| x^{m+1}_0 - \theta_i^m \|^2 \notag \\
&\stackrel{\eqref{eq:smooth_telescopic_sum}} {\leq}  f(x^m_0) - \eta \bbE_{\ell ,m} \|\nabla F^m \|_F^2  + \frac{\eta^2 L}{2}  \bbE_{\ell ,m} \|V^m \|_F^2 + \frac{c_{m+1}}{n} \sum_{i =1}^n \bbE_{\ell ,m}  \| x^{m+1}_0 - \theta_i^m \|^2 \notag \\
&\stackrel{\eqref{eq:lemm_non_convex_lemm1}}{\leq} f(x^m_0) - \eta \bbE_m \|\nabla F^m \|_F^2  + \frac{\eta^2 L}{2}  \bbE_m \|V^m \|_F^2 \notag \\ & \qquad  + c^{m+1}\eta^2 \ell  \bbE_{m} \| V^m \|_F^2 + c^{m+1}\left( 1 - \frac{\ell}{n}\right)\frac{\eta}{\gamma} \bbE_{m} \|\nabla F^m \|_F^2  \notag \\
&  \qquad + c^{m+1}\left(1+\gamma \eta \ell\right)\left(1 - \frac{\ell}{n}\right)\frac{1}{n} \sum_{i=0}^{n} \|x^m_0 - \theta^m_i \|^2 \notag \\
&=  f(x^m_0) - \left( \eta -  c^{m+1}\left( 1 - \frac{\ell}{n}\right)\frac{\eta}{\gamma} \right) \bbE_{m} \|\nabla F^m \|_F^2 + \left( \frac{\eta^2L}{2}+ c^{m+1}\eta^2\ell\right) \bbE_{m} \| V^m\|_F^2 \notag \\
& \qquad  + \left[ c^{m+1}\left(1+\gamma \eta \ell\right)\left(1 - \frac{\ell}{n}\right) \right]\frac{1}{n} \sum_{i=0}^{n} \|x^m_0 - \theta^m_i \|^2 \,.
\end{align}
Let $b_1 := \frac{1}{1 - 2L^2 \eta^2 {\ell}^2}$, as in the main text. Hence from Lemma~\ref{lemm:non_convex_lemm2}, we get:
\begin{align}
\bbE_{\ell ,m}' \mathcal{L}^m(x^{m+1}_0) &\leq f(x^m_0) - \left( \eta -  c^{m+1}\left( 1 - \frac{\ell}{n}\right)\frac{\eta}{\gamma} \right) \bbE_{\ell ,m} \|\nabla F^m \|_F^2 \notag \\
&  \qquad + \left( \frac{\eta^2L}{2}+ c^{m+1}\eta^2\ell \right) \left[ 2b_1 \bbE_{\ell ,m}  \|\nabla F^m \|_F^2 + \frac{4b_1L^2\ell}{n} \sum_{i =i}^n  \|x^m_0 - \theta^m_i \|^2\right]\notag \\
&  \qquad + \left( c^{m+1}\left(1+\gamma \eta \ell\right)\left(1 - \frac{\ell}{n}\right) \right)\frac{1}{n} \sum_{i=0}^{n} \|x^m_0 - \theta^m_i \|^2 \notag \\
&\leq f(x^m_0) - \left( \eta -  c^{m+1}\left( 1 - \frac{\ell}{n}\right)\frac{\eta}{\gamma} - b_1 \eta^2 L - 2 b_1 c^{m+1}\eta^2 \ell\right)\bbE_{\ell ,m} \|\nabla F^m \|_F^2 \notag \\
&\qquad   +\left[ c^{m+1}\left(1+\gamma \eta \ell\right)\left(1 - \frac{\ell}{n}\right) + 2b_1\eta^2L^3 \ell +4b_1 c^{m+1} \eta^2L^2 {\ell}^2 \right] \frac{1}{n}\sum_{i=1}^n \|x^m_0 - \theta_i^m \|^2 \notag \\
&\leq f(x^m_0) - \left( \eta -  c^{m+1}\left( 1 - \frac{\ell}{n}\right)\frac{\eta}{\gamma} - b_1 \eta^2 L - 2 b_1 c^{m+1}\eta^2 \ell\right) \bbE_{\ell ,m} \|\nabla F^m \|_F^2 \notag \\
&\qquad   +\left[ c^{m+1}\left(1+\gamma \eta \ell\right)\left(1 - \frac{\ell}{n}\right) + 2b_1\eta^2L^3 \ell +4b_1 c^{m+1} \eta^2L^2 {\ell}^2 \right] \frac{1}{n}\sum_{i=1}^n \|x^m_0 - \theta_i^m \|^2 \notag \\
&\leq f(x^m_0) - \underbrace{\left( \eta -  c^{m+1}\frac{\eta}{\gamma} - b_1 \eta^2 L - 2 b_1 c^{m+1}\eta^2 \ell\right)}_{= \Gamma^m} \bbE_{\ell ,m} \|\nabla F^m \|_F^2 \notag \\ 
&\qquad   + \underbrace{\left[ c^{m+1}\left(1 - \frac{\ell}{n} + \gamma \eta \ell + 4b_1 \eta^2 L^2 {\ell}^2\right) + 2b_1\eta^2L^3 \ell  \right]}_{=c^m} \frac{1}{n}\sum_{i=1}^n \|x^m_0 - \theta_i^m \|^2\,. \label{eq:recurrence_final} 
\end{align}

We finally get:
\begin{align}
& \Gamma^m \bbE_{\ell ,m} \|\nabla F^m \|^2 \leq  \mathcal{L}^m(x^{m}_0) - \bbE_{\ell ,m}' \mathcal{L}^{m+1}(x^{m+1}_0)\,,
\end{align}
and the claim follows.

\end{proof}
\begin{proof}[\textbf{Proof of Theorem~\ref{thm:main_theorem_non_convex}}]

We add equation~\eqref{eq:recurrence_non_convex} from Lemma~\ref{lem:non_convex_general_result} for $m=0$ to $m=M-1$ and take expecation with respect to the joint distribution of all the selection so far which gives:
\begin{align}
\sum_{m =0}^{M-1}\Gamma^m \bbE \|\nabla F^m \|^2 \leq \mathcal{L}^0(x^{0}_0) - \bbE \mathcal{L}^M(x^{M}_0,H^{M}) \,. \label{eq:previous}
\end{align}
Since $\Gamma = \min_{0\leq m\leq M-1}$, we get
\begin{align}
 \Gamma \sum_{i =0}^{M-1}  \bbE \|\nabla F^m \|^2 \leq {\mathcal{L}^0(x_0^0) - \bbE \mathcal{L}^M(x^{M}_0)}\,,
\end{align}
from~\eqref{eq:previous} and the first part of the theorem follows. To show the second part we need to derive a lower bound on $\Gamma$ for the given parameters $\eta=  \frac{1}{5L n^{2/3}}$, $\gamma = \frac{L}{n^{1/3}}$  and $\ell =\frac{3}{2} n^{1/3}$. Observe that for these parameters $b_1 \leq 2$.

First, let us derive an upper bound on $c^m$. Let $\lambda := \ell \left( \frac{1}{n} - \gamma \eta  - 4b_1 \eta^2 L^2 \ell\right)$. We have $\frac{8 \ell}{25 n} \leq \lambda \leq 1$, where the upper bound is immediate and the lower bound follows from $\left( \frac{1}{n} - \gamma \eta  - 4b_1 \eta^2 L^2 \ell\right) \geq \left(\frac{1}{n} - \frac{1}{5n} - \frac{12}{25n} \right) = \frac{8}{25n}$, using $b_1 \leq 2$. Observe that we have $c^{m} = c^{m+1}(1-\lambda) + 2 b_1 \eta^2 L^3 \ell$. Using this relationship and $c^{M}=0$, it is easy to see that 
\begin{align}
c^m = 2b_1\eta^2 L^3 \ell \frac{1 - (1-\lambda)^{M-m}}{\lambda} \leq \frac{2b_1\eta^2 L^3 \ell}{\lambda}  \leq \frac{L}{2n^{1/3}} \,,
\end{align}
for all $m=0,\dots, M$. Now we are ready to derive a lower bound on $\Gamma^m$. 

Using $\frac{\eta }{\gamma} = \frac{1}{5L^2 n^{1/3}}$ and $c^m \leq \frac{L}{2n^{1/3}}  $, we get:
\begin{align}
\Gamma^m & \geq \frac{1}{5L n^{2/3}} - \frac{1}{10 L n^{2/3}} - b_1 \eta^2 L - 2 b_1 c^{m+1}\eta^2 \ell \notag \\
& = \frac{1}{10 L n^{2/3}} - \underbrace{\left(b_1 \eta^2 L + 2 b_1 c^{m+1}\eta^2 \ell \right) }_{=: g_1} \label{eq:g1}
\end{align}
Now we consider the term $g_1$. As $b_1 \leq 2$ we have
\begin{align}
g_1 
&\leq 2 \eta^2 L + 4 c^{m+1} \eta^2 \ell \notag \\
& \leq \frac{2}{25 L n^{4/3} } + \frac{3}{25 L n^{4/3} }  \leq \frac{1}{30 Ln^{2/3}} \,, \label{eq:g2}. 
\end{align}
where the last inequality is due to $n > 15$.
 By combining~\eqref{eq:g1} and~\eqref{eq:g2} we get $\Gamma_m \geq \frac{1}{15 Ln^{2/3}}$ for $m=0,\dots,M$. Hence, $\Gamma \geq \frac{1}{15 Ln^{2/3}}$. 
\end{proof}

%% file: appendix-figs.tex
\section{Additional Experimental Results} \label{app:add_expts}
\subsection{Illustrative Experiment with more $k$-SVRG variants}
\begin{figure*}[h]
\centering
\hfill
  \includegraphics[width=0.32\textwidth]{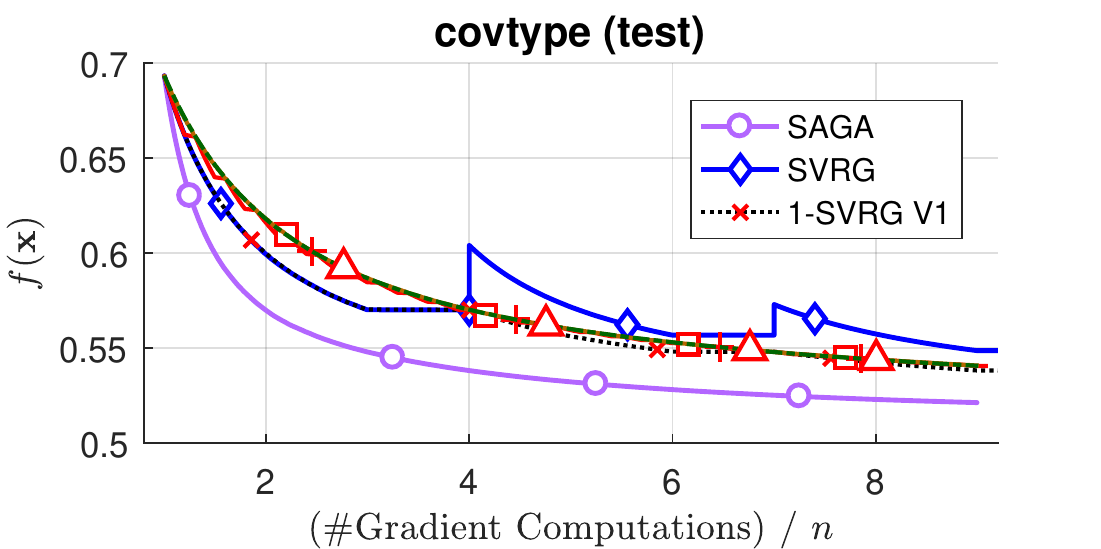}
\hfill
  \includegraphics[width=0.32\textwidth]{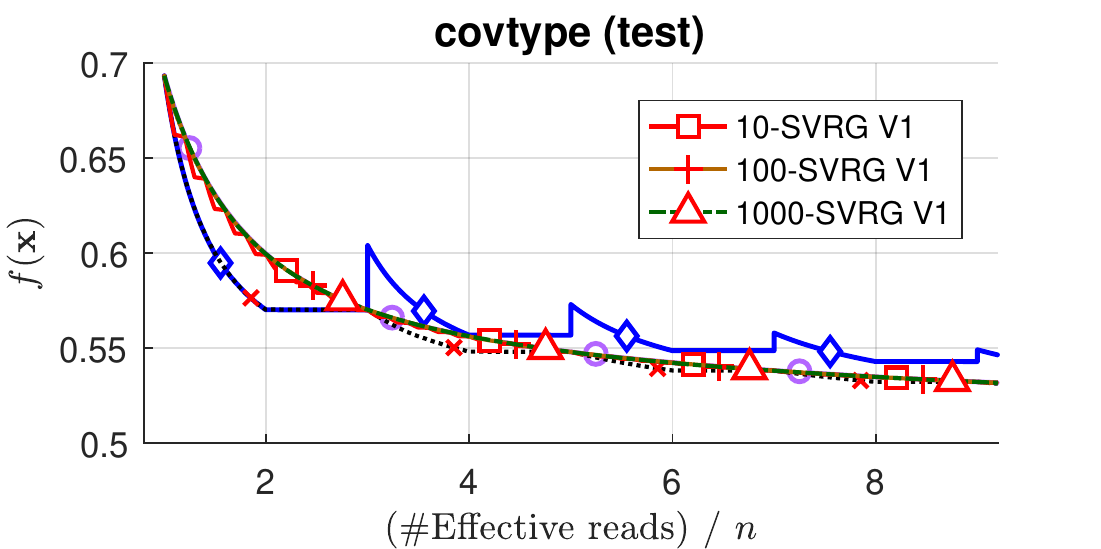}
\hfill\null  
\caption{Illustrating the different convergence behavior of SAGA, SVRG and \mVone{} for $k=\{1,10,100,1000\}$.}
\label{fig:demo}
\end{figure*}

\subsection{Dataset: \emph{covtype (test)}}
\null
\begin{multicols}{2}
\begin{figurehere}
\centering
  \includegraphics[width=.49\linewidth]{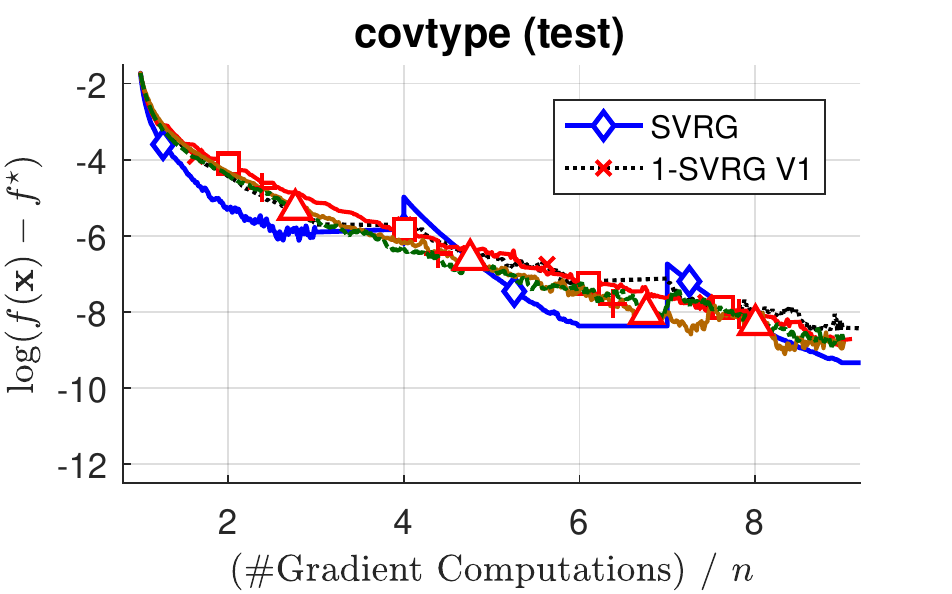}
\hfill
  \includegraphics[width=0.49\linewidth]{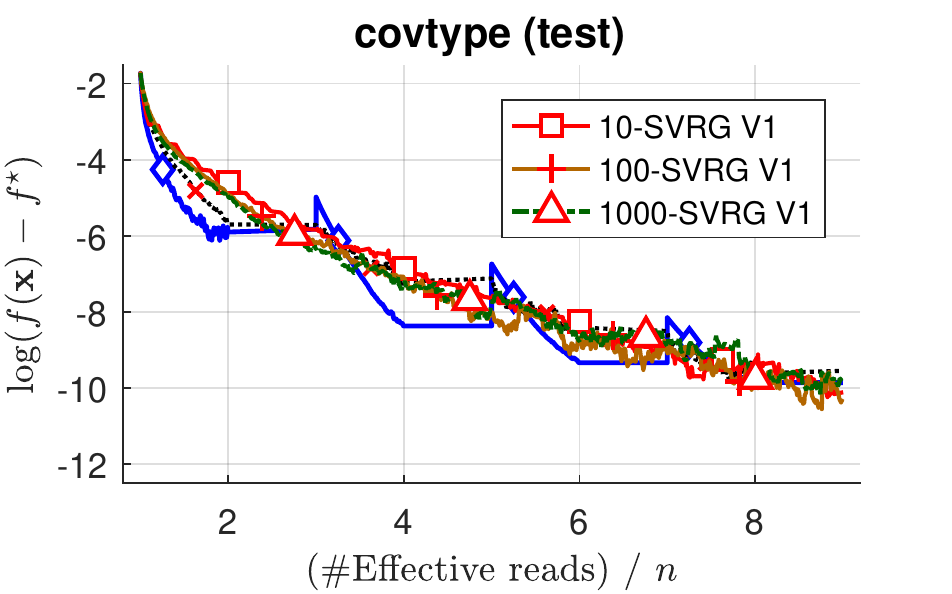}
  \vspace*{-1.5em}
\caption{Evolution of residual loss on \emph{covtype (test)} for SVRG and \mVone{} for  $k=\{1,10,100,1000\}$.}
\label{fig:compare_var1}
\end{figurehere}
\begin{figurehere}
\centering
  \includegraphics[width=0.49\linewidth]{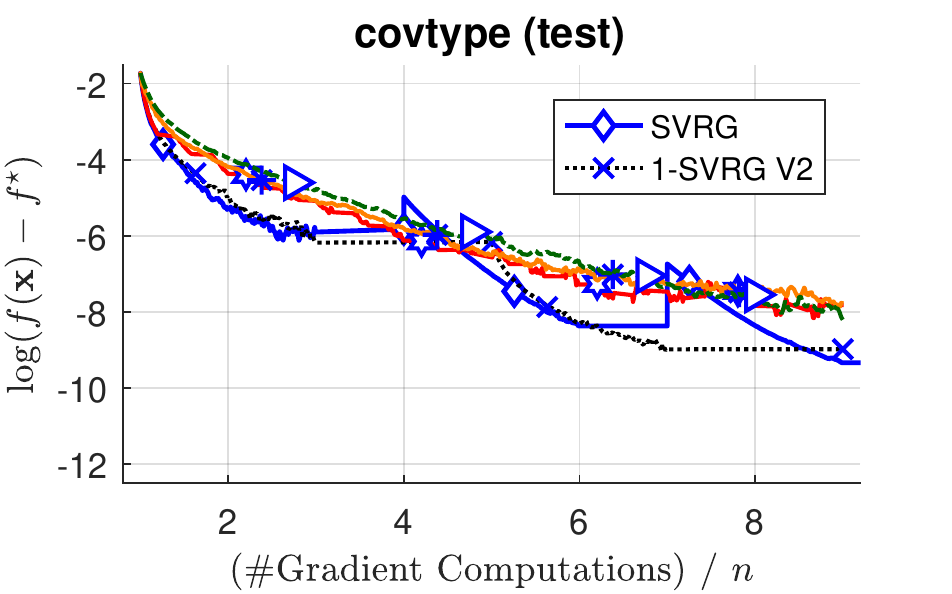}
\hfill
  \includegraphics[width=0.49\linewidth]{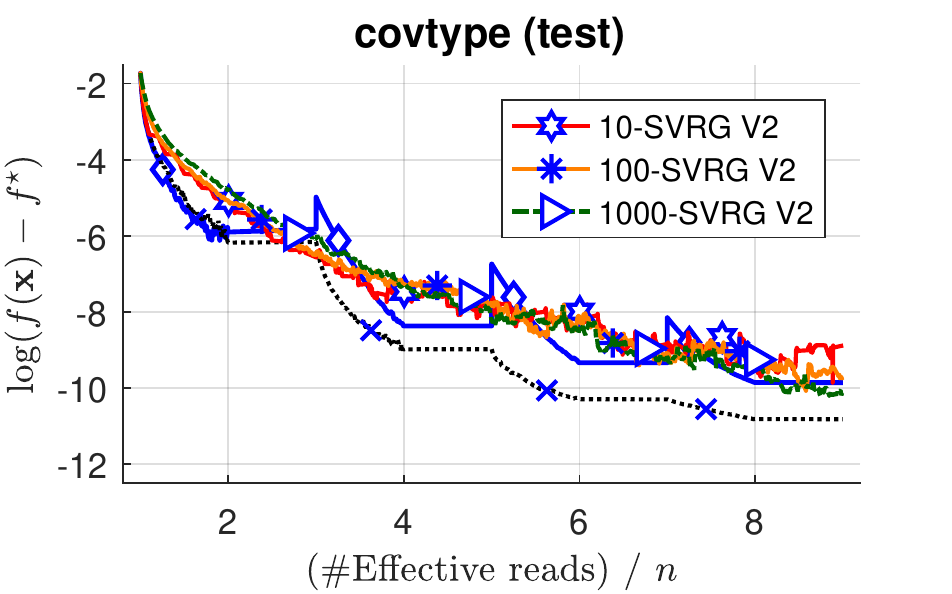}
  \vspace*{-1.5em}
\caption{Evolution of residual loss on \emph{covtype (test)} for  SVRG and \mVtwo{} for  $k=\{1,10,100,1000\}$.}
\label{fig:compare_var2}
\end{figurehere}
\end{multicols}

\begin{multicols}{2}
\begin{figurehere}
\centering
  \includegraphics[width=0.49\linewidth]{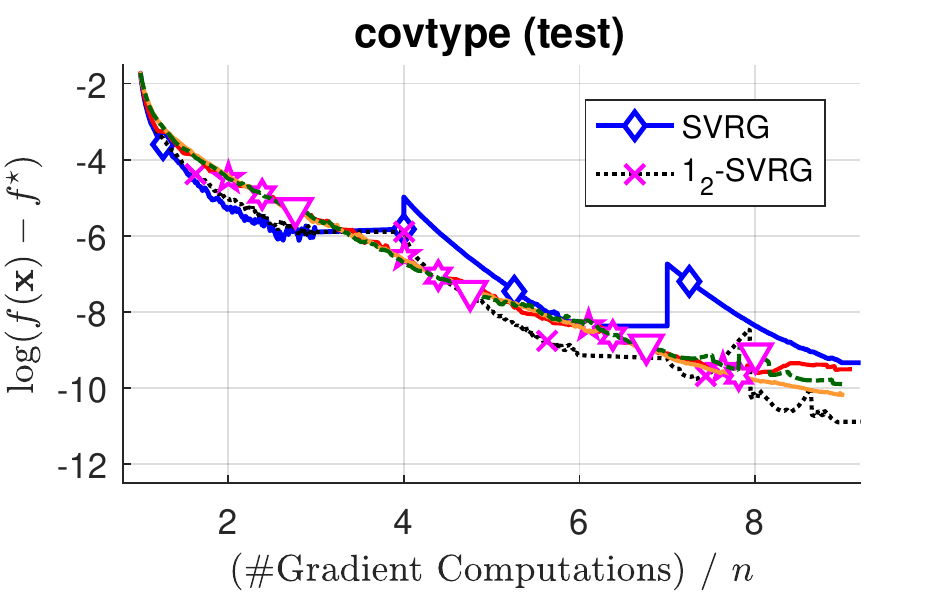}
\hfill
  \includegraphics[width=0.49\linewidth]{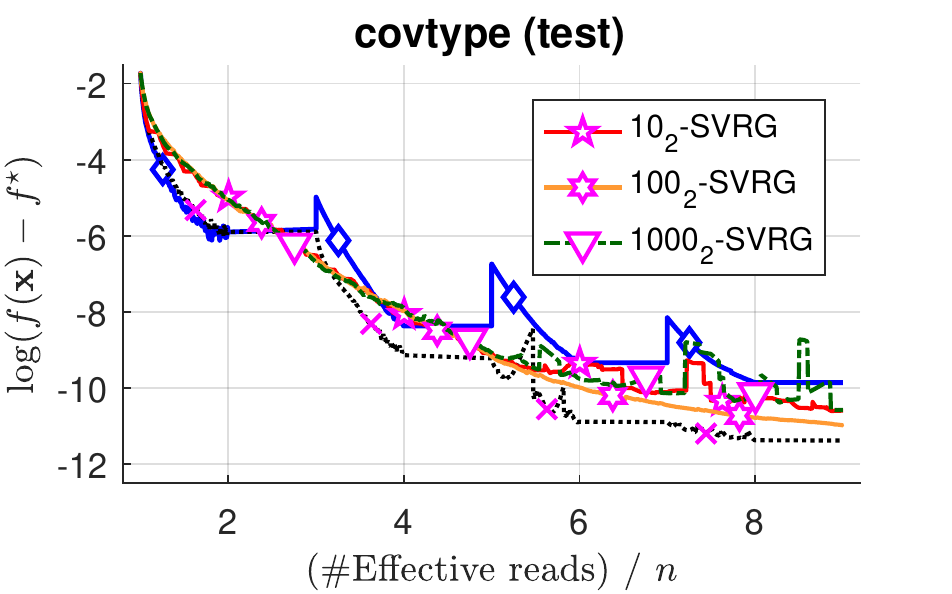}
  \vspace*{-1.5em}
\caption{Evolution of residual loss on \emph{covtype (test)} for SVRG and \mpract{} for  $k=\{1,10,100,1000\}$.}
\label{fig:compare_pract}
\end{figurehere}
\begin{figurehere}
\centering
  \includegraphics[width=0.49\linewidth]{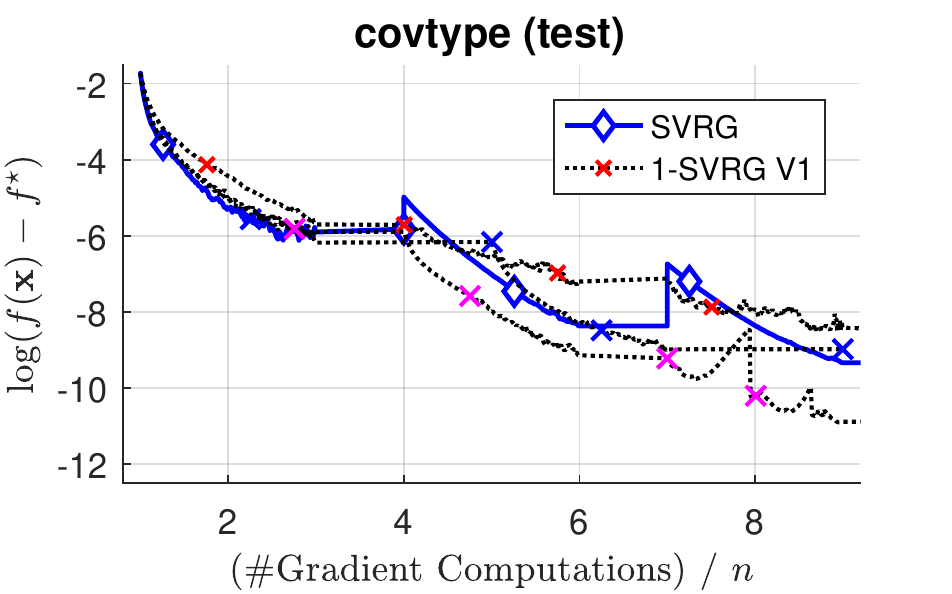}
\hfill
  \includegraphics[width=0.49\linewidth]{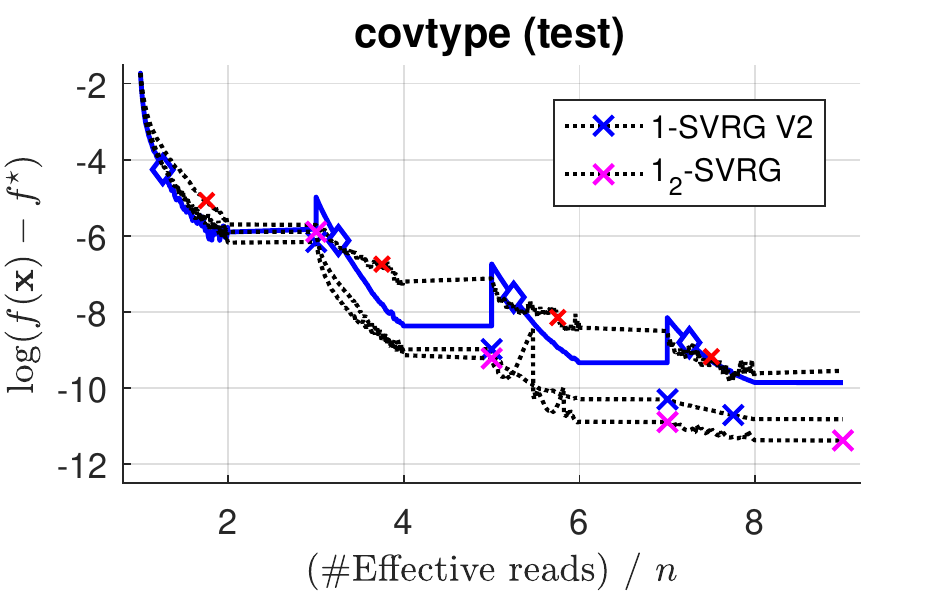}
  \vspace*{-1.5em}  
\caption{Evolution of residual loss on \emph{covtype (test)} for  SVRG, $1$\mHone, $1$\mHtwo{} and $1_2$\mHpract.}
\label{fig:compare_k1}
\end{figurehere}
\end{multicols}

\begin{multicols}{2}
\begin{figurehere}
\centering
  \includegraphics[width=0.49\linewidth]{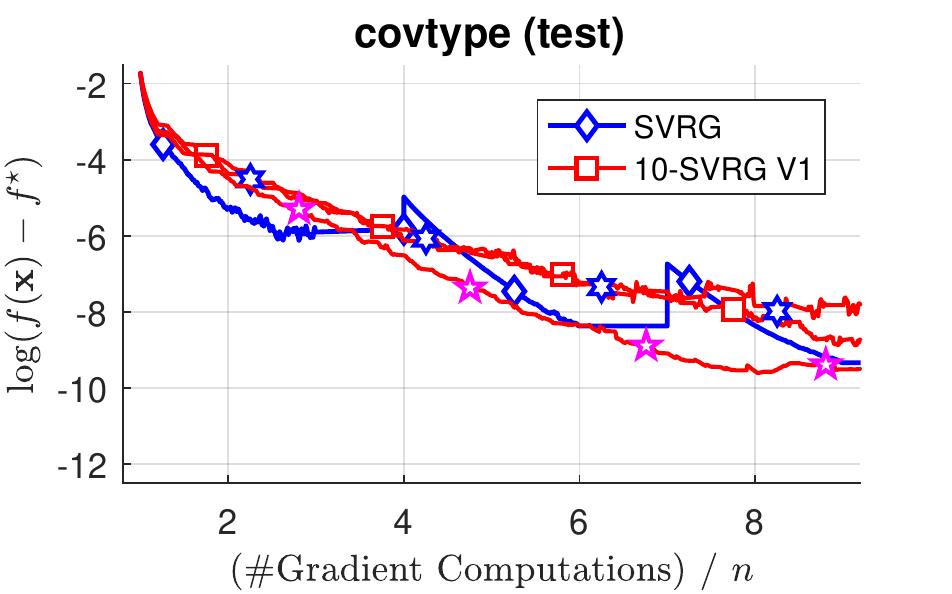}
\hfill
  \includegraphics[width=0.49\linewidth]{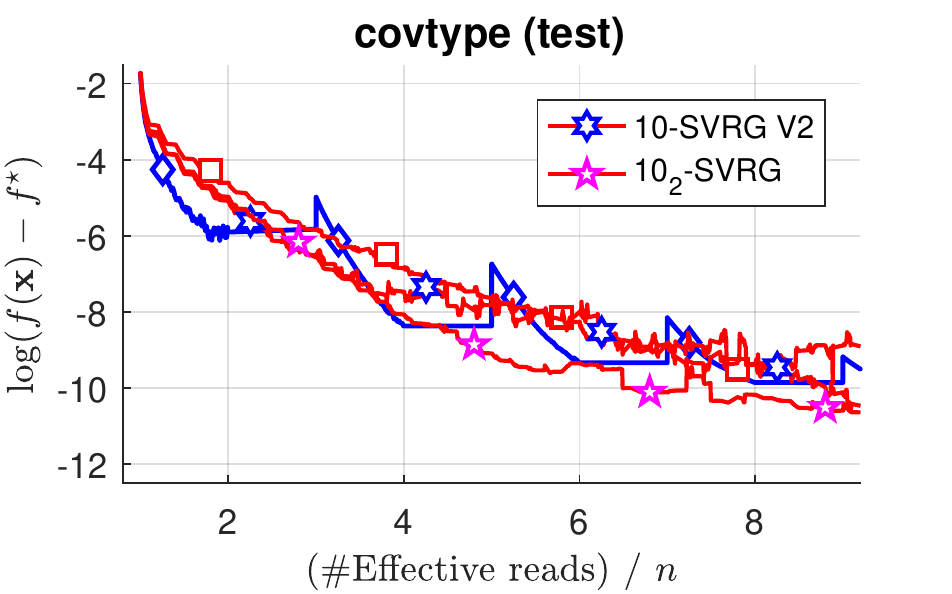}
  \vspace*{-1.5em}  
\caption{Evolution of residual loss on \emph{covtype (test)} for  SVRG, $10$\mHone, $10$\mHtwo{} and $10_2$\mHpract.}
\label{fig:compare_k10}
\end{figurehere}
\begin{figurehere}
\centering
  \includegraphics[width=0.49\linewidth]{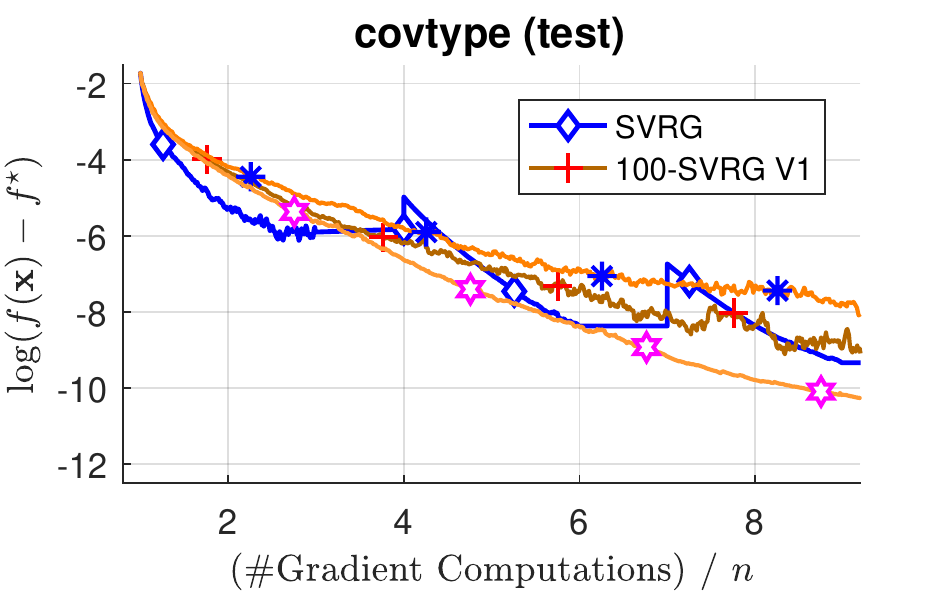}
\hfill
  \includegraphics[width=0.49\linewidth]{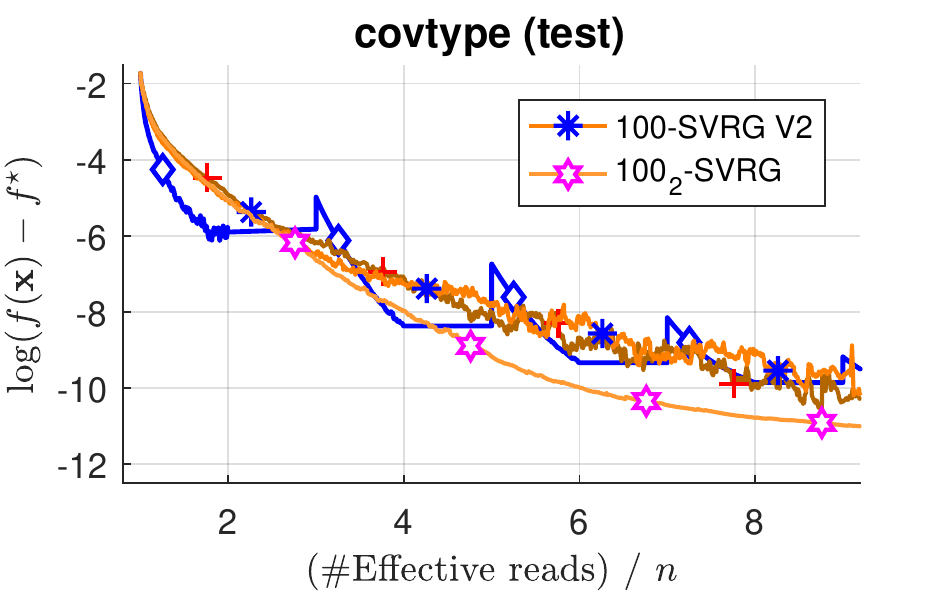}
  \vspace*{-1.5em}  
\caption{Evolution of residual loss on \emph{covtype (test)} for  SVRG, $100$\mHone, $100$\mHtwo{} and $100_2$\mHpract.}
\label{fig:compare_k100}
\end{figurehere}
\end{multicols}
\begin{multicols}{2}
\begin{figurehere}
\centering
  \includegraphics[width=0.49\linewidth]{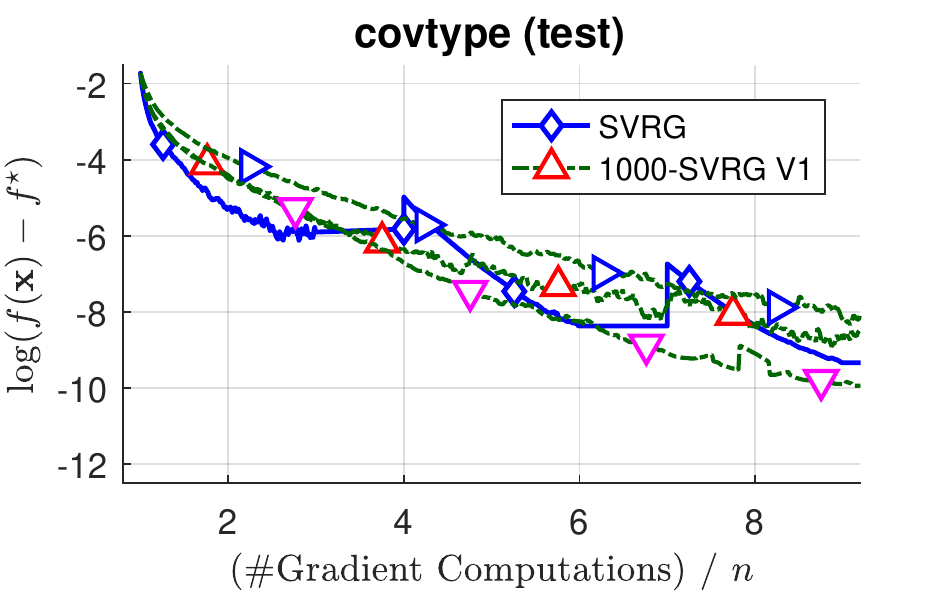}
\hfill
  \includegraphics[width=0.49\linewidth]{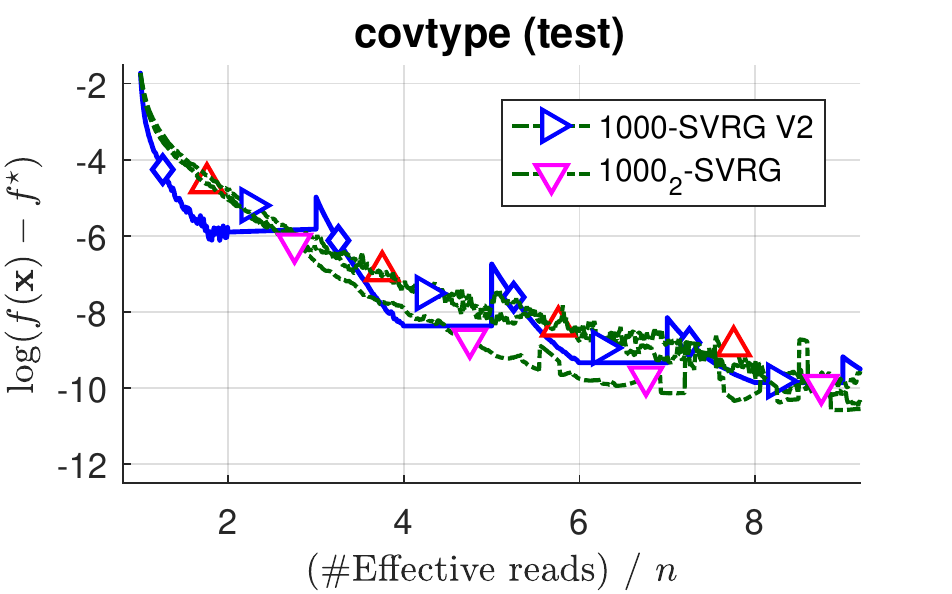}
  \vspace*{-1.5em}  
\caption{Evolution of residual loss on \emph{covtype (test)} for SVRG, $1000$\mHone, $1000$\mHtwo{} and $1000_2$\mHpract.}
\label{fig:compare_k1000}
\end{figurehere}

\end{multicols}

\clearpage
\subsection{Dataset: \emph{mnist}}
\begin{multicols}{2}
\begin{figurehere}
\centering
  \includegraphics[width=0.49\linewidth]{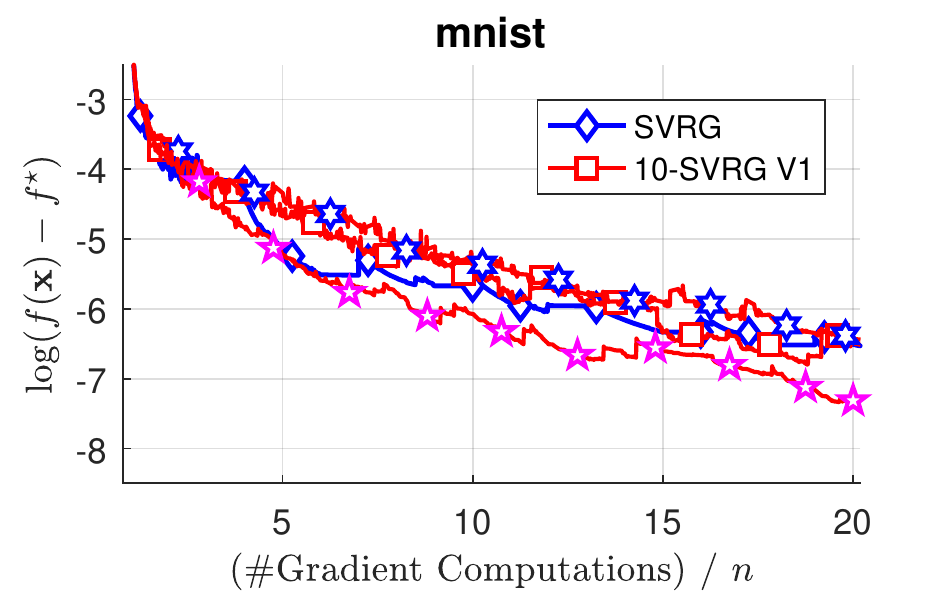}
\hfill
  \includegraphics[width=0.49\linewidth]{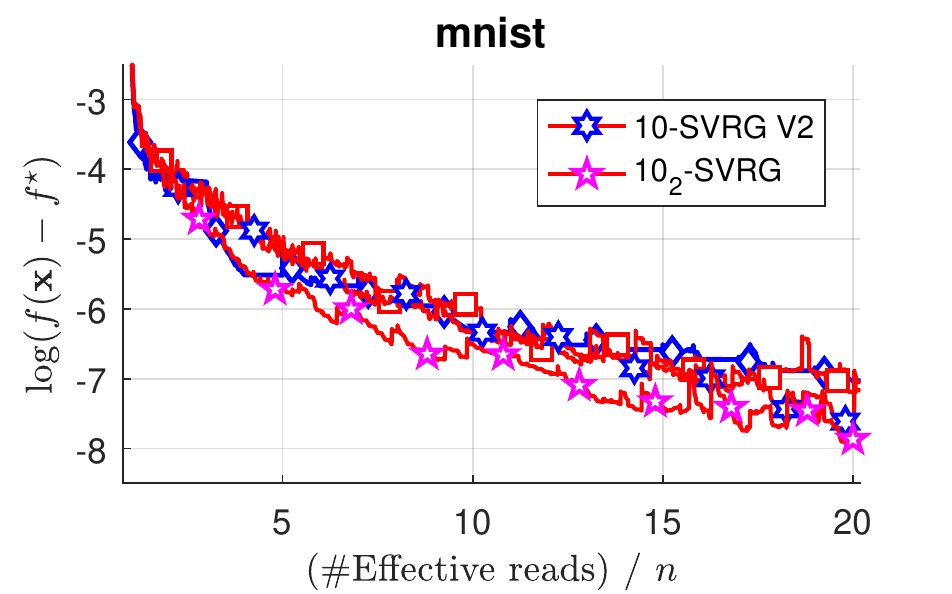}
  \vspace*{-1.5em}  
\caption{Evolution of residual loss on \emph{mnist} for SVRG, $10$\mHone, $10$\mHtwo{} and $10_2$\mHpract.}
\label{fig:mnist_compare_k10}
\end{figurehere}
\begin{figurehere}
\centering
  \includegraphics[width=0.49\linewidth]{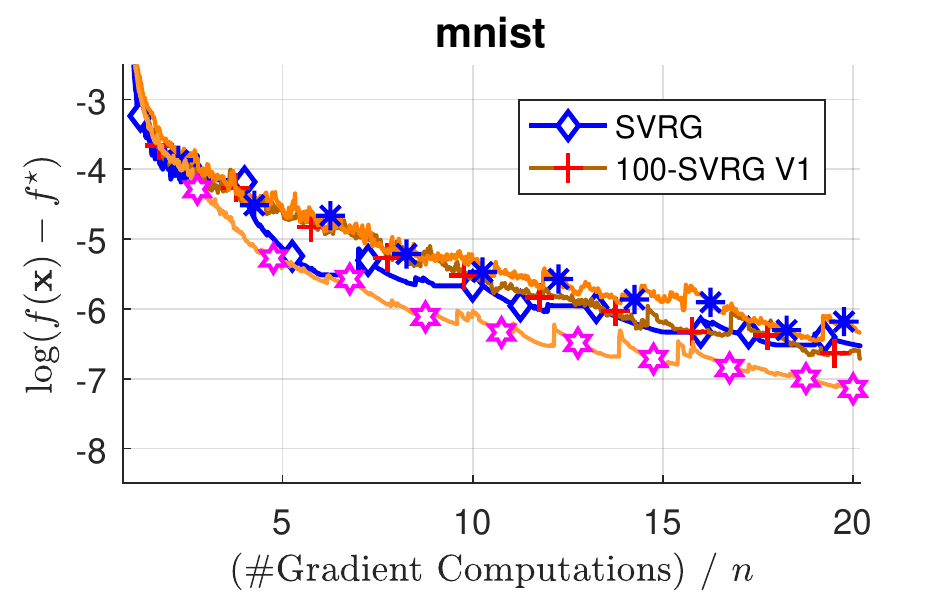}
\hfill
  \includegraphics[width=0.49\linewidth]{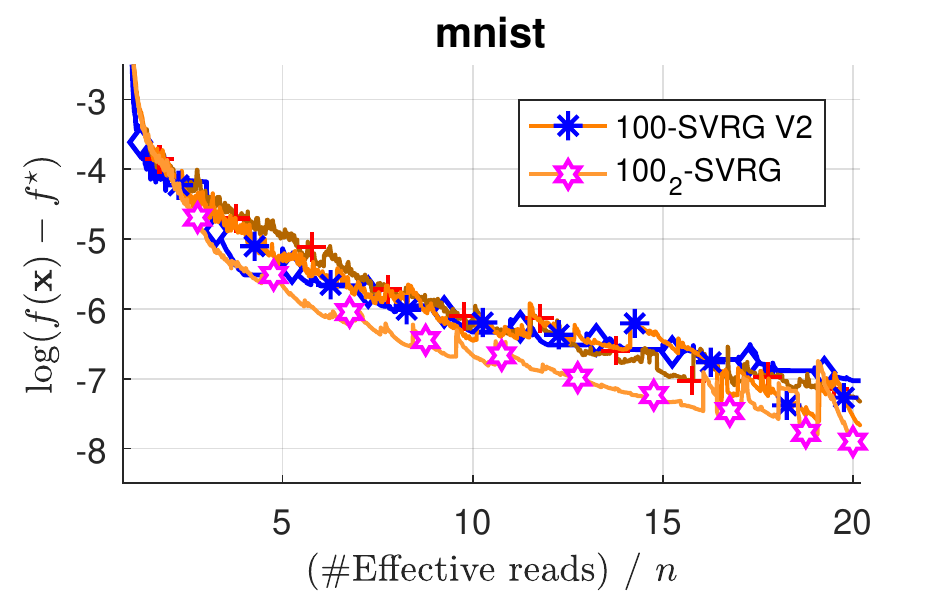}
  \vspace*{-1.5em}  
\caption{Evolution of residual loss on \emph{mnist} for SVRG, $100$\mHone, $100$\mHtwo{} and $100_2$\mHpract.}
\label{fig:mnist_compare_k100}
\end{figurehere}
\end{multicols}
\begin{multicols}{2}
\begin{figurehere}
\centering
  \includegraphics[width=0.49\linewidth]{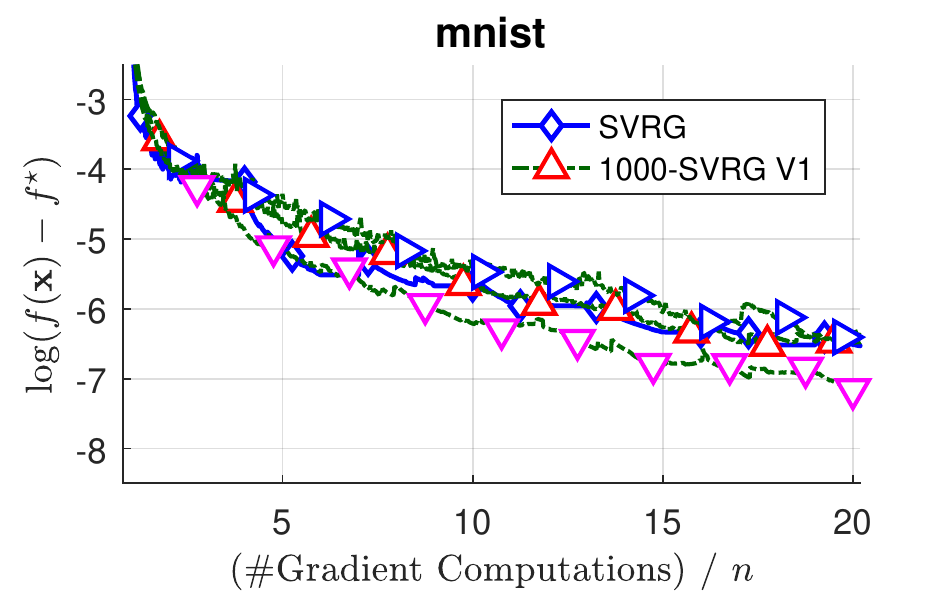}
\hfill
  \includegraphics[width=0.49\linewidth]{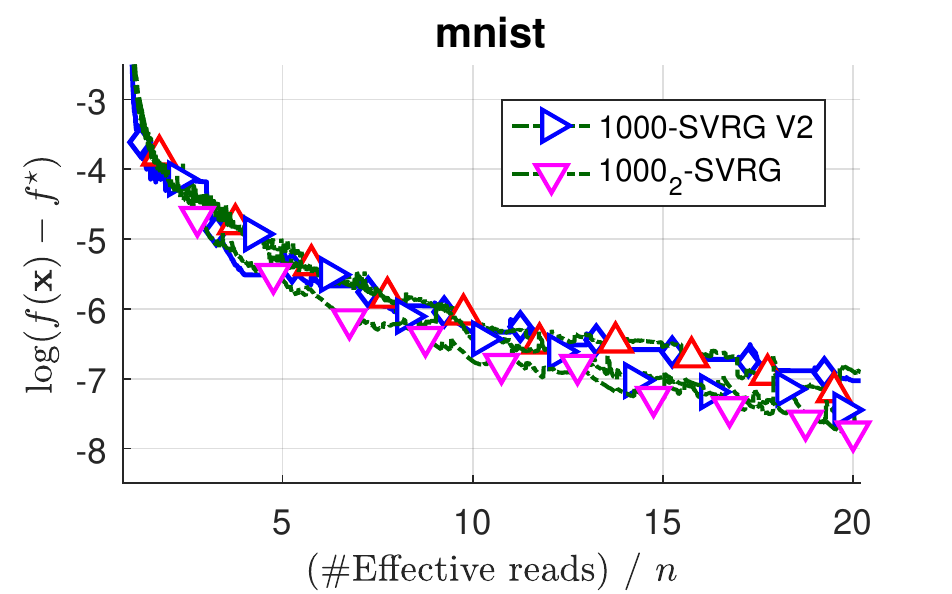}
  \vspace*{-1.5em}  
\caption{Evolution of residual loss on \emph{mnist} for  SVRG, $1000$\mHone, $1000$\mHtwo{} and $1000_2$\mHpract.}
\label{fig:mnist_compare_k1000}
\end{figurehere}

\vspace{8cm}\null
\end{multicols}

\begin{figurehere}
\centering
\hfill
  \includegraphics[width=0.325\linewidth]{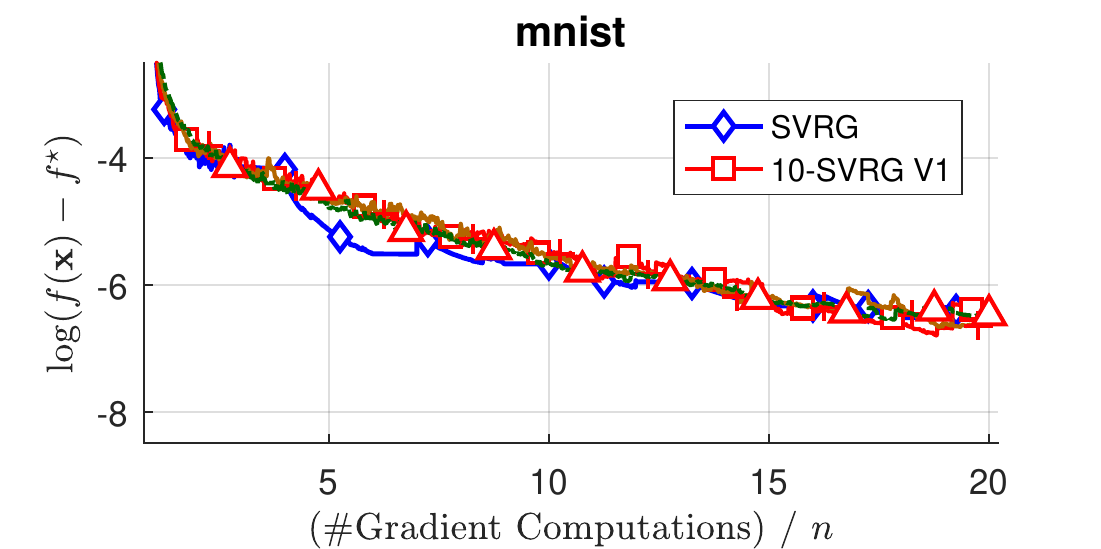}
\hfill
  \includegraphics[width=0.325\linewidth]{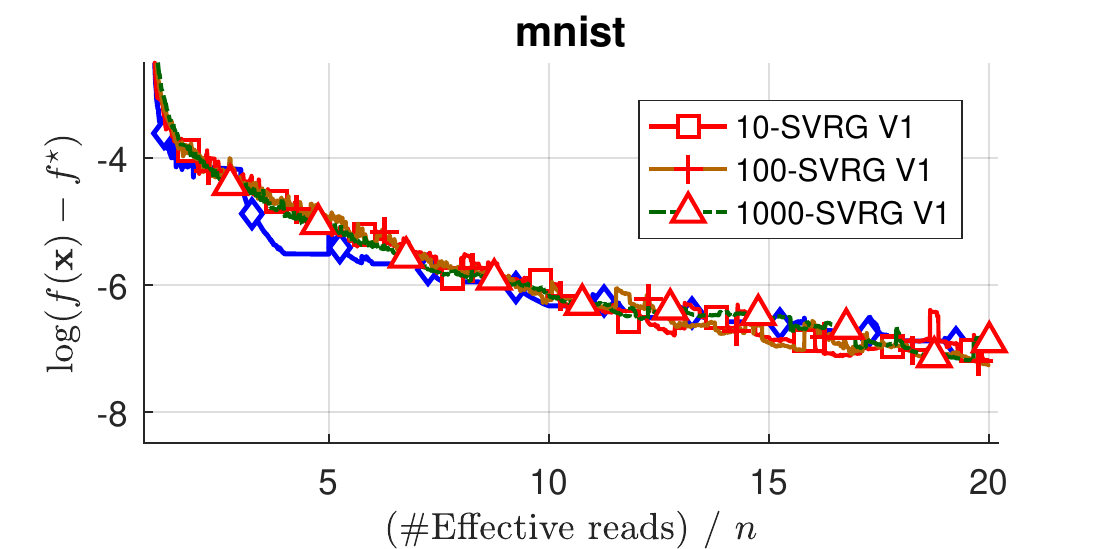}  
\hfill\null   
\caption{Evolution of residual loss on \emph{mnist} for  SVRG and \mVone{} for  $k=\{10,100,1000\}$.}
\label{fig:mnist_compare_var1}
\end{figurehere}
\begin{figurehere}
\centering
\hfill
  \includegraphics[width=0.325\linewidth]{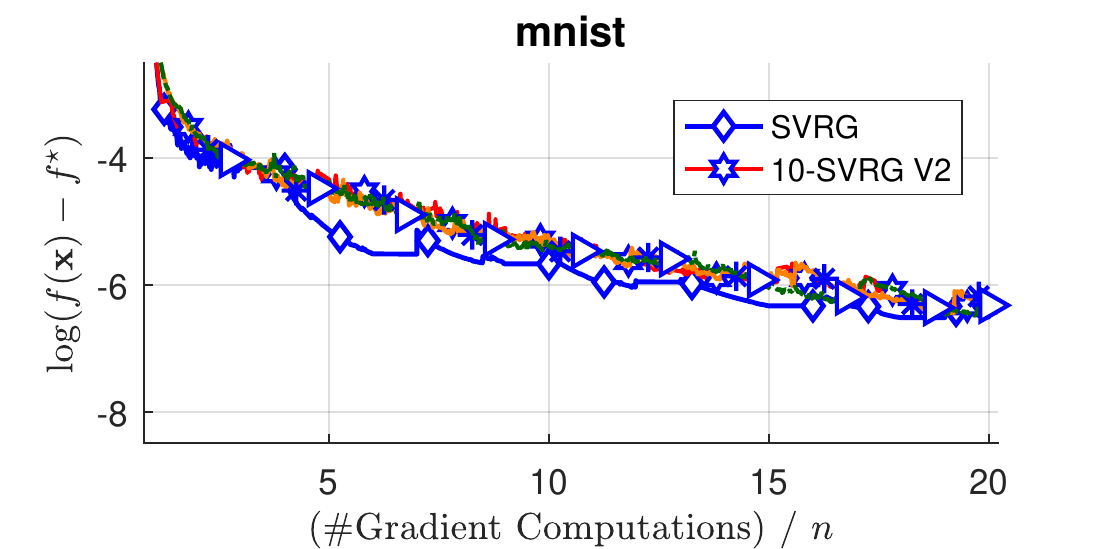}
\hfill
  \includegraphics[width=0.325\linewidth]{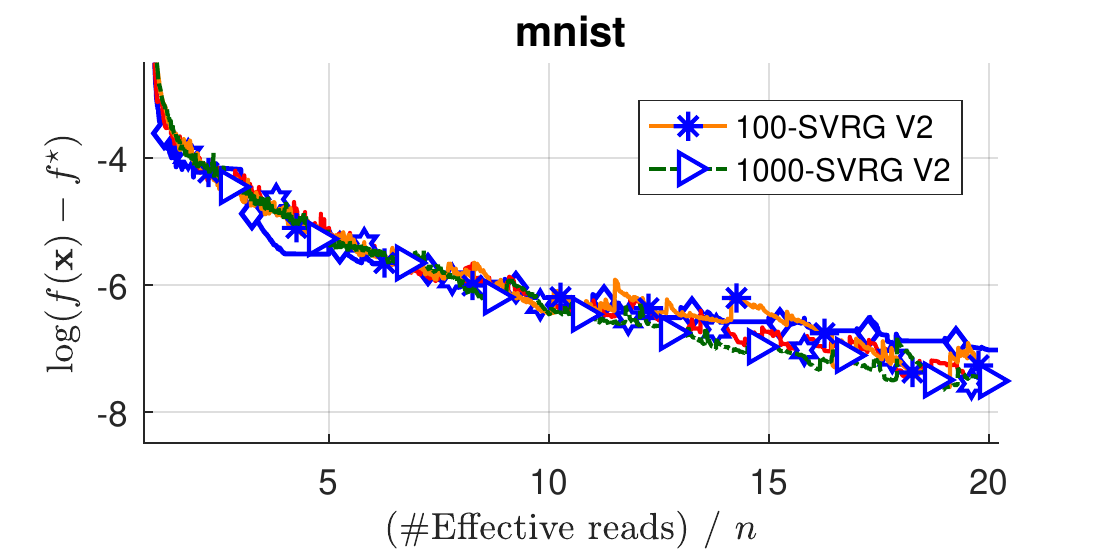}  
\hfill\null   
\caption{Evolution of residual loss on \emph{mnist} for  SVRG and \mVtwo{} for  $k=\{10,100,1000\}$.}
\label{fig:mnist_compare_var2}
\end{figurehere}
\begin{figurehere}
\centering
\hfill
  \includegraphics[width=0.325\linewidth]{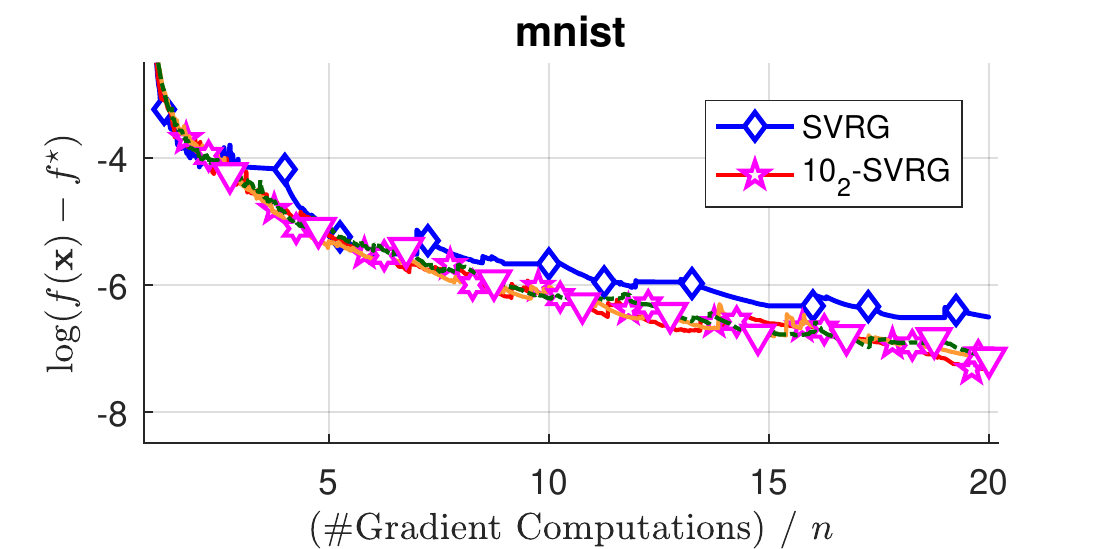}
\hfill
  \includegraphics[width=0.325\linewidth]{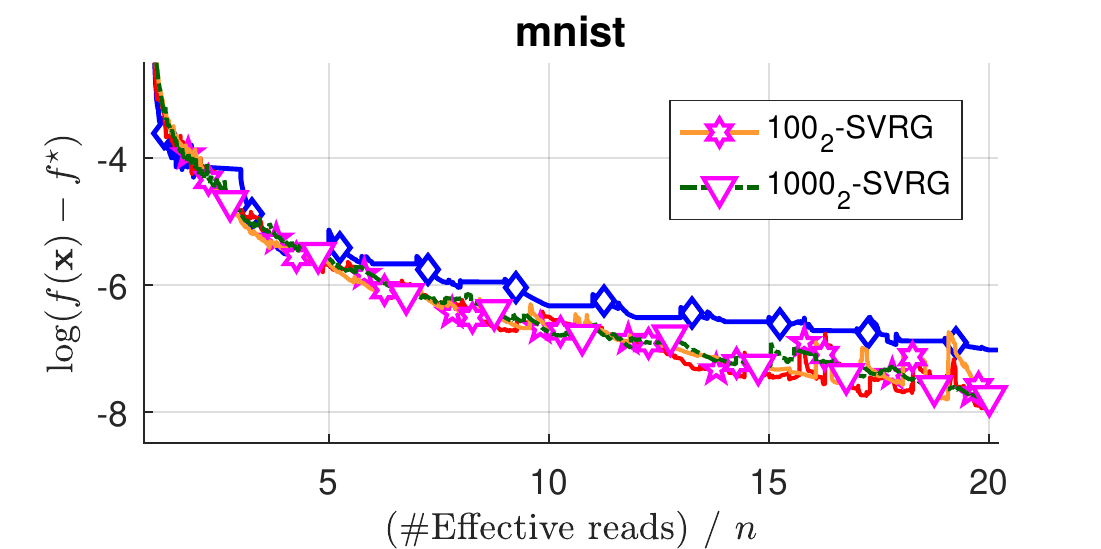}  
\hfill\null  
\caption{Evolution of residual loss on \emph{mnist} for  SVRG and \mpract{} for  $k=\{10,100,1000\}$.}
\label{fig:mnist_compare_pract}
\end{figurehere}

\clearpage
\subsection{Dataset: \emph{covtype (train)}}
\begin{multicols}{2}
\begin{figurehere}
\centering
  \includegraphics[width=0.49\linewidth]{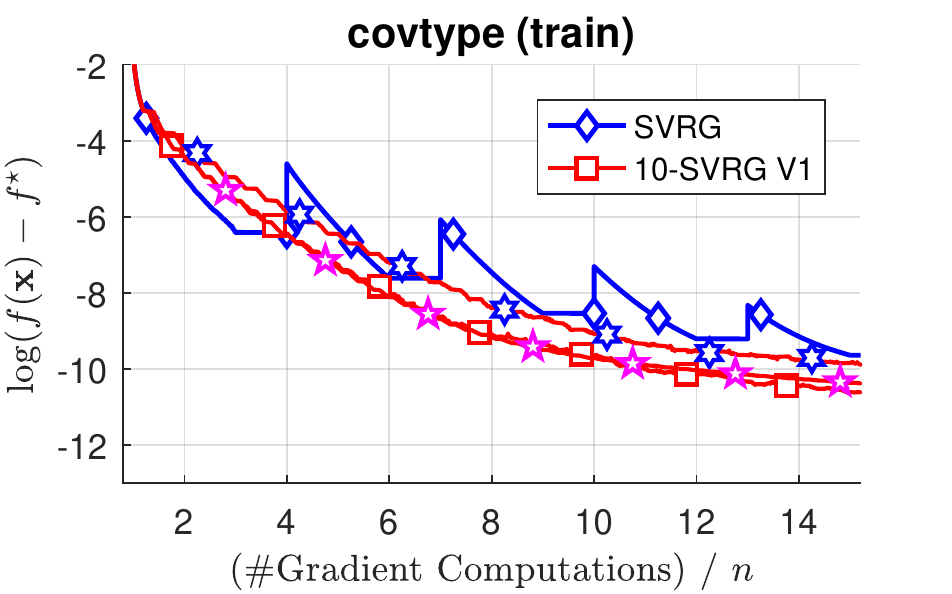}
\hfill
  \includegraphics[width=0.49\linewidth]{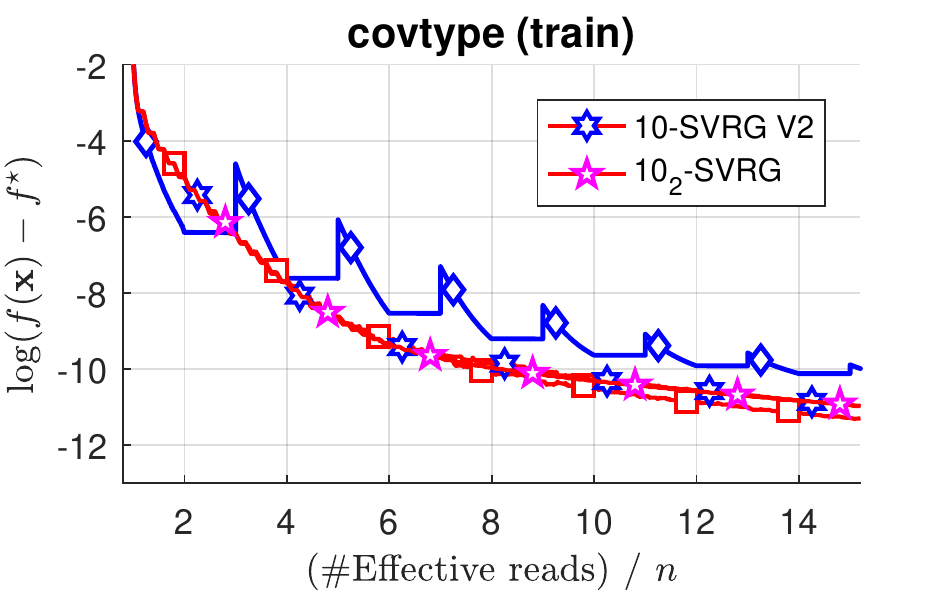}
  \vspace*{-1.5em}  
\caption{Evolution of residual loss on \emph{covtype (train)} for  SVRG, $10$\mHone, $10$\mHtwo{} and $10_2$\mHpract.}
\label{fig:cov2_compare_k10}
\end{figurehere}
\begin{figurehere}
\centering
  \includegraphics[width=0.49\linewidth]{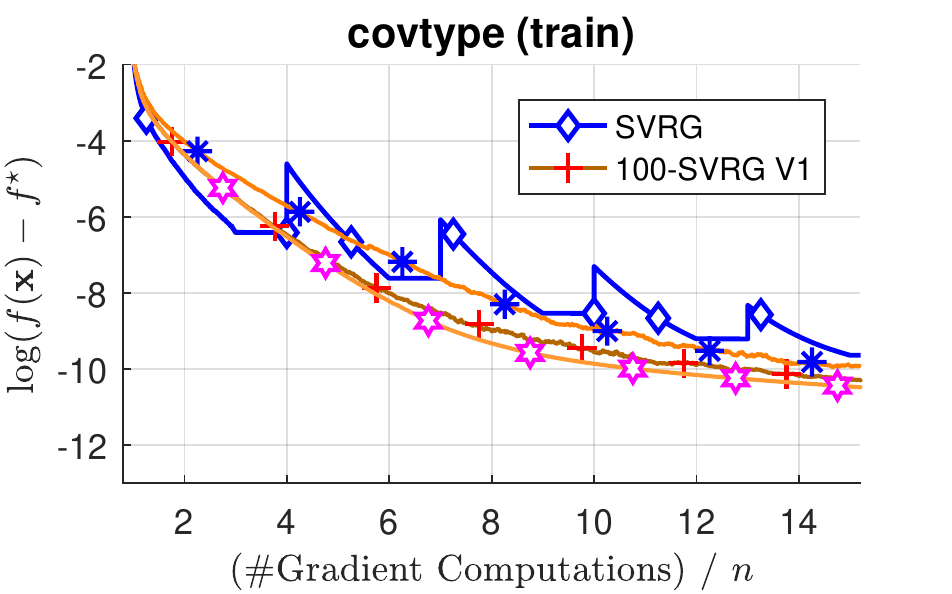}
\hfill
  \includegraphics[width=0.49\linewidth]{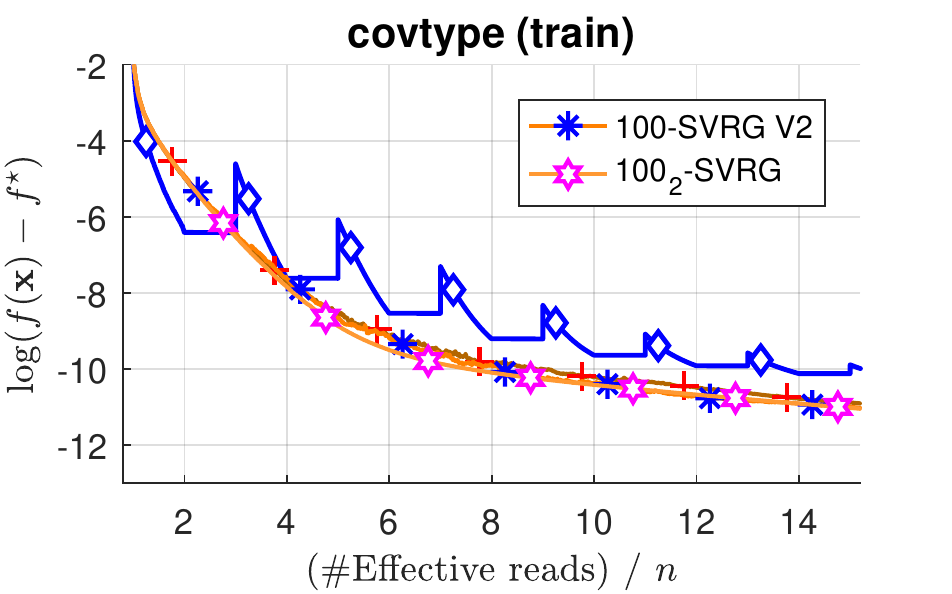}
   \vspace*{-1.5em}  
\caption{Evolution of residual loss on \emph{covtype (train)} for  SVRG, $100$\mHone, $100$\mHtwo{} and $100_2$\mHpract.}
\label{fig:cov2_compare_k100}
\end{figurehere}
\end{multicols}
\begin{multicols}{2}
\begin{figurehere}
\centering
  \includegraphics[width=0.49\linewidth]{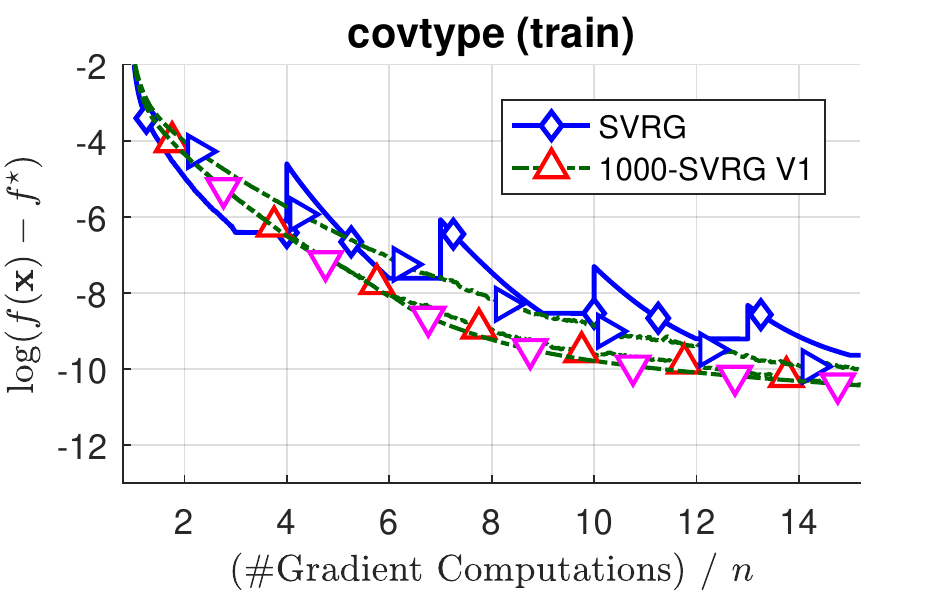}
\hfill
  \includegraphics[width=0.49\linewidth]{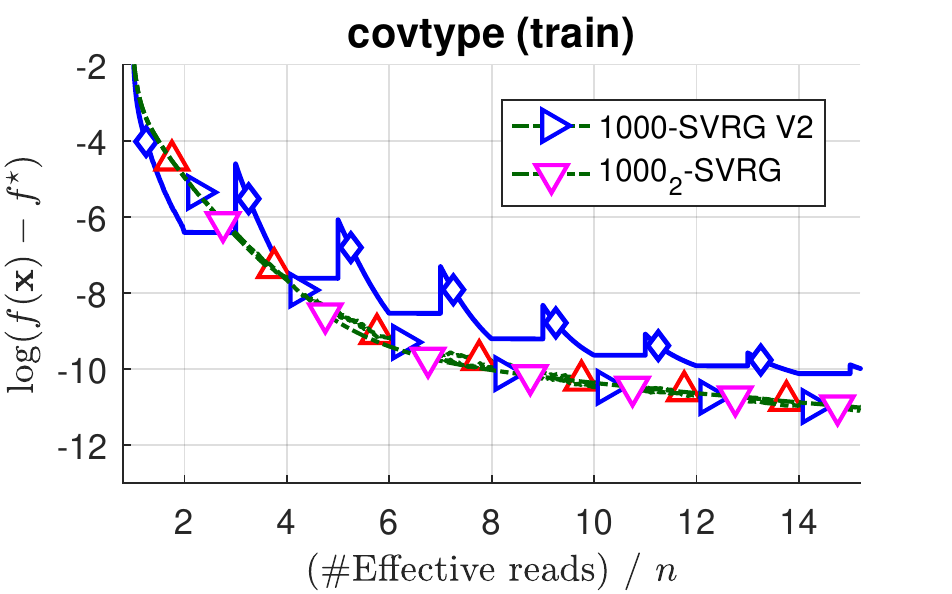}
   \vspace*{-1.5em}  
\caption{Evolution of residual loss on \emph{covtype (train)} for  SVRG, $1000$\mHone, $1000$\mHtwo{} and $1000_2$\mHpract.}
\label{fig:cov2_compare_k1000}
\end{figurehere}

\vspace{8cm}\null
\end{multicols}

\begin{figurehere}
\centering
\hfill
  \includegraphics[width=0.325\linewidth]{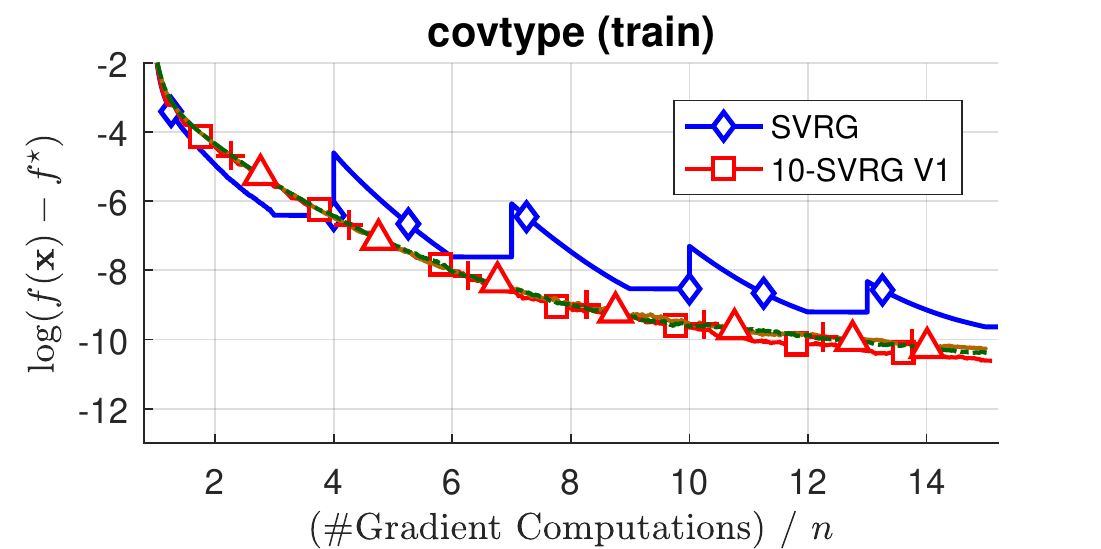}
\hfill
  \includegraphics[width=0.325\linewidth]{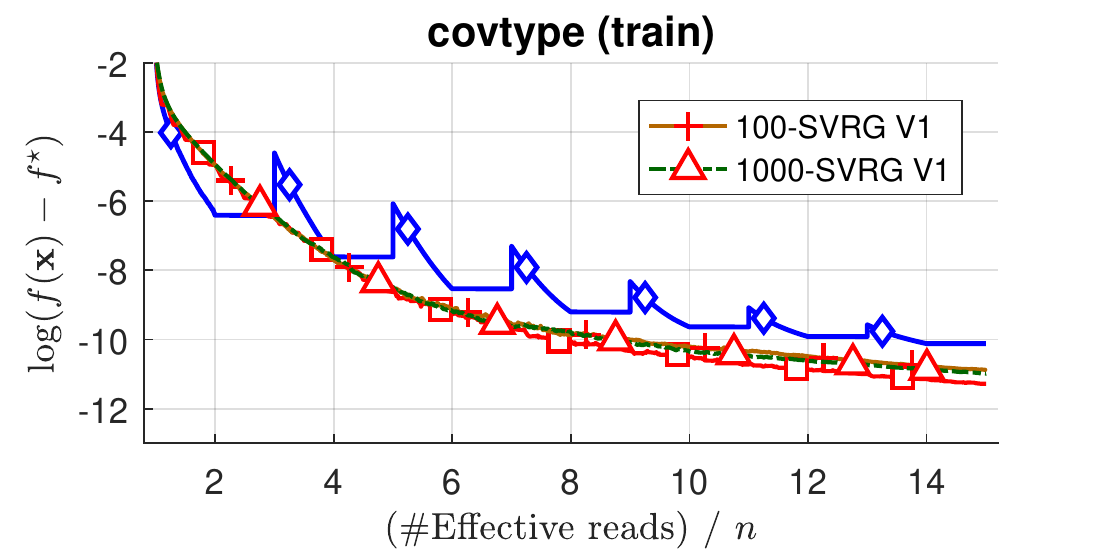}  
\hfill\null
\caption{Evolution of residual loss on \emph{covtype (train)} for SVRG and \mVone{} for  $k=\{10,100,1000\}$.}
\label{fig:cov2_compare_var1}
\end{figurehere}
\begin{figurehere}
\centering
\hfill
  \includegraphics[width=0.325\linewidth]{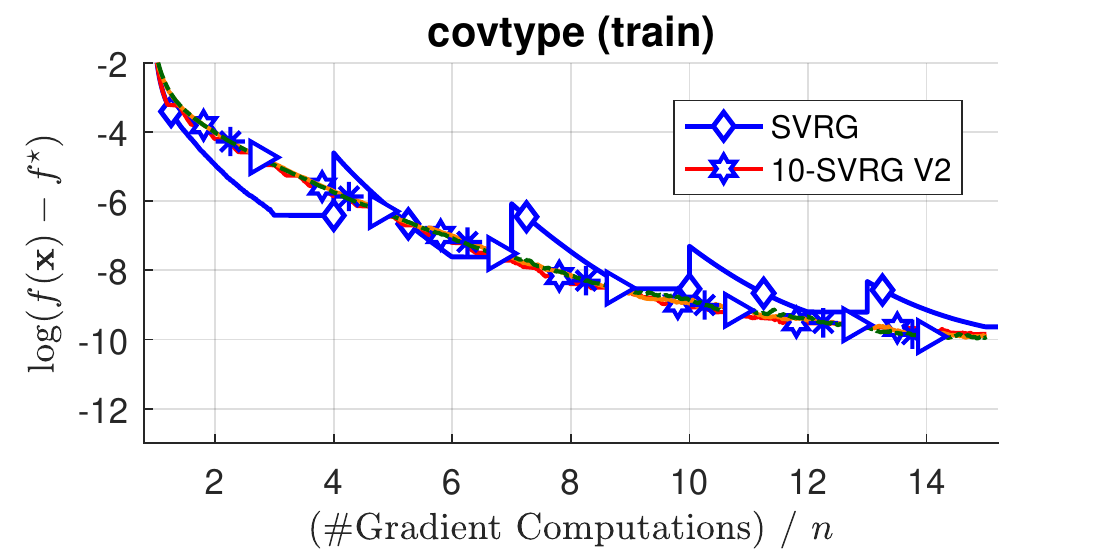}
\hfill
  \includegraphics[width=0.325\linewidth]{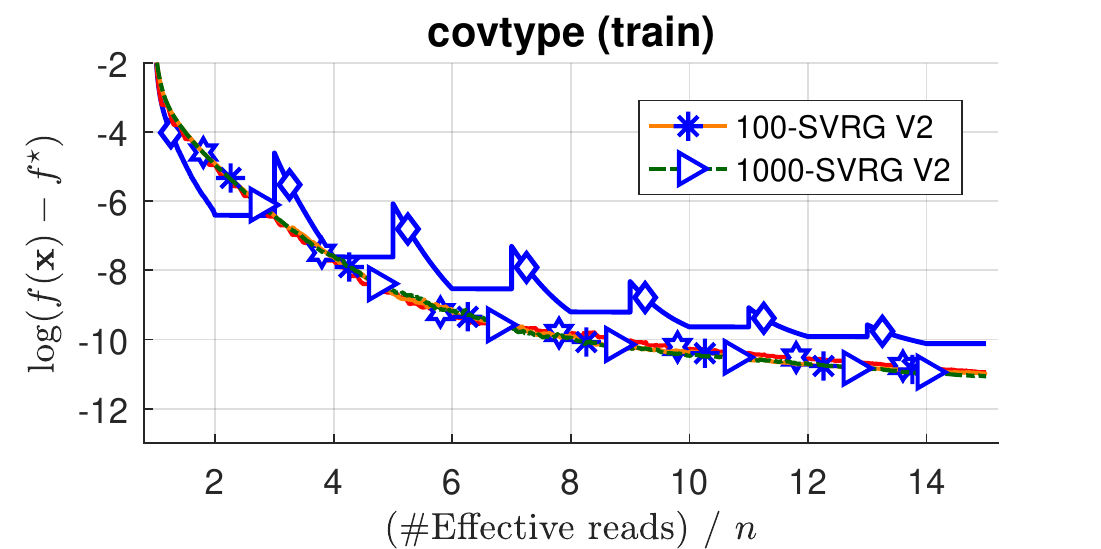}  
\hfill\null  
\caption{Evolution of residual loss on \emph{covtype (train)} for SVRG and \mVtwo{} for  $k=\{10,100,1000\}$.}
\label{fig:cov2_compare_var2}
\end{figurehere}
\begin{figurehere}
\centering
\hfill
  \includegraphics[width=0.325\linewidth]{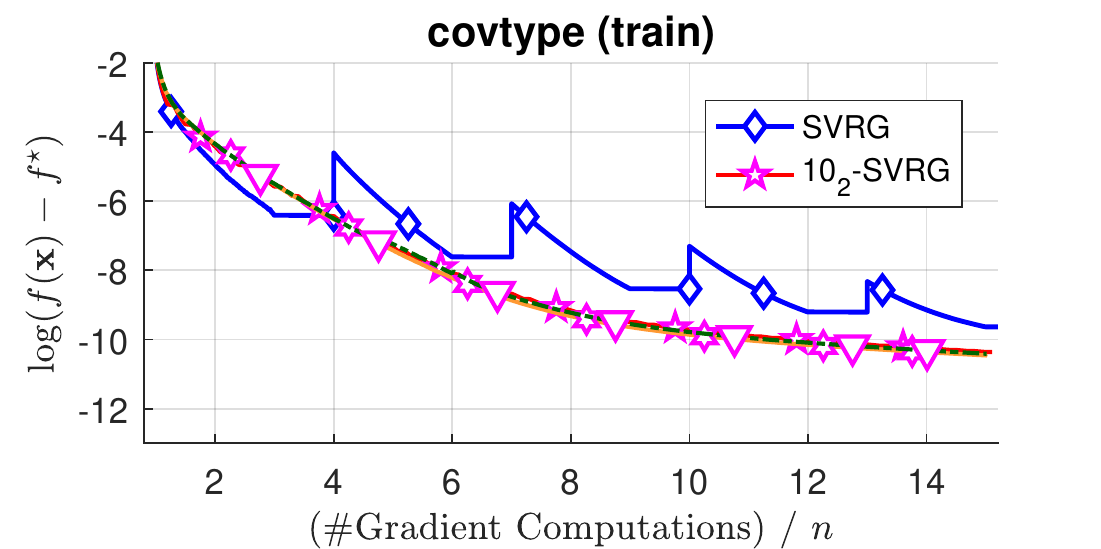}
\hfill
  \includegraphics[width=0.325\linewidth]{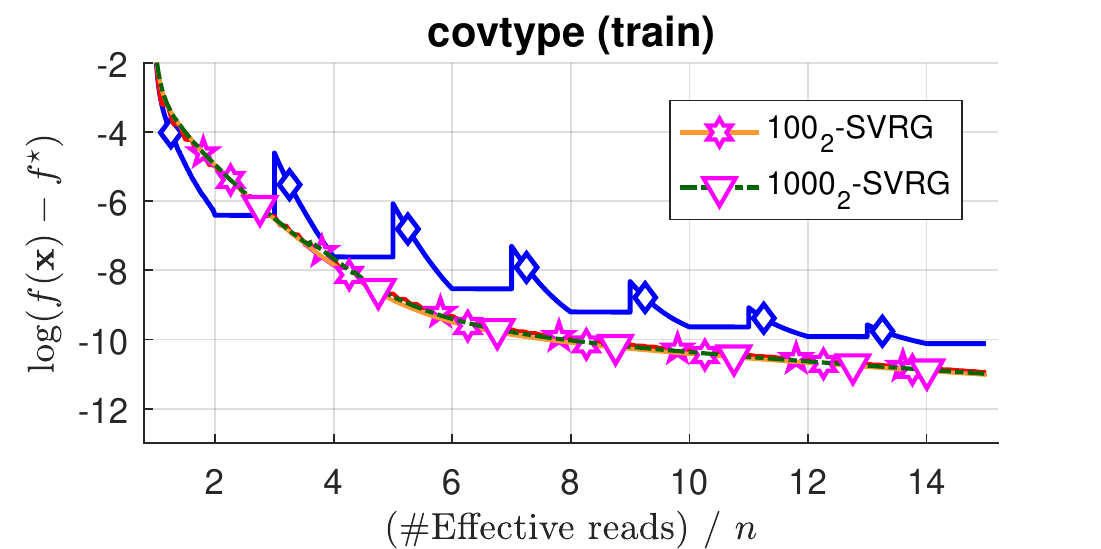}  
\hfill\null    
\caption{Evolution of residual loss on \emph{covtype (train)} for SVRG and \mpract{} for  $k=\{10,100,1000\}$.}
\label{fig:cov2_compare_pract}
\end{figurehere}